\documentclass[12pt,CJK]{article}
\usepackage{geometry}
\usepackage{amsfonts,amsmath,amssymb,amsthm,mathrsfs}
\usepackage{latexsym}
\usepackage{graphicx}
\usepackage{indentfirst}
\usepackage{color}
\usepackage{float} %设置图片浮动位置的宏包
\usepackage{subfigure} %插入多图时用子图显示的宏包
\usepackage{fancyhdr}
\usepackage[colorlinks,linktocpage,linkcolor=blue,citecolor=red]{hyperref}
\usepackage{titlesec}
\usepackage{dsfont}

\usepackage{authblk}

\usepackage{pgf}
\usepackage{tikz}
\usepackage[utf8]{inputenc}
\usetikzlibrary{arrows,automata}
\usetikzlibrary{positioning}

\tikzset{
    state/.style={
           rectangle,
           rounded corners,
           draw=black, very thick,
           minimum height=2em,
           inner sep=2pt,
           text centered,
           },
}

\tikzset{global scale/.style={
    scale=#1,
    every node/.append style={scale=#1}
  }
}

\definecolor{darkgreen}{rgb}{0.00,0.50,0.25}
\definecolor{purple}{rgb}{0.54,0.17,0.89}
\definecolor{purple1}{rgb}{1.00,0.00,1.00}

\titleformat*{\subsection}{\bfseries}

\theoremstyle{plain}                       % default
\newtheorem{lemma}{Lemma}[section]
\newtheorem{theorem}[lemma]{Theorem}
\newtheorem{corollary}[lemma]{Corollary}
\newtheorem{remark}[lemma]{Remark}
\newtheorem{definition}[lemma]{Definition}
\newtheorem{proposition}[lemma]{Proposition}

\theoremstyle{remark}

\numberwithin{equation}{section}

\def\Xint#1{\mathchoice
  {\XXint\displaystyle\textstyle{#1}}%
  {\XXint\textstyle\scriptstyle{#1}}%
  {\XXint\scriptstyle\scriptscriptstyle{#1}}%
  {\XXint\scriptscriptstyle\scriptscriptstyle{#1}}%
  \!\int}
\def\XXint#1#2#3{{\setbox0=\hbox{$#1{#2#3}{\int}$}
  \vcenter{\hbox{$#2#3$}}\kern-.5\wd0}}

\def\dashint{\Xint-}
\DeclareMathOperator*{\esssup}{ess\,sup}
\begin{document}
\allowdisplaybreaks
%\pagestyle{myheadings}
%\markboth{$~$ \hfill {\rm Q. Xu,} \hfill $~$}
%{$~$ \hfill {\rm  } \hfill$~$}
%\author{Li Wang
%%\thanks{Email: lwang10@lzu.edu.cn.}
%\quad Qiang Xu
%\thanks{Corresponding author}
%\thanks{Email: xuqiang09@lzu.edu.cn.}
%\quad Peihao Zhao
%%\thanks{Email: phzhao@lzu.edu.cn.}
%\\
%School of Mathematics and Statistics, Lanzhou University, \\
%Gansu, 730000, PR China.
%\vspace{0.5cm}
%}

\author[a]{Li Wang\thanks{Email: li$\_$wang@lzu.edu.cn.}
}

\author[a]{Qiang Xu\thanks{Email: xuq@lzu.edu.cn.}
}

\affil[a]{School of Mathematics and Statistics, Lanzhou University, Lanzhou, 730000, China.}

%
%\author[b]{Zhifei Zhang\thanks{Email: zfzhang@math.pku.edu.cn.}
%}

%\affil[b]{School of Mathematic Sciences,
%Peking University, Beijing 100871, China.
%%\authorcr Email:wangli@math.pku.edu.cn %\authorcr is for a new line.
%}

%

%\author{Weiren Zhao
%\thanks{Email: xuqiang@math.pku.edu.cn.}
%\thanks{This work was supported by the National Natural Science Foundation of China (Grant No. 11471147).}

\title{\textbf{Annealed Calder\'on-Zygmund
estimates for  elliptic \\ operators with random coefficients\\ on $C^{1}$ domains}}

\maketitle
\begin{abstract}
Concerned with elliptic operators with stationary random coefficients
governed by linear or nonlinear mixing conditions and bounded (or unbounded) $C^1$ domains, this paper mainly studies
(weighted) annealed Calder\'on-Zygmund estimates, some of which are new even in a
periodic setting. Stronger than some classical results derived by a perturbation argument in the deterministic case, our results own a scaling-invariant property, which additionally requires the  non-perturbation method (based upon a quantitative homogenization theory and a set of functional analysis techniques) recently developed by M. Joisen and F. Otto \cite{Josien-Otto22}.
To handle boundary estimates in certain UMD (unconditional martingale differences) spaces, we hand them over to  Shen's real arguments \cite{Shen05,Shen23} instead of using Mikhlin's theorem. As a by-product, we also established ``resolvent estimates''.
The potentially attractive part is to show how the two powerful kernel-free methods work together to make the results clean and robust.

\end{abstract}

\tableofcontents

\section{Introduction}

\subsection{Motivation and main results}
\noindent
Over the past few decades,
since Calder\'on-Zygmund estimates have many applications in applied mathematics, there have been an enormous amount of research activities (see e.g. \cite{Armstrong-Daniel16,Auscher07,
Auscher-Qafsaoui02,Auscher-Coulhon-Duong-Hofmann04,
Avelin-Kuusi-Mingione18,Avellaneda-Lin91,
Breit-Cianchi-Diening-Kuusi-Schwarzacher18,
Byun-Wang04,Caffarelli-Peral98,
Dong-Phan23,Dong-Kim11,Geng12,Geng-Shen24,
Gloria-Neukamm-Otto20,Jerison-Kenig95,Josien23,Krylov99,Mingione07,
Shen05,Shen-Zhuge18,Wang-Xu22,Wang03,Yao-Zhou17,Zhang-Zhou14}) concerning various versions of the classical Calder\'on-Zygmund estimates established in \cite{Calderon-Zygmund52} under different assumptions.
In the present work,  we are interested in
the so-called \emph{annealed Calder\'on-Zygmund estimates} for
linear elliptic equations with random coefficients, whose notion
was formally introduced in \cite{Duerinckx-Otto20,Josien-Otto22}, arising from the quantitative stochastic homogenization theory. Also,
the present contribution is to stress the related ``resolvent estimates''.

First, quantifying homogenization errors (in the sense of oscillating or   fluctuation) plays a crucial role in the real implementation of the numerical algorithm and therefore is one of the core research fields in  the quantitative homogenization theory (see e.g.  \cite{Armstrong-Kuusi-Mourrat19, Clozeau-Josien-Otto-Xu,Josien-Otto22,Wang-Xu22}). If one ignores the two-scale asymptotic analysis, from the perspective of harmonic analysis, the error estimates can be transformed into the
study of the boundedness of singular integrals (non-convolution type) in appropriate Banach spaces, which is known quite early in the periodic framework (see e.g. \cite{Avellaneda-Lin91}).
Second, the fluctuation estimate of (higher-order) correctors received
a great concern in the quantitative stochastic homogenization theory, which also heavily relies on the
weighted annealed Calder\'on-Zygmund estimates (see e.g. \cite{Clozeau-Josien-Otto-Xu,Duerinckx23,Duerinckx-Otto20}).
Last but not least, when addressing evolution problems arising from heterogeneous materials,
some $L^p$-resolvent estimates seem to be a promising initial approach, which are closely intertwined with the Calder\'on-Zygmund theory.
Furthermore, it leads to a nontrivial extension of the semigroup theory
and presents us with some new challenges due to the stochastic nature of the coefficients.
Therefore, these aforementioned three points serve as the primary motivation behind the present work (see also the previous work \cite[Subsection 1.1]{Wang-Xu22}).

%Different from our previous work \cite[Theorem 1.1]{Wang-Xu22}, under better boundary regularity conditions, we manage to establish more cleaner results (see \eqref{A} and $\eqref{B}$), that means there is no any spacial average at small (or large) scales involved in the present work.

Precisely, we are interested in the elliptic operator in the divergence form\footnote{For the ease of the statement, we adopt scalar notation and language throughout the paper.} $\mathcal{L}:=-\nabla\cdot a\nabla$,
where $a$ satisfies $\lambda$-uniformly elliptic conditions, i.e., for some
$\lambda\in(0,1)$ there holds
\begin{equation}\label{a:1}
\lambda|\xi|^2 \leq  \xi\cdot a(x)\xi \leq \lambda^{-1}|\xi|^2
\qquad
\text{for~all~}\xi\in\mathbb{R}^{d}\text{~and~}x\in\mathbb{R}^d.
\end{equation}

One can introduce the configuration space that is the set
of coefficient fields satisfying $\eqref{a:1}$, equipped with a
probability measure on it, which is also called as an ensemble
and denote the expectation by $\langle\cdot\rangle$. This ensemble
is assumed to be \emph{stationary}, i.e., for all shift vectors $z\in\mathbb{R}^d$, $a(\cdot+z)$ and $a(\cdot)$ have the same law under
$\langle\cdot\rangle$. Also, we
assume that the ensemble $\langle\cdot\rangle$ satisfies the \emph{spectral gap
condition} (known as a nonlinear mixing condition): For any random variable $F$, which is a functional of $a$, there exists $\lambda_1\geq 1$ such that
\begin{equation}\label{a:2}
 \big\langle (F-\langle F\rangle)^2 \big\rangle \leq \lambda_1
 \Big\langle\int_{\mathbb{R}^d}\Big(\dashint_{B_1(x)}\big|\frac{\partial F}{\partial a}\big|\Big)^2 dx\Big\rangle,
\end{equation}
where
the definition of the functional derivative of $F$ with respect to $a$ can be found in $\eqref{functional}$ (or see \cite[pp.15]{Josien-Otto22}).
Moreover, for obtaining pointwise estimates, a \emph{local smoothness} assumption for the admissible coefficients
has been postulated, i.e., there exists $\sigma_0\in(0,1)$ such that for any $0<\sigma<\sigma_0$
and $1\leq \gamma<\infty$ there holds
\begin{equation}\label{a:3}
 \big\langle  \|a\|_{C^{0,\sigma}(B_1)}^\gamma
 \big\rangle
 \leq \lambda_2(d,\lambda,\sigma,\gamma).
\end{equation}

Throughout the paper, let
$B_r(x)$ denote the open ball centered at $x$ of radius $r$, which
is also abbreviated as $B$. When we use this shorthand notation,
let $x_B$ and $r_B$ represent the center and radius of $B$, respectively.
In addition, we mention that a class of Gaussian ensembles stated in \cite[Section 3.2]{Josien-Otto22} satisfy the above assumptions $\eqref{a:1}, \eqref{a:2}$ and $\eqref{a:3}$.

\medskip

The definition of domains relevant to the main results of the paper is introduced.
\begin{definition}[Lipschitz, $C^1$ domains]
Let $\Omega\subset\mathbb{R}^d$ with $d\geq 2$ be a (bounded) domain. $\Omega$
is called a Lipschitz domain, with Lipschitz constant less than or equal to $M_0$
if we can cover a neighborhood of $\partial\Omega$ by
(finite many) balls $B$ (with $x_B\in\partial\Omega$) so
that, in an appropriate orthonormal coordinate system, $B\cap\Omega=
\{x=(x',x_d): x'\in\mathbb{R}^{d-1}\text{~and~}
x_d>\varphi(x')\}\cap B$, where $\varphi:\mathbb{R}^{d-1}\to\mathbb{R}$ is
a Lipschitz function, i.e., $|\varphi(x')-\varphi(y')|\leq M_0|x'-y'|$ for
any $x',y'\in\mathbb{R}^{d-1}$, with $\varphi(0)=0$. The domain is called a
$C^1$ domain if the function $\varphi$ can be additionally chosen to be of
class $C^1$.
\end{definition}

We now consider the following Dirichlet boundary value problems:
\begin{equation}\label{pde:0}
(\text{D})~\left\{\begin{aligned}
-\nabla\cdot a \nabla u
&= \nabla\cdot f
&\quad&\text{in}~~\Omega;\\
u &= 0
&\quad&\text{on}~~\partial\Omega,
\end{aligned}\right.
\end{equation}
where $\Omega$ could be an unbounded domain such as the
half-space $\mathbb{R}^d_{+}:=\{x\in\mathbb{R}^d:x_d>0\}$
and exterior domains. Then,
the main results are stated as follows.

\begin{theorem}[annealed Calder\'on-Zygmund estimates]\label{thm:1}
Let $\Omega\subset\mathbb{R}^d$ with $d\geq 2$ be a bounded $C^{1}$ domain.
Suppose that $\langle\cdot\rangle$ is stationary and satisfies the spectral gap condition $\eqref{a:2}$,
and the (admissible) coefficient additionally satisfies the local regularity condition $\eqref{a:3}$.
Let $\nabla u$ and square-integrable $f$ be
associated by the elliptic systems $\eqref{pde:0}$.
Then, for any $1<p,q<\infty$ with $\bar{p}>p$, we have the annealed Calder\'on-Zygmund estimate
\begin{equation}\label{A}
\Big(\int_{\Omega}\langle|\nabla u|^p\rangle^{\frac{q}{p}} \Big)^{\frac{1}{q}}
\lesssim_{\lambda,\lambda_1,\lambda_2,d,M_0,p,\bar{p},q} \Big(\int_{\Omega}\langle|f|^{\bar{p}}\rangle^{\frac{q}{\bar{p}}} \Big)^{\frac{1}{q}}.
\end{equation}
In particular, if $\Omega$ is an unbounded $C^1$ domain, then
the estimate $\eqref{A}$ still holds true.
\end{theorem}

To present the weighted counterparts of the annealed Calder\'on-Zygmund estimates, we start from
the definition of the Muckenhoupt's weight class.
\begin{definition}[$A_q$-weights]\label{def:weight}
Let $1\leq q<\infty$. A non-negative, locally integrable function $\omega$
is said to be an $A_q$ weight, if there exists a positive constant
$A$ such that, for every ball $B\subset\mathbb{R}^d$,
\begin{equation}\label{f:w-1}
 \bigg(\dashint_{B}\omega\bigg)\bigg(\dashint_{B}\omega^{-1/(q-1)}\bigg)^{q-1}
  \leq A,
\end{equation}
if $q>1$, where
the average integral symbol $\dashint_{B}$ is defined by $\frac{1}{|B|}\int_B$
with $|B|$ being the Lebesgue measure of $B$,  or
\begin{equation}\label{f:w-2}
\bigg(\dashint_{B}\omega\bigg)\esssup_{x\in B}\frac{1}{\omega(x)}\leq A,
\end{equation}
if $q=1$. The infimum over all such constants $A$ is called
the $A_q$ Muckenhoupt characteristic constant of $\omega$, denoted by
$[\omega]_{A_q}$. Moreover, $A_q$ represents
the set of all $A_q$ weights.
\end{definition}

\begin{theorem}[weighted annealed Calder\'on-Zygmund estimates]\label{thm:1-w}
Let $\Omega\subset\mathbb{R}^d$ with $d\geq 2$ be a bounded $C^{1}$ domain.
Suppose that $\langle\cdot\rangle$ is stationary and satisfies the spectral gap condition $\eqref{a:2}$,
and the (admissible) coefficient additionally satisfies the local regularity condition $\eqref{a:3}$.
Let $\nabla u$ and square-integrable $f$ be
associated by the elliptic systems $\eqref{pde:0}$.
Then, for any $\omega\in A_q$,
there holds the weighted
annealed Calder\'on-Zygmund estimate
\begin{equation}\label{B}
\Big(\int_{\Omega}\big\langle|\nabla u|^p
\big\rangle^{\frac{q}{p}}\omega\Big)^{\frac{1}{q}}
\lesssim_{\lambda,\lambda_1,\lambda_2,d,M_0,p,\bar{p},q,[\omega]_{A_q}} \Big(\int_{\Omega}\big\langle|f|^{\bar{p}}
\big\rangle^{\frac{q}{\bar{p}}}\omega\Big)^{\frac{1}{q}}
\end{equation}
for any $1<p,q<\infty$ with $\bar{p}>p$.
In particular, if $\Omega$ is an unbounded $C^1$ domain, then
the stated estimate $\eqref{B}$ is still valid.
\end{theorem}

In fact, to handle Theorems $\ref{thm:1}$ and $\ref{thm:1-w}$, we appeal to the massive elliptic operator $\mathcal{L}+1/T$ with $T>0$. Therefore, by additionally
assuming the symmetry condition
\begin{equation}\label{a:4}
a=a^* \quad\text{i.e.,}\quad a_{ij}=a_{ji},
\end{equation}
it is reasonable to consider the resolvent problem:
\begin{equation}\label{pde:0-r}
\left\{\begin{aligned}
\mu \mathbf{u} -\nabla\cdot a \nabla \mathbf{u}
&= \mathbf{g} + \nabla\cdot \mathbf{f}
&\quad&\text{in}~~\Omega;\\
u &= 0
&\quad&\text{on}~~\partial\Omega,
\end{aligned}\right.
\end{equation}
where $\mu\in\Sigma_{\theta}\subset\mathbb{C}$ is an parameter, and
\begin{equation}\label{a:5}
\Sigma_{\theta}:=\{z\in\mathbb{C}:z\not=0\text{~and~}|\text{arg}(z)|
<\pi-\theta\}
\quad \text{with}\quad \theta\in(0,\pi/2),
\end{equation}
and $\mathbf{u}$ and $\mathbf{g}$ are complex-valued function.
The following is the main results.

\begin{theorem}[``resolvent estimates'']\label{thm:resolvent}
Let $\Omega\subset\mathbb{R}^d$ with $d\geq 2$ be a bounded $C^{1}$ domain.
Let $1<p,q<\infty$ with $\bar{p}>p$ and $\mu\in\Sigma_\theta$.
Suppose that $\langle\cdot\rangle$ is stationary and satisfies the spectral gap condition $\eqref{a:2}$,
and the (admissible) coefficient satisfies $\eqref{a:3}$ and $\eqref{a:4}$.
Let $\mathbf{u}$ and square-integrable $\mathbf{g}$ with $\mathbf{f}$ be associated by the elliptic systems $\eqref{pde:0-r}$.
Then, for any $\omega\in A_q$,
there holds the weighted resolvent estimate
\begin{equation}\label{C-1}
\begin{aligned}
(|\mu|+R_0^{-2})\Big(\int_{\Omega}\big\langle|\mathbf{u}|^p
\big\rangle^{\frac{q}{p}}\omega\Big)^{\frac{1}{q}}
&+ (|\mu|+R_0^{-2})^{\frac{1}{2}}
\Big(\int_{\Omega}\big\langle|\nabla\mathbf{u}|^p
\big\rangle^{\frac{q}{p}}\omega\Big)^{\frac{1}{q}}\\
&\lesssim \Big(\int_{\Omega}\big\langle|\mathbf{g}|^{\bar{p}}
\big\rangle^{\frac{q}{\bar{p}}}\omega\Big)^{\frac{1}{q}}
+(|\mu|+R_0^{-2})^{\frac{1}{2}}
\Big(\int_{\Omega}\big\langle|\mathbf{f}|^{\bar{p}}
\big\rangle^{\frac{q}{\bar{p}}}\omega\Big)^{\frac{1}{q}},
\end{aligned}
\end{equation}
where $R_0$ represents the diameter of $\Omega$.
If $\Omega$ is an unbounded $C^1$ domain, then we have
\begin{equation}\label{C-2}
|\mu|\Big(\int_{\Omega}\big\langle|\mathbf{u}|^p
\big\rangle^{\frac{q}{p}}\omega\Big)^{\frac{1}{q}}
+ |\mu|^{\frac{1}{2}}
\Big(\int_{\Omega}\big\langle|\nabla\mathbf{u}|^p
\big\rangle^{\frac{q}{p}}\omega\Big)^{\frac{1}{q}}
\lesssim \Big(\int_{\Omega}\big\langle|\mathbf{g}|^{\bar{p}}
\big\rangle^{\frac{q}{\bar{p}}}\omega\Big)^{\frac{1}{q}}
+|\mu|^{\frac{1}{2}}\Big(\int_{\Omega}\big\langle|\mathbf{f}|^{\bar{p}}
\big\rangle^{\frac{q}{\bar{p}}}\omega\Big)^{\frac{1}{q}}.
\end{equation}
In particular, if $\mu>0$ and the elliptic systems $\eqref{pde:0-r}$ are real valued, then the estimates $\eqref{C-1}$ and $\eqref{C-2}$ are still true without the symmetry condition $\eqref{a:4}$. The multiplicative constants of $\eqref{C-1}$ and $\eqref{C-2}$ additionally depend on
$\theta$ compared to their counterpart in $\eqref{B}$.
\end{theorem}

A few remarks on theorems are in order.

\begin{remark}
\emph{From the scale point of view in the regularity theory of PDEs, the perturbation method mainly acts on the small-scale range, and that is where the assumption $\eqref{a:3}$ entered. To get large-scale estimates,
a class of additional structure conditions (such as periodicity, stationarity) on the coefficients seems inevitable, and, in the stochastic setting, a quantified
ergodicity assumption is crucial according to the important contributions
(see e.g. \cite{Armstrong-Kuusi22,Armstrong-Kuusi-Mourrat19,
Armstrong-Smart16,Gloria-Otto11,Gloria-Neukamm-Otto20,Gloria-Neukamm-Otto21}).
Here we adopt assumption $\eqref{a:2}$ only for the convenience of writing, and the approach used in this paper does not directly depend on the assumption $\eqref{a:2}$. Therefore, our results stated in Theorems $\ref{thm:1}$, $\ref{thm:1-w}$ and $\ref{thm:resolvent}$ can be hopefully generalized to
the stationary random coefficients that have slowly decaying correlations addressed in \cite{Gloria-Neukamm-Otto21}, as well as,
to other cases presented in the literatures \cite{Armstrong-Kuusi22,Armstrong-Kuusi-Mourrat19}. }
\end{remark}

\begin{remark}
\emph{In view of rescaling,
homogenization simply corresponds to the case of $R_0\gg 1$
and the multiplicative constants in
$\eqref{A}$, $\eqref{B}$, $\eqref{C-1}$, and $\eqref{C-2}$ never depend on $R_0$, so we can omit the
microcosmic parameter $\varepsilon$ for the ease of the statement. Roughly speaking, there are three approaches to arrive at the annealed Calder\'on-Zygmund estimates
stated in Theorems $\ref{thm:1}$ and $\ref{thm:1-w}$:
\begin{itemize}
  \item [(i).] Based upon quenched estimates (large-scale regularities),
  use the modern kernel-free methods originally developed in \cite{Caffarelli-Peral98} and further contributed by \cite{Shen05,Wang03}; This set of arguments was carried out by
  Duerinckx and Otto \cite{Duerinckx-Otto20} and the recent work \cite{Wang-Xu22}.
  \item [(ii).] Leaning towards functional analysis techniques, a non-perturbation method recently developed by Josien and Otto  \cite{Josien-Otto22},
      which completely avoided using large-scale estimates as in (i).
  \item [(iii).] On account of harmonic analysis techniques,
  the kernel method
  is still valid in the random setting\footnote{The second author discussed with M. Josien and F. Otto about this approach with details when he cooperated with them in the same project \cite{Clozeau-Josien-Otto-Xu} in 2020. However, they did not seek for a publication.}, which was traced to
  the pioneering work of Avellaneda and Lin \cite{Avellaneda-Lin91}.
\end{itemize}}

\emph{In our previous work \cite[Theorem 1.1]{Wang-Xu22}, purely based upon Shen's real arguments \cite{Shen05,Shen23} according to the regime mentioned in (i), we had already established
\begin{equation}\label{f:1.1}
\bigg(\int_{\Omega}
\Big\langle\Big(\dashint_{B_{1}(x)\cap\Omega}|\nabla u|^2
\Big)^{\frac{p}{2}}
\Big\rangle^{\frac{q}{p}}dx\bigg)^{\frac{1}{q}}
\lesssim_{\lambda,d,M_0,p,\bar{p},q}
\bigg(
\int_{\Omega}
\Big\langle\Big(\dashint_{B_{1}(x)\cap\Omega}
|f|^{2}\Big)^{\frac{\bar{p}}{2}}\Big\rangle^{\frac{q}{\bar{p}}}dx
\bigg)^{\frac{1}{q}},
\end{equation}
but the estimates $\eqref{A}$ and $\eqref{B}$ can not simply derived from $\eqref{f:1.1}$ together with
the local estimates at small scales (see Proposition $\ref{P:6}$), since
any special size relationship between $p,q\in(1,\infty)$ is not desired.
Besides, it is crucial to emphasize that despite the possibility of using approach (i), employing the non-perturbation method developed in (ii) becomes inevitable.  Therefore, we mainly adopt the approach (ii) in the present work for simplicity.}
\end{remark}

\begin{remark}
\emph{The results presented in Theorem $\ref{thm:resolvent}$
can be viewed from two distinct perspectives.
The first approach involves extending the results of Theorems $\ref{thm:1}$ and $\ref{thm:1-w}$ to encompass the case of the elliptic operators with complex coefficients. Subsequently, one can treat the massive elliptic operator as a new homogeneous operator in the region that expands one dimension, thereby enabling us to derive the resolvent estimates based upon the previous result on the homogenous operator. This methodology was well  developed by Gr\"oger and Rehberg \cite{Groger-Rehberg89}, initially inspired by Agmon \cite{Agmon62}. The second approach is to directly seek the corresponding estimates that never relies on $\mu$ in the resolvent
of $\mathcal{L}$, which had already been provided by the real methods employed in this work. Therefore, this approach merely requires us to obtain the corresponding estimates in the energy framework, and we adopt it here. However, the estimate $\eqref{C-1}$ is a little bit more complicated than $\eqref{C-2}$ and the reader will find the concrete reason in the proof of Theorem $\ref{thm:resolvent}$. Although the stated resolvent estimates $\eqref{C-1}$ and $\eqref{C-2}$ cannot be used directly in the semigroup theory, the demonstration of $\eqref{C-1}$  actually reveals
that one may employ the ``total probability rule'' to study the related
evolution problems, and we leave it in a separated work.  }
\end{remark}

\begin{remark}
\emph{The the above outcomes in Theorems $\ref{thm:1}$, $\ref{thm:1-w}$, and
$\ref{thm:resolvent}$ can be extended to some Lipschitz domains with
small Lipschitz constants (i.e., $M_0\ll1$).
Concerned with general Lipschitz domains,
we guess that the range of stochastic integrability should not be affected by the boundary regularity, (i.e., keeping the stochastic integrability $p\in(1,\infty)$
in Theorems $\ref{thm:1}$, $\ref{thm:1-w}$, and
$\ref{thm:resolvent}$) but we failed to provide a proof by using the methods in this paper. Therefore, it is still unclear for us.}
\end{remark}

\subsection{Organization of the paper}

\noindent
In Section $\ref{sec:2}$, we start from the annealed Calder\'on-Zygmund estimates for the massive operator with constant coefficients (see Lemma $\ref{lemma:ap-2}$); As a preparation for suboptimal estimates, we established annealed Calder\'on-Zygmund estimates for Meyer's type index in Lemma $\ref{lemma:2.2}$. Both of them rely on
Shen's real arguments \cite{Shen05,Shen23}.

In Section $\ref{sec:3}$, we derive an suboptimal annealed Calder\'on-Zygmund estimate (see Theorem $\ref{thm:2}$), which actually consists of two parts: (1). We show the related suboptimal estimates on the mixed norm defined in $\eqref{mnorm}$ in Proposition $\ref{P:3}$; (2).
Some annealed estimates at small scales are derived in Proposition $\ref{P:6}$. Both of them essentially rely on a local regularity
while the nestedness properties of the mixed norm in Lemma $\ref{lemma:3.1}$ still plays a key role in the demonstrations.
The main contribution is manifested in the aspect
that we take into account weighted estimates, correspondingly.

In Section $\ref{sec:4}$, the non-perturbation approach is employed to
promote the suboptimal estimate derived in Section $\ref{sec:3}$
to the optimal one, inspired by Josien and Otto \cite{Josien-Otto22}.
We complete the whole arguments by three main ingredients which has been named as the different subsections. To handle
the discrepancy on the boundary, caused by the standard two-scale expansion, we introduce a cut-off function (at small scales) therein, and it is indeed  compatible with infra-red cut-off very well in the whole regime (see Lemma $\ref{lemma:7.0}$).

In Section $\ref{sec:6}$, we mainly provide the proof of Theorem $\ref{thm:resolvent}$. Due to the assumption $\eqref{a:4}$, its $L^2$-theory is more or less standard and stated in Theorem $\ref{thm:5.0}$. To show the estimate $\eqref{C-1}$, we establish
the quenched-type Calder\'on-Zygmund estimates in Theorem $\ref{thm:5.1}$,
which also provides us with a new stationary random field to make
the idea of ``total probability rule'' realizable. Then, we hand
the $L^p$-theory over to Shen's real argument (see Lemma $\ref{shen's lemma}$) again.

Section $\ref{sec:5}$ is the appendix of the paper,
which contains some properties of
$A_q$ class, Meyer's estimates, and Shen's real methods for the reader's convenience.
%The first one is a real
%interpolation argument that is stated in Lemmas $\ref{lemma:4.2}$ and
%$\ref{lemma:5.2}$; The complex interpolation stated in Lemma $\ref{lemma:5.1}$, as the second ingredient, provides the precondition for Lemma $\ref{lemma:5.2}$. Then, we carry out a bootstrap process
%to finish the demonstration of Theorem $\ref{P:4}$. In the end, we give the proof of Theorem $\ref{thm:1}$.

\subsection{Notations}

\begin{enumerate}

  \item Notation for estimates.
\begin{enumerate}
  \item $\lesssim$ and $\gtrsim$ stand for $\leq$ and $\geq$
  up to a multiplicative constant,
  which may depend on some given parameters in the paper,
  but never on $\varepsilon$ and $R_0$. The subscript form $\lesssim_{\lambda,\cdots,\lambda_n}$ means that the constant depends only on parameters $\lambda,\cdots,\lambda_n$.
\item We use superscripts like $\lesssim^{(1.5)}$ to indicate the formula or estimate referenced.
  We write $\sim$ when both $\lesssim$ and $\gtrsim$ hold.
  \item We use $\gg$ instead of $\gtrsim$ to indicate that the multiplicative constant is much larger than 1 (but still finite),
      and it's similarly for $\ll$.
\end{enumerate}
  \item Notation for derivatives.
  \begin{enumerate}
    \item Spatial derivatives:
  $\nabla v = (\partial_1 v, \cdots, \partial_d v)$ is the gradient of $v$, where
  $\partial_i v = \partial v /\partial x_i$ denotes the
  $i^{\text{th}}$ derivative of $v$.
  $\nabla^2 v$  denotes the Hessian matrix of $v$;
  $\nabla\cdot v=\sum_{i=1}^d \partial_i v_i$
  denotes the divergence of $v$, where
  $v = (v_1,\cdots,v_d)$ is a vector-valued function.
   \item Functional (or vertical) derivative:  the random tensor field $\frac{\partial F}{\partial a}$ (depending on $(a,x)$) is the functional derivative of $F$ with respect to $a$, defined by
       \begin{equation}\label{functional}
       \lim_{\varepsilon\to 0} \frac{F(a+\varepsilon\delta a)
       -F(a)}{\varepsilon}
       =\int_{\mathbb{R}^d} \frac{\partial F(a)}{\partial a_{ij}(x)}
       (\delta a)_{ij}(x)dx,
       \end{equation}
       where the Einstein's summation convention for repeated indices is
used throughout.
  \end{enumerate}

\item Geometric notation.
\begin{enumerate}
\item $d\geq 2$ is the dimension, and $R_0\geq 1$ represents
  the diameter of $\Omega$.
\item For a ball $B$, we set $B=B_{r_B}(x_B)$ and abusively write
  $\alpha B:=B_{\alpha r_B}(x_B)$.
\item For any  $x\in\partial\Omega$,
we introduce $D_r(x):= B_r(x)\cap\Omega$ and $\Delta_r(x):=
B_r(x)\cap\partial\Omega$.
We will omit the center point of
$B_r(x),D_r(x),\Delta_r(x)$ only when $x=0$.
For any $x\in\Omega$, we introduce $U_r(x):=B_r(x)\cap\Omega$.
\end{enumerate}

\item Notation for functions.
\begin{enumerate}
    \item The function $\mathds{1}_{E}$ is the indicator function of region $E$.
\item We denote $F(\cdot/\varepsilon)$ by $F^{\varepsilon}(\cdot)$ for simplicity, and $(f)_r = \dashint_{B_r} f  = \frac{1}{|B_r|}\int_{B_r} f(x) dx$.
%\item We denote the support of $f$ by $\text{supp}(f)$.
%\item The nontangential maximal function of $u$ is denoted by
%$(u)^*$ (see \cite[pp.207-209]{Shen18} for its definition).
  \end{enumerate}
\end{enumerate}

\section{Preliminaries}\label{sec:2}

\begin{lemma}\label{lemma:ap-2}
Let $\Omega\subset\mathbb{R}^d$ with $d\geq 2$
be a (bounded) $C^1$ domain and $T>0$.
Suppose that $\bar{a}$ is a constant coefficient satisfying a uniform ellipticity condition $\eqref{a:1}$.
Assume that $v\in H_0^1(\Omega;\mathbb{C})$ is associated with the given complex-valued function $(f,g)$ by the following equations:
\begin{equation}\label{pde:1}
\left\{\begin{aligned}
\frac{1}{T}v-\nabla\cdot \bar{a}\nabla v&= \nabla\cdot f + \frac{1}{T}g
&~&\text{in}~~\Omega;\\
v&= 0 &~&\text{on}~\partial\Omega.
\end{aligned}\right.
\end{equation}
Then, for any $1<p<\infty$, there holds the Calder\'on-Zygmund estimate
\begin{equation}\label{pri:A-1}
  \int_{\Omega}
  |(\nabla v,\frac{v}{\sqrt{T}})|^p
  \lesssim_{\lambda,d,M_0,p}
  \int_{\Omega}
  |(f,\frac{g}{\sqrt{T}})|^p.
\end{equation}
Moreover, for any $1<p,q<\infty$, we have
annealed Calder\'on-Zygmund estimate
\begin{equation}\label{pri:A-2}
  \int_{\Omega}
  \big\langle |(\nabla v,\frac{v}{\sqrt{T}})|^p\rangle^{\frac{q}{p}}
  \lesssim_{\lambda,d,M_0,p,q}
  \int_{\Omega}
  \langle |(f,\frac{g}{\sqrt{T}})|^p\big\rangle^{\frac{q}{p}},
\end{equation}
and weighted annealed Calder\'on-Zygmund estimates, i.e.,
for any $\omega\in A_{q}$, there holds
\begin{equation}\label{pri:A-3}
  \Big(\int_{\Omega}
  \big\langle |(\nabla v,\frac{v}{\sqrt{T}})|^p\rangle^{\frac{q}{p}}
  \omega\Big)^{\frac{1}{q}}
  \lesssim_{\lambda,d,M_0,p,q,[\omega]_{A_q}}
  \Big(\int_{\Omega}
  \langle |(f,\frac{g}{\sqrt{T}})|^p\big\rangle^{\frac{q}{p}}
  \omega\Big)^{\frac{1}{q}}.
\end{equation}
In particular,  if $\Omega$ is an unbounded $C^1$ domain, then
the estimates $\eqref{pri:A-1}$, $\eqref{pri:A-2}$, and
$\eqref{pri:A-3}$ are still true.
\end{lemma}

\begin{proof}
Before providing specific demonstration, it is necessary to address two key points. (1). Our proof is based upon Shen's real argument (see Lemma $\ref{shen's lemma2}$), and how to use it has been systematically stated
in the previous work \cite[Subsection 3.1]{Wang-Xu22} with details. So,
to shorten the proof, we omit some unnecessary steps, such as, reducing the desired estimates to some local estimates, a covering argument, and
duality arguments, in the present statement.
(2). The fundamental idea of the estimates $\eqref{pri:A-1}$ and $\eqref{pri:A-2}$ remains consistent regardless of whether the domain is bounded or unbounded. Therefore, we solely present the proof of
$\eqref{pri:A-3}$ for the distinguishing part (see Part IV).

\medskip
\noindent
\textbf{Part I. Arguments for $\eqref{pri:A-1}$.}
Let $\Omega$ be a bounded or unbounded domain. The proof will be finished by two steps.

\noindent
\textbf{Step 1.} Reduction and outline of the proof.
In view of Lemma $\ref{shen's lemma2}$
with $\omega =1$ therein, it suffices to consider that for any ball $B$ satisfying $x_B\cap\partial\Omega$ or $2B\subset\Omega$, and for any solution of
\begin{equation*}
\left\{\begin{aligned}
 \frac{1}{T}w - \nabla\cdot\bar{a}\nabla w &=
 \nabla\cdot fI_{B}+\frac{1}{T}gI_{B}
 &\quad& \text{in}\quad \Omega;\\
 w &= 0
 &\quad& \text{on}\quad \partial\Omega,
\end{aligned}\right.
\end{equation*}
with $z:=v-w$ in $B\cap\Omega$ satisfying
\begin{equation*}
\left\{\begin{aligned}
\nabla\cdot\bar{a}\nabla z &= \frac{1}{T}z &\quad&\text{in}\quad B\cap\Omega;\\
z &=0 &\quad&\text{on}\quad B\cap\partial\Omega,
\end{aligned}\right.
\end{equation*}
one can verify the following estimates: for any $2<q<\infty$,
\begin{subequations}
\begin{align}
& \Big(\dashint_{\frac{1}{4}B\cap\Omega}
 |(\nabla z,\frac{z}{\sqrt{T}})|^q\Big)^{\frac{1}{q}}
 \lesssim_{\lambda,d,M_0,q}  \Big(\dashint_{B\cap\Omega}
 |(\nabla z,\frac{z}{\sqrt{T}})|^2\Big)^{\frac{1}{2}};
 \label{pri:ap-3}\\
& \dashint_{\frac{1}{2}B\cap\Omega}
 |(\nabla w,\frac{w}{\sqrt{T}})|^2
 \lesssim_{\lambda,d,M_0} \dashint_{B\cap\Omega}|(f,\frac{g}{\sqrt{T}})|^2.
 \label{pri:ap-4}
\end{align}
\end{subequations}

Once the above two estimates are established, we can show that
the desired estimate $\eqref{pri:A-1}$ for any $2<p<q$, while the case $1<p<2$
follows from a simple duality argument, where we also
employ the energy estimate
\begin{equation}\label{f:ap-10}
\int_{\Omega} |(\nabla v,\frac{v}{\sqrt{T}})|^2
\lesssim_{\lambda,d}
\int_{\Omega} |(f,\frac{g}{\sqrt{T}})|^2.
\end{equation}
Therefore, the remainder of the proof is to show the estimates $\eqref{pri:ap-3}$
and $\eqref{pri:ap-4}$. In fact, we can derive the estimate $\eqref{pri:ap-4}$
immediately by $\eqref{f:ap-10}$.

\medskip
\noindent
\textbf{Step 2.} Show the reverse H\"older's inequality $\eqref{pri:ap-3}$.
To see this, it is fine to assume $r_B=1$ by a rescaling argument (and therefore we slip the related estimates into the cases $T\geq 1$ and $0<T<1$.).
We divided the proof into two cases: (1) $B\cap\partial\Omega\not=\emptyset$; (2) $2B\subset\Omega$. The idea is
to appeal to the corresponding estimates for homogeneous elliptic operator.
For the case (1), let $D:=B_{2r_B}(\bar{x})\cap\Omega$ with
$\bar{x}$ being such that $\text{dist}(x_B,\partial\Omega)=|\bar{x}-x_B|$. It follows from $L^q$ estimates (see for example \cite[Chapter 7]{Giaquinta-Martinazzi12}) that, in the case of $0<T<1$,
\begin{equation}\label{f:ap-1}
 \|\nabla z\|_{L^q(\frac{1}{2}D)}
 \lesssim_{\lambda,d,M_0,q}
 \|z\|_{L^{2}(D)}
 +\frac{1}{T}\|z\|_{L^{\underline{q}}(D)}
 \lesssim \|z\|_{L^{2}(D)}
 +\frac{1}{T}\|\nabla z\|_{L^{\underline{q}}(D)}
\end{equation}
for $1/q=1/\underline{q}-1/d$ (it is fine to assume $q\gg 1$
and $\underline{q}\geq 2$), where the second line above follows from Poincar\'e's inequality; Similarly, in the case of $T\geq 1$, we have
\begin{equation}\label{f:ap-1*}
 \|\nabla z\|_{L^q(\frac{1}{2}D)}
 \lesssim_{\lambda,d,M_0,q}
 \|\nabla z\|_{L^{2}(D)}
 +\frac{1}{T}\|z\|_{L^{\underline{q}}(D)}
 \lesssim \|\nabla z\|_{L^{2}(D)}
 +\|\nabla z\|_{L^{\underline{q}}(D)},
\end{equation}
where the second line above follows from Poincar\'e's inequality
and the condition $T\geq 1$.

For the case (2), we start from the case of $0<T<1$, and
\begin{equation}\label{f:ap-2}
\begin{aligned}
 \|\nabla z\|_{L^q(\frac{1}{2}B)}
 \lesssim_{\lambda,d,M_0,q}
 \|z\|_{L^{2}(\frac{3}{4}B)}
 +\frac{1}{T}\|z-c\|_{L^{\underline{q}}(\frac{3}{4}B)} + \frac{|c|}{T}
\lesssim \|z\|_{L^{2}(B)}
 +\frac{1}{T}\|\nabla z\|_{L^{\underline{q}}(B)}.
\end{aligned}
\end{equation}
where $c:=\dashint_{B}z$, and we also employ Poincar\'e's inequality
and Caccioppoli's inequality for
the second inequality above; Similarly, in the case of $T\geq 1$, it follows that
\begin{equation}\label{f:ap-2*}
\begin{aligned}
 \|\nabla z\|_{L^q(\frac{1}{2}B)}
 \lesssim_{\lambda,d,M_0,q}
 \|\nabla z\|_{L^{2}(B)}
 +\frac{1}{T}\|z-c\|_{L^{\underline{q}}(B)} + \frac{|c|}{T}
\lesssim \|(\nabla z,\frac{z}{\sqrt{T}})\|_{L^{2}(B)}
 +\|\nabla z\|_{L^{\underline{q}}(B)},
\end{aligned}
\end{equation}
where we use Poincar\'e's inequality
and $T\geq 1$ in the second inequality above;
Moreover, we have
\begin{equation}\label{f:ap-3}
 \|\frac{z}{\sqrt{T}}\|_{L^q(\frac{1}{2}B\cap\Omega)}
 \lesssim_{d} \|\frac{z-\tilde{c}}{\sqrt{T}}\|_{L^q(\frac{1}{2}B\cap\Omega)}
 +\frac{|\tilde{c}|}{\sqrt{T}}
 \lesssim \|\nabla z\|_{L^q(\frac{1}{2}B\cap\Omega)}+\|\frac{z}{\sqrt{T}}\|_{L^2(B\cap\Omega)},
\end{equation}
where $\tilde{c}:=\dashint_{B\cap\Omega}z$.
Combining the above estimates
$\eqref{f:ap-1*}$, $\eqref{f:ap-2*}$ and $\eqref{f:ap-3}$, we obtain
\begin{equation}\label{f:ap-4*}
  \|(\nabla z,\frac{z}{\sqrt{T}})\|_{L^q(\frac{1}{2}B\cap\Omega)}
  \lesssim_{\lambda,d,M_0,q}
  \|(\nabla z,\frac{z}{\sqrt{T}})\|_{L^{\underline{q}}(B\cap\Omega)}
\quad\text{with}~\underline{q}\geq 2,
\end{equation}
and this is the iteration formula for the case of $T\geq 1$. In terms of
the case of $0<T<1$, combining the estimates $\eqref{f:ap-1}$ and
$\eqref{f:ap-2}$ leads to
\begin{equation}\label{f:ap-4}
  \|\nabla z\|_{L^q(\frac{1}{2}B\cap\Omega)}
  \lesssim_{\lambda,d,M_0,q}
  \frac{1}{T}\|\nabla z\|_{L^{\underline{q}}(B\cap\Omega)}
  + \|z\|_{L^{2}(B\cap\Omega)}
\quad\text{with}~\underline{q}\geq 2.
\end{equation}

In view of $\eqref{f:ap-4*}$ and $\eqref{f:ap-4}$,
we can employ a bootstrap method for our purpose. To do so, we first introduce a critical number: $q_{k_0}:=\frac{2d}{d-2k_0}$
where $k_0\geq 0$ is an integer satisfying $k_0<(d/2)\leq k_0+1$.
Then, the demonstration can be split into two cases: (1) $2\leq q\leq q_{k_0}$; (2) $q>q_{k_0}$. Therefore, after at most $k_0+1$ iterations
of using $\eqref{f:ap-4*}$, we can show the desired estimate $\eqref{pri:ap-3}$ in the case of $T\geq 1$. For the case $0<T<1$, it follows from the estimate $\eqref{f:ap-4}$ and the bootstrap process that
\begin{equation}\label{f:ap-12}
\|\nabla z\|_{L^q(\frac{1}{4}B\cap\Omega)}
\lesssim_{\lambda,d,M_0,q} (\frac{1}{T^{k_0}}+1)\|z\|_{L^{2}(\frac{1}{2}B\cap\Omega)}
\lesssim_{\lambda,d,M_0,q} \|z\|_{L^{2}(B\cap\Omega)},
\end{equation}
where we employ Caccioppoli's inequality by $2k_0$ times. This together with Sobolev embedding theorem implies
\begin{equation}\label{f:ap-13}
\|z\|_{L^q(\frac{1}{4}B\cap\Omega)}
\lesssim \|z\|_{L^{2}(B\cap\Omega)}.
\end{equation}
Thus, collecting the estimates $\eqref{f:ap-12}$ and $\eqref{f:ap-13}$ we
can derive that
\begin{equation*}
\|(\nabla z,\frac{z}{\sqrt{T}})\|_{L^q(\frac{1}{4}B\cap\Omega)}
\lesssim \|\frac{z}{\sqrt{T}}\|_{L^{2}(B\cap\Omega)}
+ \|z\|_{L^{2}(B\cap\Omega)}
\lesssim \|\frac{z}{\sqrt{T}}\|_{L^{2}(B\cap\Omega)}
+ \|\nabla z\|_{L^{2}(B\cap\Omega)}
\end{equation*}
where we also employ the condition $0<T<1$ in the last inequality,
and this gives the desired estimate $\eqref{pri:ap-3}$ for the case $0<T<1$.
%
%Finally, we mention that we should keep shrinking
%the integral region in the left-hand of $\eqref{f:ap-4}$ to fit for the related estimates, but we don't worry it would shrink down to a point since it is merely finite iterations.

\medskip
\noindent
\textbf{Part II. Arguments for $\eqref{pri:A-2}$.} The proof will be finished by
three steps.

\noindent
\textbf{Step 1}. Reduction and outline of the proof. Let $\bar{q}>p$
be arbitrarily given, and $1<p<\infty$.
We continue to follow the decomposition above, and introduce the following
notation for the ease of statement.
\begin{equation*}
\begin{aligned}
& V:= \big\langle|(\nabla v,\frac{v}{\sqrt{T}})|^p\big\rangle^{\frac{1}{p}};
 &\quad& F:= \big\langle|(f,\frac{g}{\sqrt{T}})|^p\big\rangle^{\frac{1}{p}};\\
& Z_B:=\big\langle|(\nabla z,\frac{z}{\sqrt{T}})|^p\big\rangle^{\frac{1}{p}};
 &\quad& W_B:= \big\langle|(\nabla w,\frac{w}{\sqrt{T}})|^p\big\rangle^{\frac{1}{p}}.
\end{aligned}
\end{equation*}
Due to the relationship $v=w+z$ in $B\cap\Omega$, it is not hard to see that $V\leq W_B + Z_B$. Appealing to
Shen's real argument again (see \cite[Theorem 4.2.6]{Shen18} or
\cite[Theorem 3.3]{Shen05}), it suffices to establish the following estimates:
\begin{subequations}
\begin{align}
& \dashint_{\frac{1}{2}B\cap\Omega} W_B^{p}
\lesssim_{\lambda,d,M_0,p}
\dashint_{B\cap\Omega} F^{p};
\label{pri:ap-6} \\
& \Big(\dashint_{\frac{1}{4}B\cap\Omega} Z_B^{\bar{q}}\Big)^{\frac{p}{\bar{q}}}
\lesssim_{\lambda,d,M_0,p}
\dashint_{B\cap\Omega} (V^{p}+F^{p}).
\label{pri:ap-5}
\end{align}
\end{subequations}
Admitting the above two estimates for a while, for any
$p<q<\bar{q}$, we can derive that
\begin{equation*}
\Big(\dashint_{\Omega}V^q\Big)^{\frac{1}{q}}
\lesssim \Big(\dashint_{\Omega}V^p\Big)^{\frac{1}{p}}
+ \Big(\dashint_{\Omega}F^q\Big)^{\frac{1}{q}}
\lesssim^{\eqref{pri:A-1}}
\Big(\dashint_{\Omega}F^p\Big)^{\frac{1}{p}}
+\Big(\dashint_{\Omega}F^q\Big)^{\frac{1}{q}}
\lesssim \Big(\dashint_{\Omega}F^q\Big)^{\frac{1}{q}}.
\end{equation*}
Noting $\bar{q}>p$ is arbitrary, the case $1<q<p$ follows from
a duality argument.

\medskip
\noindent
\textbf{Step 2.} Show the estimate $\eqref{pri:ap-6}$. Based upon the
estimate $\eqref{pri:A-1}$, we have
\begin{equation*}
\dashint_{\frac{1}{2}B\cap\Omega} W_B^{p}
\lesssim_d \frac{1}{|B|}\Big\langle\int_{\Omega}
|(\nabla w,\frac{w}{\sqrt{T}})|^p\Big\rangle
\lesssim
\frac{1}{|B|}\Big\langle\int_{\Omega}
|I_{B}(f,\frac{g}{\sqrt{T}})|^p\Big\rangle
\lesssim \dashint_{B\cap\Omega} F^{p},
\end{equation*}
which gives the stated estimate $\eqref{pri:ap-6}$.

\medskip
\noindent
\textbf{Step 3.} Show the estimate $\eqref{pri:ap-5}$. On account of
the improved Meyer's inequality\footnote{It actually follows
from a convexity argument independent of PDEs, or
see \cite[pp.184-186]{Giaquinta-Martinazzi12}.} (see e.g.\cite[Lemma 5.1]{Wang-Xu22}), we
upgrade the estimate $\eqref{pri:ap-3}$ to
\begin{equation}\label{pri:ap-3*}
\Big(\dashint_{\frac{1}{4}B\cap\Omega}
 |(\nabla z,\frac{z}{\sqrt{T}})|^q\Big)^{\frac{1}{q}}
 \lesssim_{\lambda,d,\partial\Omega,q,q_0}  \Big(\dashint_{B\cap\Omega}
 |(\nabla z,\frac{z}{\sqrt{T}})|^{q_0}\Big)^{\frac{1}{q_0}}
\end{equation}
for any $q_0>0$.
By Minkowski's inequality, we have
\begin{equation*}
\begin{aligned}
&\Big(\dashint_{\frac{1}{4}B\cap\Omega} Z_B^{\bar{q}}\Big)^{\frac{p}{\bar{q}}}
= \Big(\dashint_{\frac{1}{4}B\cap\Omega} \big\langle|(\nabla z,\frac{z}{\sqrt{T}})|^p\big\rangle^{\frac{\bar{q}}{p}}
\Big)^{\frac{p}{\bar{q}}}\\
&\leq \Big\langle
\big(\dashint_{\frac{1}{4}B\cap\Omega}
|(\nabla z,\frac{z}{\sqrt{T}})|^{\bar{q}}\big)^{\frac{p}{\bar{q}}}\Big\rangle
\lesssim^{\eqref{pri:ap-3*}}
\Big\langle
\dashint_{\frac{1}{2}B\cap\Omega}
|(\nabla z,\frac{z}{\sqrt{T}})|^{p}\Big\rangle
\lesssim^{\eqref{pri:ap-6}}
\dashint_{B\cap\Omega} (V^{p}+F^{p}),
\end{aligned}
\end{equation*}
where we also use the fact $z=v-w$ in $B\cap\Omega$ in the last inequality.

\medskip
\noindent
\textbf{Part III. Arguments for $\eqref{pri:A-3}$.} The proof will be finished by three steps. Assume $\Omega$ is bounded, and $B_0$ and $B$ are  given as in Lemma $\ref{shen's lemma2}$.

\noindent
\textbf{Step 1.} Reduction and outline of the proof.
Let $0<\underline{q}-1\ll 1$, and $\bar{q}\in(1,\infty)$
is arbitrarily large. For any $\underline{q}<q<\bar{q}$ and $\omega\in A_{q/\underline{q}}$, it suffices to show
\begin{subequations}
\begin{align}\label{}
& \dashint_{\frac{1}{2}B\cap\Omega}W_B^{\underline{q}}
 \leq \dashint_{B\cap\Omega}F^{\underline{q}};
\label{f:ap-8} \\
& \Big(\dashint_{\frac{1}{2}B\cap\Omega}Z_B^{\bar{p}}\omega
\Big)^{\underline{q}/\bar{q}}
 \leq \dashint_{B\cap\Omega}(V^{\underline{q}}+F^{\underline{q}})
 \Big(\dashint_{B}\omega\Big)^{\underline{q}/\bar{q}}.
\label{f:ap-9}
\end{align}
\end{subequations}

Once we have the above two estimates, it follows from the estimate $\eqref{shen-3}$ that
\begin{equation*}
\Big(\dashint_{B_0\cap\Omega}V^q\omega\Big)^{\frac{1}{q}}
\lesssim \bigg\{\Big(\dashint_{4B_0\cap\Omega}|V|^{\underline{q}}\Big)^{\frac{1}{\underline{q}}}
\Big(\dashint_{B_0}\omega\Big)^{1/{q}}
+\Big(\dashint_{4B_0\cap\Omega}|F|^q\omega\Big)^{\frac{1}{q}}\bigg\}.
\end{equation*}
Since the above estimate is translation-invariant alone $\partial\Omega$,
together with its interior counterpart, a covering argument leads to
the following estimate
\begin{equation}\label{f:ap-11}
\begin{aligned}
\dashint_{\Omega}V^q\omega
&\lesssim
\Big(\dashint_{\Omega}V^{\underline{q}}\Big)^{q/\underline{q}}
\Big(\dashint_{B_0}\omega\Big)
+\dashint_{\Omega}|F|^{q}\omega\\
&\lesssim^{\eqref{pri:A-2}}
\Big(\dashint_{\Omega}F^{\underline{q}}\Big)^{q/\underline{q}}
\Big(\dashint_{\Omega}\omega\Big)
+\dashint_{\Omega}|F|^{q}\omega
\lesssim^{\eqref{f:w-1}} \dashint_{\Omega}|F|^{q}\omega,
\end{aligned}
\end{equation}
where we also employ H\"older's inequality in the last inequality. By noting the condition $0<\underline{q}-1\ll 1$,
we have already proved the stated estimate $\eqref{pri:A-3}$ for any $\omega\in A_{q}$.

\medskip
\noindent
\textbf{Step 2.} Show the estimate $\eqref{f:ap-8}$, which is parallel to $\eqref{pri:ap-6}$ but based upon the estimate $\eqref{pri:A-2}$. Thus,
a similar computation leads to
\begin{equation*}
\dashint_{\frac{1}{2}B\cap\Omega} W_B^{\underline{q}}
\lesssim_d \frac{1}{|B|}
\int_{\Omega}\big\langle
|(\nabla w,\frac{w}{\sqrt{T}})|^p\big\rangle^{\underline{q}/p}
\lesssim
\frac{1}{|B|}\int_{\Omega}\big\langle
|I_{B}(f,\frac{g}{\sqrt{T}})|^p\big\rangle^{\underline{q}/p}
\lesssim \dashint_{B\cap\Omega} F^{\underline{q}},
\end{equation*}
This completes the proof for the estimate $\eqref{f:ap-8}$.

\medskip
\noindent
\textbf{Step 3.} Show the estimate $\eqref{f:ap-9}$, which is parallel to $\eqref{pri:ap-5}$. By using Minkowski's inequality, Holder's inequality
and the reverse Holder's inequality for $z$ and the weight $\omega$ in
the order, we obtain
\begin{equation*}
\begin{aligned}
&\Big(\dashint_{\frac{1}{4}B\cap\Omega} Z_B^{\bar{q}}\omega\Big)^{\underline{q}/\bar{q}}
= \Big(\dashint_{\frac{1}{4}B\cap\Omega} \big\langle|(\nabla z,\frac{z}{\sqrt{T}})|^p\big\rangle^{\frac{\bar{q}}{p}}\omega
\Big)^{\underline{q}/\bar{q}}\\
&\leq \Big\langle
\Big(\dashint_{\frac{1}{4}B\cap\Omega}
|(\nabla z,\frac{z}{\sqrt{T}})|^{\bar{q}}\omega\Big)^{\underline{q}/\bar{q}}\Big\rangle
\lesssim
\Big\langle
\Big(\dashint_{\frac{1}{4}B\cap\Omega}
|(\nabla z,\frac{z}{\sqrt{T}})|^{\bar{q}\gamma}
\Big)^{\frac{\underline{q}}{\gamma\bar{q}}}\Big\rangle
\Big(\dashint_{B}\omega^{\gamma'}\Big)^{\frac{\underline{q}}{\gamma'\bar{q}}}\\
&\lesssim^{\eqref{pri:ap-3}}
\Big\langle
\dashint_{\frac{1}{2}B\cap\Omega}
|(\nabla z,\frac{z}{\sqrt{T}})|^{\underline{q}}\Big\rangle
\Big(\dashint_{B}\omega\Big)^{\underline{q}/\bar{q}}
\lesssim^{\eqref{f:ap-8}}
\dashint_{B\cap\Omega} (V^{\underline{q}}+F^{\underline{q}})
\Big(\dashint_{B}\omega\Big)^{\underline{q}/\bar{q}},
\end{aligned}
\end{equation*}

\medskip
\noindent
\textbf{Part IV. Arguments for $\eqref{pri:A-3}$ for unbounded domain.}
In such the case, the estimates $\eqref{f:ap-8}$ and $\eqref{f:ap-9}$
are still valid. Let $B_0\subset\Omega$ be arbitrary with $x_{B_0}\in\partial\Omega$. We merely modify the estimate $\eqref{f:ap-11}$
as follows:
\begin{equation*}\label{}
\begin{aligned}
\int_{B_0\cap\Omega} V^q\omega
\lesssim r_{B_0}^{d(1-q/\underline{q})}\Big(
\dashint_{B_0}\omega\Big)\Big(
\int_{\Omega}V^{\underline{q}}\Big)^{\frac{q}{\underline{q}}}
+\int_{\Omega}F^q\omega
\lesssim^{\eqref{pri:A-2}}r_{B_0}^{d(1-q/\underline{q})}\Big(
\dashint_{B_0}\omega\Big)
\Big(
\int_{\Omega}F^{\underline{q}}\Big)^{\frac{q}{\underline{q}}}
+\int_{\Omega}F^q\omega.
\end{aligned}
\end{equation*}
To see the desired estimate, it suffices to show
\begin{equation}\label{f:ap-14}
\lim_{r_{B_0}\to\infty}r_{B_0}^{d(1-q/\underline{q})}\Big(
\dashint_{B_0}\omega\Big) = 0.
\end{equation}
On account of Lemma $\ref{weight}$, we can infer that
that $\omega\in A_{q/\underline{q}}$ and further implies $\omega\in A_{p}$
with $1<p<q/\underline{q}$. By definition, we obtain
\begin{equation*}
r_{B_0}^{d(1-q/\underline{q})}\Big(
\dashint_{B_0}\omega\Big)
\lesssim^{\eqref{f:w-1}}_{[\omega]_{A_p}}
r_{B_0}^{d(p-q/\underline{q})}\Big(
\int_{B_0}\omega^{-\frac{1}{p-1}}\Big)^{1-p}
\end{equation*}
where we note that the integral on the right-hand side above is not increasing as $r_{B_0}\to\infty$. Consequently,
let $r_{B_0}$ go to infinity, and we have $\eqref{f:ap-14}$
and this ends the whole proof.
\end{proof}

%\begin{remark}
%\emph{The condition $\sqrt{T}\geq R_0$ in Lemma $\ref{lemma:ap-2}$ is merely a
%technical assumptions, which could be further weaken into $T>0$ by modifying the estimates $\eqref{f:ap-1}$ and $\eqref{f:ap-2}$. This consequently leads to the resolvent estimates.
%It is of independent interest and is therefore left to the reader. }
%\end{remark}

\begin{lemma}[Meyer type estimates]\label{lemma:2.2}
Let $\Omega$ be a (bounded) Lipschitz domain and $T>0$. Suppose that the admissible coefficient $a$ satisfies the uniform ellipticity condition $\eqref{a:1}$.
Assume that $v\in H_0^1(\Omega)$ is associated with the given data $(f,g)$ by the following equations:
\begin{equation}\label{pde:2}
\left\{\begin{aligned}
\frac{1}{T}v-\nabla\cdot a\nabla v&= \nabla\cdot f + \frac{1}{T}g
&~&\text{in}~~\Omega;\\
v&= 0 &~&\text{on}~\partial\Omega.
\end{aligned}\right.
\end{equation}
Then, for any $0<|p-2|\ll 1$ and $0<|q-2|\ll 1$, there holds
the following  estimates
\begin{subequations}
\begin{align}
&\int_{\Omega} \Big(\dashint_{U_1(x)}|(\frac{v}{\sqrt{T}},\nabla v)|^2\Big)^{\frac{p}{2}}dx
 \lesssim
 \int_{\Omega} \Big(\dashint_{U_1(x)}|(\frac{g}{\sqrt{T}},f)|^2\Big)^{\frac{p}{2}}dx;
\label{pri:2.4} \\
&\int_{\Omega}
\Big\langle\Big(\dashint_{U_1(x)}\big|(\frac{v}{\sqrt{T'}},\nabla v)\big|^2\Big)^{\frac{p}{2}}
\Big\rangle^{\frac{q}{p}}dx
\lesssim
\int_{\Omega}
\Big\langle\Big(\dashint_{U_1(x)}
\big|(\frac{g}{\sqrt{T'}},f)\big|^2\Big)^{\frac{p}{2}}
\Big\rangle^{\frac{q}{p}}dx.
\label{pri:2.3}
\end{align}
\end{subequations}
\end{lemma}

\begin{proof}
The main idea of the proof can be found in \cite{Wang-Xu22}, which is based upon Lemma $\ref{shen's lemma2}$ (by simply taking $\omega=1$ therein).
How to utilize Shen's real arguments has been well stated in \cite[Subsection 3.1]{Wang-Xu22}. Since the interior estimates can be similarly derived, the present demonstration focuses on boundary estimates.

\noindent
\textbf{Part I. Arguments for $\eqref{pri:2.4}$.}
Let $0<p_1-2\ll 1$, and $2<p<p_1$, and the case of $0<2-p\ll 1$ follows from a duality argument (see e.g. \cite[Subsection 3.1]{Wang-Xu22}).
Let $B_0$ be fixed with $x_{B_0}\in\partial\Omega$, and we introduce the following notations:
\begin{equation*}
 V(x):= \Big(\dashint_{U_1(x)}|(\frac{v}{\sqrt{T}},\nabla v)|^2\Big)^{\frac{1}{2}};
 \qquad
 G(x):= \Big(\dashint_{U_1(x)}|(\frac{g}{\sqrt{T}},f)|^2\Big)^{\frac{1}{2}}.
\end{equation*}
Then, to see the stated estimate $\eqref{pri:2.4}$, it is reduced to show
\begin{equation}\label{f:6.3}
\dashint_{B_0\cap\Omega} V^p
 \lesssim
\Big(\dashint_{4B_0\cap\Omega} V^2\Big)^{\frac{p}{2}}
+ \dashint_{4B_0\cap\Omega} G^p.
\end{equation}
This together with a covering argument (see e.g. \cite[Subsection 3.1]{Wang-Xu22}) and
\begin{equation}\label{f:6.2}
 \int_{\Omega} V^2
 \lesssim^{\eqref{g-1*}}
 \int_{\Omega} |(\frac{v}{\sqrt{T}},\nabla v)|^2
 \lesssim_{d,\lambda}
 \int_{\Omega} |(\frac{g}{\sqrt{T}},f)|^2
 \lesssim^{\eqref{g-1*}}\int_{\Omega} G^2
\end{equation}
leads to the desired estimate $\eqref{pri:2.4}$, and the estimate $\eqref{f:6.2}$ is due to the energy estimate.
The proof of $\eqref{f:6.3}$ is divided into
three steps.

\noindent
\textbf{Step 1.}
Reduction and decomposition of the field $(\frac{v}{\sqrt{T}},\nabla v)$.
Let $B$ be arbitrary ball
with the properties: $r_B\leq (1/100)r_{B_0}$ and
either $x_B\in 2B_0\cap\partial\Omega$ or
$4B\subset 2B_0\cap\Omega$. Here, we merely
fucus on the case $x_B\in 2B_0\cap \partial\Omega$ and
set $D:=B\cap\Omega$ and $\Delta:=B\cap\partial\Omega$.
\begin{equation*}
\left\{\begin{aligned}
 \frac{1}{T}w - \nabla\cdot a\nabla w &=
 \nabla\cdot fI_{10D}+\frac{1}{T}gI_{10D}
 &\quad& \text{in}\quad \Omega;\\
 w &= 0
 &\quad& \text{on}\quad \partial\Omega,
\end{aligned}\right.
\end{equation*}
with $z:=v-w$ in $10D$ satisfying
\begin{equation*}
\left\{\begin{aligned}
\nabla\cdot a\nabla z &= \frac{1}{T}z &\quad&\text{in}\quad 10D;\\
z &=0 &\quad&\text{on}\quad 10\Delta.
\end{aligned}\right.
\end{equation*}
We also introduce
the following notations:
\begin{equation*}
 Z_B(x):= \Big(\dashint_{U_1(x)}|(\frac{z}{\sqrt{T}},\nabla z)|^2\Big)^{\frac{1}{2}};
 \qquad
 W_B(x):= \Big(\dashint_{U_1(x)}|(\frac{w}{\sqrt{T}},\nabla w)|^2\Big)^{\frac{1}{2}},
\end{equation*}
and one can observe that
$V\leq W_{B}+Z_{B}$
on $2D$. We claim that:
\begin{subequations}\label{}
\begin{align}
&\Big(\dashint_{D}
Z_B^{p_1}
\Big)^{\frac{1}{p_1}} \lesssim
\Big(\dashint_{20D}
(V^{2}+G^{2})\Big)^{\frac{1}{2}};
\label{f:6.4-1}\\
&\dashint_{D}
W_B^{2}
\lesssim \dashint_{20D}G^{2}.
\label{f:6.4-2}
\end{align}
\end{subequations}

Then, admitting the estimates $\eqref{f:6.4-1}$ and $\eqref{f:6.4-2}$ for a while, by using Lemma $\ref{shen's lemma2}$, we obtain the desired estimate
$\eqref{f:6.3}$. Now, the proof of the above estimates $\eqref{f:6.4-1}$ and $\eqref{f:6.4-2}$ will be divided into two cases:
(I). $0<r_B<(1/4)$;
(II). $r_B\geq (1/4)$.

\noindent
\textbf{Step 2.}
Establish the
estimates $\eqref{f:6.4-1}$ and $\eqref{f:6.4-2}$
in the case of (I). In such the case, we merely choose
$Z_B = V$
and $W_B = 0$, and
the estimate $\eqref{f:6.4-2}$ is trivial while the
estimate $\eqref{f:6.4-1}$ is reduced to show
\begin{equation}\label{f:6.5}
\Big(\dashint_{D}
V^{p_1}\Big)^{\frac{1}{p_1}}
\lesssim \Big(\dashint_{5D}
V^{2}\Big)^{\frac{1}{2}}.
\end{equation}
To see this, for any $x_0\in D$, it follows from the estimate $\eqref{g-0}$ that
\begin{equation*}
\big[V(x_0)\big]^2\lesssim \dashint_{5D}V^{2}
\quad \Longrightarrow \quad
\big[V(x_0)\big]^{p_1}\lesssim \Big(\dashint_{5D}V^{2}\Big)^{\frac{p_1}{2}}.
\end{equation*}
Integrating both sides above with respective to
$x_0\in D$, we obtain the stated estimate $\eqref{f:6.5}$ and therefore
implies the desired estimate $\eqref{f:6.4-1}$ in such the case.

\noindent
\textbf{Step 3.} Show the estimates
$\eqref{f:6.4-1}$ and $\eqref{f:6.4-2}$ in the case of (II). We
first handle the estimate $\eqref{f:6.4-2}$, and start from
\begin{equation}\label{f:6.6}
\begin{aligned}
\dashint_{D}
W_B^{2}
\lesssim^{\eqref{g-1}}
\dashint_{4D}|(\frac{w}{\sqrt{T}},\nabla w)|^2
\lesssim_{\lambda,d}
\dashint_{10D}|(\frac{g}{\sqrt{T}},f)|^2
\lesssim^{\eqref{g-1}}
\dashint_{20D}G^2,
\end{aligned}
\end{equation}
where we employ energy estimates in the second inequality.
Then, it follows that
\begin{equation*}
\begin{aligned}
\dashint_{D}Z_B^{p_1}
&\lesssim^{\eqref{g-1}}
\dashint_{4D}\big|(\frac{z}{\sqrt{T}},\nabla z)\big|^{p_1}\\
&\lesssim^{(\text{Meyers})}
\Big(\dashint_{5D}\big|(\frac{z}{\sqrt{T}},\nabla z)\big|^{2}\Big)^{\frac{p_1}{2}}
\lesssim^{\eqref{f:6.6},\eqref{g-1}}
\Big(\dashint_{10D} V^2 +
\dashint_{20D}G^2\Big)^{\frac{p_1}{2}},
\end{aligned}
\end{equation*}
and this implies the desired estimate $\eqref{f:6.4-1}$. Moreover,
for any $s>0$, there holds
\begin{equation}\label{f:6.13}
\Big(\dashint_{D}Z_B^{p_1}\Big)^{\frac{1}{p_1}}
%\lesssim \Big(\dashint_{D}Z_B^{2}\Big)^{\frac{1}{2}}
\lesssim \Big(\dashint_{\alpha D}Z_B^{s}\Big)^{\frac{1}{s}},
\end{equation}
which follows from a convexity argument (see e.g. \cite[pp.173]{Fefferman-Stein72}), where $\alpha>1$.

\medskip
\noindent
\textbf{Part II. Arguments for $\eqref{pri:2.3}$.}
Let $0<q_1-2\ll 1$ and $0<|2-p|\ll 1$, and $p\leq q<q_1$, and the case of $q<p$ follows from the duality argument.
For any ball $B$ with
the property that either $x_B\in\partial\Omega$ or
$4B\subset\Omega$ (and here we just focus on the case $x_B\in\partial\Omega$), let $w,z\in H_0^1(\Omega)$ satisfy the following equations:
\begin{equation}
 \frac{1}{T}w-\nabla\cdot a\nabla w =
 \nabla \cdot fI_{2D}+\frac{1}{T}gI_{2D},
 \quad \text{and}\quad
 \frac{1}{T}w-\nabla\cdot a\nabla z
 =\nabla \cdot fI_{\Omega\setminus2D}
 +\frac{1}{T}gI_{\Omega\setminus 2D}\quad\text{in}\quad \Omega,
\end{equation}
and we introduce the following notations:
\begin{equation*}
\begin{aligned}
& V(x)
 :=\Big\langle\Big(\dashint_{U_{1}(x)}|(\frac{v}{\sqrt{T}},\nabla v)|^2\Big)^{\frac{p}{2}}\Big\rangle^{\frac{1}{p}};
 &\qquad&
 G(x)
 :=\Big\langle\Big(\dashint_{U_{1}(x)}|(\frac{g}{\sqrt{T}},
 f)|^2\Big)^{\frac{p}{2}}\Big\rangle^{\frac{1}{p}};\\
& W_B(x):=\Big\langle
 \Big(\dashint_{U_{1}(x)}
 |(\frac{w}{\sqrt{T}},\nabla w)|^2\Big)^{\frac{p}{2}}\Big\rangle^{\frac{1}{p}};
 &\qquad&
 Z_B(x)
 :=\Big\langle
 \Big(\dashint_{U_{1}(x)}
 |(\frac{z}{\sqrt{T}},\nabla z)|^2\Big)^{\frac{p}{2}}\Big\rangle^{\frac{1}{p}}.
\end{aligned}
\end{equation*}

Thus, we have $F\leq W_B+V_B$ due to the relationship
$u_\varepsilon = w_\varepsilon + v_\varepsilon$ in $\Omega$.
According to the preconditions stated in
Lemma $\ref{shen's lemma2}$, to see $\eqref{pri:2.3}$, it is reduced to establish
the following two estimates:
\begin{subequations}
\begin{align}
& \dashint_{\frac{1}{2}D}W_B^{p}
 \lesssim \dashint_{7D}
 G^{p}; \label{f:6.7}\\
& \Big(\dashint_{\frac{1}{8}D}Z_B^{q_1}\Big)^{\frac{p}{q_1}}
 \lesssim \dashint_{7D}(V^{p}+g^{p}).
 \label{f:6.8}
\end{align}
\end{subequations}
Therefore, the remainder of the proof is mainly to show the estimates $\eqref{f:6.7}$ and $\eqref{f:6.8}$ by the following two steps.

\noindent
\textbf{Step 1.} Show the estimate $\eqref{f:6.7}$.
It is reduced to showing
\begin{equation}\label{f:6.10}
\dashint_{\frac{1}{2}D}
\Big(\dashint_{U_{1}(x)}
|(\frac{w}{\sqrt{T}},\nabla w)|^2\Big)^{\frac{p}{2}}
 \lesssim \dashint_{7D}
\Big(\dashint_{U_{1}(x)}
|(f,\frac{g}{\sqrt{T}})|^{2}\Big)^{\frac{p}{\bar{2}}}.
\end{equation}
We discuss it by two cases: (1) $r_B\geq (1/4)$;
(2) $r_B < (1/4)$. For the case (1), we first observe that for any $x\in\Omega$ being such that
\begin{equation}\label{f:6.9}
 |x-x_B|\leq 2r_B + 1 \leq 6r_B,
\end{equation}
there holds $\dashint_{U_1(x)}|(fI_{2D},\frac{gI_{2D}}{\sqrt{T}})|^2\not=0$.
Then, we have
\begin{equation}\label{f:6.12}
\begin{aligned}
\dashint_{D}
\Big(\dashint_{U_{1}(x)}
|(\frac{w}{\sqrt{T}}
&,\nabla w)|^2\Big)^{\frac{p}{2}}
\leq \frac{1}{|D|}
\int_{\Omega}
\Big(\dashint_{U_{1}(x)}
|(\frac{w}{\sqrt{T}},\nabla w)|^2\Big)^{\frac{p}{2}}\\
&\lesssim^{\eqref{pri:2.4}}
\frac{1}{|D|}
\int_{\Omega}
\Big(\dashint_{U_{1}(x)}
|(fI_{2D},\frac{gI_{2D}}{\sqrt{T}})|^2\Big)^{\frac{p}{2}}
\lesssim^{\eqref{f:6.9}} \dashint_{6D}
\Big(\dashint_{U_{1}(x)}
|(f,\frac{g}{\sqrt{T}})|^2\Big)^{\frac{p}{2}}.
\end{aligned}
\end{equation}
We now address the case (2).
For any $x_0\in\frac{1}{2}D$,
by noting $2B\subset B_{1}(x_0)$ and using
the energy estimate, we have
\begin{equation}\label{f:6.11}
 \dashint_{U_{1}(x_0)} |(\frac{w}{\sqrt{T}}
,\nabla w)|^2
 \lesssim
 \dashint_{U_{1}(x_0)} |(f,\frac{g}{\sqrt{T}})|^{2}.
\end{equation}
Then, appealing to Lemma $\ref{integral-geometry}$,
it follows that
\begin{equation*}
\Big(\dashint_{U_{1}(x_0)} |(\frac{w}{\sqrt{T}}
,\nabla w)|^2\Big)^{\frac{p}{2}}
\lesssim^{\eqref{f:6.11}}
\Big(\dashint_{U_{1}(x_0)} |f|^{2}
\Big)^{\frac{p}{2}}
\lesssim^{\eqref{g-0}}
\dashint_{5D}
\Big(\dashint_{U_{1}(x)} |(f,\frac{g}{\sqrt{T}})|^{2}\Big)^{\frac{p}{2}}.
\end{equation*}
Integrating both sides above with respect to $x_0\in\frac{1}{2}D$ yields the estimate $\eqref{f:6.10}$ in such the case. This coupled with
$\eqref{f:6.12}$ leads to the whole argument of the estimate $\eqref{f:6.10}$, and finally taking expectation on both
sides of $\eqref{f:6.10}$ provides us the desired estimate $\eqref{f:6.7}$.

\noindent
\textbf{Step 2.} Show the estimate $\eqref{f:6.8}$.
By using Minkowski's inequality, we have
\begin{equation*}
\begin{aligned}
\Big(\dashint_{\frac{1}{8}D}
Z_B^{q_1}\Big)^{\frac{p}{q_1}}
&\leq
\bigg\langle
\Big(\dashint_{\frac{1}{8}D}
\Big(\dashint_{
U_{1}(x)}|(\frac{z}{\sqrt{T}},\nabla z)|^2
\Big)^{\frac{q_1}{2}}dx\Big)^{\frac{p}{q_1}}\bigg\rangle dx\\
&\lesssim^{\eqref{f:6.13}}
\dashint_{\frac{1}{2}D}\Big\langle
\Big(\dashint_{U_{1}(x)}
|(\frac{z}{\sqrt{T}},\nabla z)|^2\Big)^{\frac{p}{2}}\Big\rangle dx
\lesssim^{\eqref{f:6.7}} \dashint_{\frac{1}{2}D}V^{p}
+ \dashint_{7D}G^{p},
\end{aligned}
\end{equation*}
which implies the desired estimate $\eqref{f:6.8}$, and this ends
the whole proof.
\end{proof}

\section{Suboptimal annealed Calder\'on-Zygmund estimates}\label{sec:3}

\begin{theorem}\label{thm:2}
Let $\Omega$ be a (bounded) $C^1$ domain and $T\geq 1$.
Suppose that the ensemble $\langle\cdot\rangle$
is stationary and satisfies the spectral gap condition $\eqref{a:2}$,
as well as $\eqref{a:3}$. Let $u$ be associated with
$\tilde{f}$ and $\tilde{g}$ by the following equations:
\begin{equation}\label{pde:3.2x}
\left\{\begin{aligned}
\frac{1}{T}u-\nabla\cdot a\nabla u
&= \nabla\cdot f +\frac{1}{T}g
&\quad&\text{in}~~\Omega;\\
u &= 0
&\quad&\text{on}~~\partial\Omega.
\end{aligned}\right.
\end{equation}
Then, for any $1<p,r<\infty$ with $\bar{r}>r$, there holds
the following annealed estimate
\begin{equation}\label{pri:5.0}
\begin{aligned}
\Big(\int_{\Omega} \big\langle|(\nabla u,\frac{u}{\sqrt{T}} )|^{r}\big\rangle^{\frac{p}{r}}\Big)^{\frac{1}{p}}
&\leq C_{p,r,\bar{r}}(T)
\Big(\int_{\Omega}
\big\langle|(f,\frac{g}{\sqrt{T}})|^{\bar{r}}
\big\rangle^{\frac{p}{\bar{r}}}\Big)^{\frac{1}{p}}.
\end{aligned}
\end{equation}
Moreover,  for any $\omega\in A_p$, we obtain that
\begin{equation}\label{pri:5.0-w}
\begin{aligned}
\Big(\int_{\Omega} \big\langle|(\nabla u,\frac{u}{\sqrt{T}} )|^{r}\big\rangle^{\frac{p}{r}}\omega\Big)^{\frac{1}{p}}
&\leq \tilde{C}_{p,r,\bar{r}}(T)
\Big(\int_{\Omega}
\big\langle|(f,\frac{g}{\sqrt{T}})|^{\bar{r}}
\big\rangle^{\frac{p}{\bar{r}}}\omega\Big)^{\frac{1}{p}},
\end{aligned}
\end{equation}
where the multiplicative constants are estimated by
\begin{equation}\label{rv-bound-1}
C_{p,r,\bar{r}}(T)\lesssim_{\lambda,d,M_0,p,q}
T^{\frac{d}{2}};\quad
\tilde{C}_{p,r,\bar{r}}(T)\lesssim_{
\lambda,d,M_0,p,q,[\omega]_{A_q}}
T^{\frac{3d}{4}}.
\end{equation}
\end{theorem}

\begin{remark}
\emph{For the case of $0<T<1$, the estimates $\eqref{pri:5.0}$ and
$\eqref{pri:5.0-w}$ in Theorem $\ref{thm:2}$ are true for the bounded multiplicative constants independent of $T$. In fact, Section $\ref{sec:4}$
is devoted to showing that this is also true for the case of $T\geq 1$.}
\end{remark}

\subsection{Suboptimal estimates on the mixed norm}\label{subsec:3.1}

\begin{proposition}\label{P:3}
Let $\Omega$ be a (bounded) Lipschitz domain and $T\geq 1$.
Suppose that the ensemble $\langle\cdot\rangle$
is stationary.
Let $u$ be associated with
$f$ and $g$ by the equations $\eqref{pde:3.2x}$.
Then, for any $1<p,q<\infty$ and $1<r'<r<\infty$, there holds
the following annealed estimate
\begin{equation}\label{pri:2.1}
\bigg(\int_{\Omega}
\Big\langle\Big(\dashint_{U_1(x)}\big|(\frac{u}{\sqrt{T}},\nabla u)\big|^q\Big)^{\frac{r'}{q}}
\Big\rangle^{\frac{p}{r'}}dx\bigg)^{\frac{1}{p}}
\lesssim \sqrt{T}^{d}
\bigg(\int_{\Omega}
\Big\langle\Big(\dashint_{U_1(x)}
\big|(\frac{g}{\sqrt{T}},f)\big|^q\Big)^{\frac{r}{q}}
\Big\rangle^{\frac{p}{r}}dx\bigg)^{\frac{1}{p}}.
\end{equation}
Moreover, for any $\omega\in A_p$, we have
\begin{equation}\label{pri:2.1-w}
\begin{aligned}
\bigg(\int_{\Omega}
\Big\langle\Big(\dashint_{U_1(x)}
\big|(\frac{u}{\sqrt{T}},\nabla u)\big|^q\Big)^{\frac{r'}{q}}
&\Big\rangle^{\frac{p}{r'}}
\Big(\dashint_{B_1(x)}\omega\Big)dx\bigg)^{\frac{1}{p}}\\
&\lesssim T^{\frac{3d}{4}}
\bigg(\int_{\Omega}
\Big\langle\Big(\dashint_{U_1(x)}
\big|(\frac{g}{\sqrt{T}},f)\big|^q\Big)^{\frac{r}{q}}
\Big\rangle^{\frac{p}{r}}
\omega dx\bigg)^{\frac{1}{p}}.
\end{aligned}
\end{equation}
In particular, for the case of $T\in(0,1)$, the estimates $\eqref{pri:2.1}$ and $\eqref{pri:2.1-w}$ hold for the multiplicative constants
independent of $T$.
\end{proposition}

For the ease of the statements,
we impose the following notation for the mixed norms:
\begin{equation}\label{mnorm}
\begin{aligned}
&\big\|h\big\|_{p,r,q}
:=\bigg(\int_{\mathbb{R}^d}
\Big\langle\Big(\dashint_{B_1(x)}|h|^q\Big)^{\frac{r}{q}}
\Big\rangle^{\frac{p}{r}}dx\bigg)^{\frac{1}{p}},
\qquad
\big\|h\big\|_{p,r}
:=\Big(\int_{\mathbb{R}^d}\big\langle
|h|^r\big\rangle^{\frac{p}{r}}\Big);\\
&\big\|h;\omega^{\frac{1}{p}}\big\|_{p,r,q}
:=\bigg(\int_{\mathbb{R}^d}
\Big\langle\Big(\dashint_{B_1(x)}|h|^q\Big)^{\frac{r}{q}}
\Big\rangle^{\frac{p}{r}} \omega(x) dx\bigg)^{\frac{1}{p}},
\end{aligned}
\end{equation}

\begin{lemma}\label{lemma:3.1}
Let $\|\cdot\|_{p,r,q}$ be the mixed norm defined above, then there holds
the following properties:
\begin{subequations}
\begin{align}
& \|h\|_{\bar{p},r,q}\lesssim \|h\|_{p,r,q}\qquad  &~&\text{provided}~p\leq\bar{p}; \label{pri:3.1-a}\\
& \|h\|_{p,r'}\leq \|h\|_{p,r} \quad\text{and}\quad
\|h\|_{p,r',q}\leq \|h\|_{p,r,q}
\qquad
&~&\text{provided}~r'\leq r;  \\
& \|h\|_{p,r,q} \leq \|h\|_{p,r}
\qquad
&~&\text{provided}~q\leq \min\{p,r\}.
\label{pri:3.1-c}
\end{align}
\end{subequations}
\end{lemma}

\begin{proof}
See \cite[pp.43-44]{Josien-Otto22}.
\end{proof}

\noindent
\textbf{Proof of Proposition $\ref{P:3}$.}
The main idea and ingredients are similar to
those given for \cite[Proposition 7.3]{Josien-Otto22}.
We mention that once the energy estimate and local regularity estimate hold correspondingly, the rest part of the proof is actually independent of PDEs, and this is the reason why the original proof can be easily applied to the current (bounded) domain.
To deal with weighted estimate $\eqref{pri:2.1-w}$,
we mostly keep the same idea as in the estimate $\eqref{pri:2.1}$.
Using the properties of the $A_p$-weight class, we surprisingly obtain a similar suboptimal estimate. The whole proof is divided into 9 steps.
The first five steps are to prove the estimate $\eqref{pri:2.1}$,
while the last four steps are to show the weighted one $\eqref{pri:2.1-w}$. By zero extension of $u$, $g$, and $f$ from $\Omega$ to $\mathbb{R}^d$,
we still adopt the same notation here.

\medskip
\noindent
\textbf{Step 1.} Outline the proof and reduction. Let
$p,q\in(1,\infty)$ and $2< r'<r$.
To show the estimate $\eqref{pri:2.1}$ in such the case, it suffices to establish
\begin{equation}\label{f:2.2*}
\begin{aligned}
\|\omega_{\frac{T}{4}}\big(\frac{u}{\sqrt{T}},\nabla u\big)\|_{p,r',q}
\lesssim \sqrt{T}^{d}
\|\omega_{4T}\big(\frac{g}{\sqrt{T}},f\big)\|_{p,r',q},
\end{aligned}
\end{equation}
since the above estimate is shift-invariant and therefore one can rewrite it as
\begin{equation}\label{f:2.1*}
\begin{aligned}
\|\omega_{\frac{T}{4}}(\cdot-z)\big(\frac{u}{\sqrt{T}},\nabla u\big)\|_{p,r',q}
\lesssim \sqrt{T}^{d}
\|\omega_{4T}(\cdot-z)\big(\frac{g}{\sqrt{T}},f\big)\|_{p,r',q}.
\end{aligned}
\end{equation}
Then taking $L^p$-norm on the both sides of $\eqref{f:2.1*}$ with respect to $z\in\mathbb{R}^d$, there holds
\begin{equation*}
\|\big(\frac{u}{\sqrt{T}},\nabla u\big)\|_{p,r',q}
\lesssim \sqrt{T}^{d}
\|\big(\frac{g}{\sqrt{T}},f\big)\|_{p,r',q},
\end{equation*}
which implies the stated estimate $\eqref{pri:2.1}$ by noting the zero-extension of $u$, $g$, and $f$.

To see $\eqref{f:2.2*}$, we need to establish
\begin{equation}\label{f:2.11}
 \|\omega_{T}\big(\frac{u}{\sqrt{T}},\nabla u\big)\|_{\infty,r',q}
 \lesssim \|\omega_{T}\big(\frac{u}{\sqrt{T}},\nabla u\big)\|_{2,r,q}.
\end{equation}
Then, by H\"older's inequality and $\omega_{T}(x):=\exp(-\frac{|x|}{C\sqrt{T}})$, one can derive that
\begin{equation*}
\begin{aligned}
&\|\omega_{\frac{T}{4}}\big(\frac{u}{\sqrt{T}},\nabla u\big)\|_{p,r',q}\\
&\leq \Big(\int_{\mathbb{R}^d}\omega_{T}^p\Big)^{\frac{1}{p}}
\|\omega_{T}\big(\frac{u}{\sqrt{T}},\nabla u\big)\|_{\infty,r',q}
\lesssim^{\eqref{f:2.11}}
 \Big(\int_{\mathbb{R}^d}\omega_{T}^p\Big)^{\frac{1}{p}}
\|\omega_{T}\big(\frac{g}{\sqrt{T}},f\big)\|_{2,r,q}
\end{aligned}
\end{equation*}
and this together with
\begin{equation*}
\|\omega_{T}\big(\frac{g}{\sqrt{T}},f\big)\|_{2,r,q}
\leq \left\{\begin{aligned}
&\Big(\int_{\mathbb{R}^d}\omega_{4T}^{\frac{2p}{p-2}}\Big)^{\frac{p-2}{2p}}
\|\omega_{4T}\big(\frac{g}{\sqrt{T}},f\big)\|_{p,r,q},
&\quad&\text{if}~p> 2;\\
&\|\omega_{T}\big(\frac{g}{\sqrt{T}},f\big)\|_{p,r,q},
&\quad&\text{if}~p\in(1,2],
\end{aligned}\right.
\end{equation*}
where the second line is due to $\eqref{pri:3.1-a}$, gives
\begin{equation}\label{f:3.7}
\|\omega_{\frac{T}{4}}\big(\frac{u}{\sqrt{T}},\nabla u\big)\|_{p,r',q}
\lesssim
\left\{\begin{aligned}
&\Big(\int_{\mathbb{R}^d}\omega_{T}^p\Big)^{\frac{1}{p}}
\Big(\int_{\mathbb{R}^d}\omega_{4T}^{\frac{2p}{p-2}}\Big)^{\frac{p-2}{2p}} \|\omega_{4T}\big(\frac{g}{\sqrt{T}},f\big)\|_{p,r,q},
&\quad&\text{if}~p> 2;\\
&\Big(\int_{\mathbb{R}^d}\omega_{T}^p\Big)^{\frac{1}{p}}
\|\omega_{T}\big(\frac{g}{\sqrt{T}},f\big)\|_{p,r,q}, &\quad&\text{if}~p\in(1,2].
\end{aligned}\right.
\end{equation}
As a result, plugging the following computations
\begin{equation}\label{f:3.8}
\Big(\int_{\mathbb{R}^d}\omega_{T}^p\Big)^{\frac{1}{p}}
\Big(\int_{\mathbb{R}^d}\omega_{4T}^{\frac{2p}{p-2}}\Big)^{\frac{p-2}{2p}}
\lesssim_{p,d} \sqrt{T}^{\frac{d}{2}}
~~\text{for}~p> 2;
\qquad
\Big(\int_{\mathbb{R}^d}\omega_{T}^p\Big)^{\frac{1}{p}}
\lesssim_{p,d} \sqrt{T}^{\frac{d}{p}}
\end{equation}
back into $\eqref{f:3.7}$
leads to the stated estimate $\eqref{f:2.2*}$, where we note that $T\geq 1$.  Consequently,
for the case of $1<r'<r\leq 2$, the estimate $\eqref{pri:2.1}$
follows from the duality argument. Also, we mention that
the estimates in $\eqref{f:3.8}$ holds true for $T>0$.

\medskip
\noindent
\textbf{Step 2.} Show the estimate $\eqref{f:2.11}$. In view of
Lemma $\ref{lemma:3.1}$, we first have
\begin{equation}\label{f:2.12}
\|\omega_{T}\big(\frac{u}{\sqrt{T}},\nabla u\big)\|_{\infty,r',q}
 \lesssim^{\eqref{pri:3.1-a}} \|\omega_{T}\big(\frac{u}{\sqrt{T}},\nabla u\big)\|_{2,r',q}.
\end{equation}
Once we established
\begin{equation}\label{f:2.13}
\|\omega_{T}\big(\frac{u}{\sqrt{T}},\nabla u\big)\|_{2,2,q}
\lesssim \|\omega_{T}\big(\frac{u}{\sqrt{T}},\nabla u\big)\|_{2,s',q},
\end{equation}
where $s'>2$ will be fixed later on,  we would obtain
\begin{equation}\label{f:2.14}
\|\omega_{T}\big(\frac{u}{\sqrt{T}},\nabla u\big)\|_{2,r',q}
\lesssim \|\omega_{T}\big(\frac{u}{\sqrt{T}},\nabla u\big)\|_{2,r,q}.
\end{equation}
Admitting it for a while, plugging the estimate $\eqref{f:2.14}$
back into $\eqref{f:2.12}$ gives the stated estimate $\eqref{f:2.11}$.
To see estimate $\eqref{f:2.14}$, we employ a duality argument. Let
$F$ be arbitrary random variable as a test function. Therefore, it follows from $\eqref{f:2.13}$ that
\begin{equation*}
\|\omega_{T}F\big(\frac{u}{\sqrt{T}},\nabla u\big)\|_{2,2,q}
\lesssim \|\omega_{T}F\big(\frac{u}{\sqrt{T}},\nabla u\big)\|_{2,s',q}
\lesssim \|\omega_{T}\big(\frac{u}{\sqrt{T}},\nabla u\big)\|_{2,r,q}
\langle|F|^{\frac{2r'}{r'-2}}\rangle^{\frac{r'-2}{2r'}},
\end{equation*}
where $s'>2$ satisfies $1/s':=(r'-2)/(2r')+1/r$. By noting $1/2=(r'-2)/(2r')+1/r'$,
the above estimate implies the desired one $\eqref{f:2.14}$.

\medskip
\noindent
\textbf{Step 3.} Arguments for $\eqref{f:2.13}$. It suffices to establish
\begin{equation}\label{f:3.2}
 \|\omega_{T}(\frac{u}{\sqrt{T}},\nabla u)\|_{2,2,2}
 \lesssim
 \|\omega_{T}(\frac{g}{\sqrt{T}},f)\|_{2,2,2},
\end{equation}
and for any $q'\in(1,\infty)$ and $s>2$ (to be fixed later on),
\begin{equation}\label{f:3.3}
 \|(\frac{u}{\sqrt{T}},\nabla u)\|_{2,2,q'}
 \lesssim \|(\frac{g}{\sqrt{T}},f)\|_{2,s,q'}.
\end{equation}

Then, it follows from
a complex interpolation argument that
\begin{equation*}
\|\omega_{T}^{\theta}\big(\frac{u}{\sqrt{T}},\nabla u\big)\|_{2,2,q}
\lesssim \|\omega_{T}^{\theta}\big(\frac{u}{\sqrt{T}},\nabla u\big)\|_{2,s',q},
\end{equation*}
where one can fix $\theta\in(0,1)$ and then define
$q\in(1,\infty)$ and $s'\in(2,s)$ by
\begin{equation*}
1/q=\theta/2+(1-\theta)/q',
\qquad
1/s' = \theta/2 + (1-\theta)/s.
\end{equation*}
Consequently, one can derive $\eqref{f:2.13}$ by
merely modifying the constant $C$ in the definition of $\omega_{T}$.

To see $\eqref{f:3.2}$, we start from the estimate
\begin{equation}\label{f:3.0}
\Big\langle \int_{\Omega}\big|\omega_{T}(\frac{u}{\sqrt{T}},\nabla u)\big|^2\Big\rangle
\lesssim \Big\langle\int_{\Omega}
|\omega_{T}(\frac{g}{\sqrt{T}},f)|^2\Big\rangle,
\end{equation}
where $\omega_{T}:=\exp(-\frac{|x|}{C\sqrt{T}})$ and $C=C(\lambda,d)$
is sufficiently large.
By zero extension of $u$, $g$ and $f$, we have
\begin{equation*}
 \|\omega_{T}(\frac{u}{\sqrt{T}},\nabla u)\|_{2,2}
 \lesssim^{\eqref{f:3.0}}
 \|\omega_{T}(\frac{g}{\sqrt{T}},f)\|_{2,2},
\end{equation*}
which, together with the geometry property of integrals (see \cite[formula (55)]{Josien-Otto22}), leads to $\eqref{f:3.2}$.

%Then, for any $s'>2$ and $1<q<\infty$, we turn to establish
%\begin{equation}\label{f:3.4}
% \|\omega_{T'}(\frac{u}{\sqrt{T'}},\nabla u)\|_{2,2,q}
% \lesssim \|\omega_{T'}(\frac{\tilde{g}}{\sqrt{T'}},f)\|_{2,s',q},
%\end{equation}
%which, similar to \emph{Step 3} in \cite[pp.46]{Josien-Otto22}, follows from a complex interpolation (between $\eqref{f:3.2}$ and $\eqref{f:3.3}$).
To see $\eqref{f:3.3}$, by the same token,
it is reduced to show
\begin{equation}\label{f:3.1}
\int_{\Omega}
\Big\langle\big(\dashint_{U_1(x)}
|(\frac{u}{\sqrt{T}},\nabla u)|^{q'}\big)^{\frac{2}{q'}}\Big\rangle
\lesssim \int_{\Omega}
\Big\langle\big(\dashint_{U_1}
|(\frac{g}{\sqrt{T}},f)|^{q'}
\big)^{\frac{s}{q'}}\Big\rangle^{\frac{2}{s}}.
\end{equation}

%Finally, based upon Lemma $\ref{lemma:3.1}$ and the estimates
%$\eqref{f:3.2}$ and $\eqref{f:3.4}$,
%using the same arguments \cite[pp.46-47]{Josien-Otto22}
%developed for the suboptimal estimates (which is independent of PDEs as mentioned before), we obtain that
%for any $p,q\in(1,\infty)$ and $2<r'<r$, and there holds
%\begin{equation*}
%\|(\frac{u}{\sqrt{T'}},\nabla u)\|_{p,r',q}
%\lesssim \sqrt{T'}^{d}
%\|(\frac{g}{\sqrt{T'}},f)\|_{p,r,q},
%\end{equation*}
%which implies the desired estimate $\eqref{pri:2.1}$ by merely noting
%the zero-extension.

\medskip
\noindent
\textbf{Step 4.} Show the estimate $\eqref{f:3.0}$. Let $v\in H_0^1(\Omega)$ be a test function, and it follows from
the equation $\eqref{pde:3.2x}$ that
\begin{equation}\label{f:3.5}
\int_{\Omega}\frac{1}{T}uv
+ a\nabla u\cdot\nabla v
= \int_{\Omega}\frac{1}{T}gv - f\cdot\nabla v.
\end{equation}
One may prefer $v=u\omega_{T}^2$ in $\eqref{f:3.5}$, and there holds
\begin{equation*}
\int_{\Omega}\frac{1}{T}u^2\omega_{T}^2
+ \lambda
\int_{\Omega}|\nabla u|^2 \omega_{T}^2
\lesssim_{\lambda,d} \int_{\Omega} \big(\frac{1}{T}g^2+f^2\big)\omega_{T}^2
+ \frac{1}{C\sqrt{T}}\int_{\Omega}|u||\nabla u|\omega_{T}^2.
\end{equation*}
By using Young's inequality again (we also employ the condition $C\gg 1$), we have that
\begin{equation*}
 \int_{\Omega}\big(\frac{1}{T}u^2 + |\nabla u|^2\big)\omega_{T}^2
 \lesssim_{\lambda,d}
 \int_{\Omega}\big(\frac{1}{T}g^2 + |f|^2\big)\omega_{T}^2,
\end{equation*}
and then taking $\langle\cdot\rangle$ on the both sides above, we consequently derive the stated estimate $\eqref{f:3.0}$.

\medskip
\noindent
\textbf{Step 5.} Show the estimate $\eqref{f:3.1}$.
It suffices to consider the case $q'\in(1,2]$ since we can
make a few modification in the proof below to cover the case of $q'\in(2,\infty)$.
By a duality argument, it suffices to establish for any $q\geq 2$ and $s'<2$ that
\begin{equation}\label{f:3.6}
\int_{\Omega}
\Big\langle\big(\dashint_{U_1(x)}
|(\frac{u}{\sqrt{T}},\nabla u)|^{q}\big)^{\frac{s'}{q}}\Big\rangle^{\frac{2}{s'}}dx
\lesssim \int_{\Omega}
\Big\langle\big(\dashint_{U_1(x)}
|(\frac{g}{\sqrt{T}},f)
|^{q}\big)^{\frac{2}{q}}\Big\rangle dx.
\end{equation}
By classical $L^q$ estimates with $q\geq 2$, for any $x\in\tilde{\Omega}$, we have
\begin{equation}\label{f:2.17}
\|(\frac{u}{\sqrt{T}},\nabla u)\|_{L^q(U_1(x))}
\lesssim \|a\|_{C^{0,\sigma}(B_1(x))}^{\frac{1}{\sigma}(\frac{d}{2}-\frac{d}{q})}
\|(\frac{u}{\sqrt{T}},\nabla u)\|_{L^2(B_4(x))}
+\|(\frac{g}{\sqrt{T}},f)\|_{L^q(B_4(x))},
\end{equation}
where we denote the zero-extension of
$u,f,g$ by themselves in the right-hand side above.
Then, taking $\langle|\cdot|^{s'}\rangle^{\frac{1}{s'}}$ and $\|\cdot\|_{L^2(\tilde{\Omega})}$ in the order, we derive that
\begin{equation}\label{f:2.16}
\begin{aligned}
\int_{\Omega}
\Big\langle\big(\dashint_{U_1(x)}
|(\frac{u}{\sqrt{T}},\nabla u)|^{q}\big)^{\frac{s'}{q}}\Big\rangle^{\frac{2}{s'}}
&\lesssim \int_{\mathbb{R}^d}
\Big\langle \dashint_{B_4(x)}
|(\frac{u}{\sqrt{T}},\nabla u)|^{2}\Big\rangle
+ \int_{\mathbb{R}^d}
\Big\langle\big(\dashint_{B_4(x)}
|(\frac{g}{\sqrt{T}},
f)|^{q}\big)^{\frac{2}{q}}\Big\rangle\\
&\lesssim^{\eqref{w-1}}
\|(\frac{u}{\sqrt{T}},\nabla u)\|_{2,2,2}^2
+ \|(\frac{g}{\sqrt{T}},f)\|_{2,2,q}^2.
\end{aligned}
\end{equation}
%where we employ the same computation given for \cite[formula (136)]{Josien-Otto22} in the last inequality.
This together with
\begin{equation*}
\begin{aligned}
&\|(\frac{u}{\sqrt{T}},\nabla u)\|_{2,2,2}
= \|(\frac{u}{\sqrt{T}},\nabla u)\|_{2,2}\\
&\lesssim \Big(\int_{\tilde{\Omega}}
\big\langle|(\frac{u}{\sqrt{T}},\nabla u)|^2
\big\rangle\Big)^{\frac{1}{2}}
\lesssim
\Big(\int_{\tilde{\Omega}}
\big\langle|(\frac{g}{\sqrt{T}},f)|^2
\big\rangle\Big)^{\frac{1}{2}}
= \|(\frac{g}{\sqrt{T}},f)\|_{2,2,2}
\leq \|(\frac{g}{\sqrt{T}},f)\|_{2,2,q}
\end{aligned}
\end{equation*}
leads to the stated estimate $\eqref{f:3.6}$ for the case of $q'\in(1,2]$. Now, we turn to the case
$q'\in(2,\infty)$. Let $2<s_0<s$ and $0<s_0-2\ll 1$. In view of
$\eqref{f:2.17}$, we modify the estimate $\eqref{f:2.16}$ as follows:
\begin{equation*}
\begin{aligned}
\int_{\Omega}
\Big\langle\big(\dashint_{U_1(x)}
|(\frac{u}{\sqrt{T}},\nabla u)|^{q'}\big)^{\frac{2}{q'}}\Big\rangle
&\lesssim \int_{\Omega}
\Big\langle \big(\dashint_{U_4(x)}
|(\frac{u}{\sqrt{T}},\nabla u)|^{2}\big)^{\frac{s_0}{2}}\Big\rangle^{\frac{2}{s_0}}
+ \int_{\Omega}
\Big\langle\big(\dashint_{U_4(x)}
|(\frac{g}{\sqrt{T}},
f)|^{q'}\big)^{\frac{s}{q'}}\Big\rangle^{\frac{2}{s}}\\
&\lesssim^{\eqref{w-1},\eqref{pri:2.3}}
\int_{\Omega}
\Big\langle\big(\dashint_{U_1(x)}
|(\frac{g}{\sqrt{T}},
f)|^{q'}\big)^{\frac{s}{q'}}\Big\rangle^{\frac{2}{s}},
\end{aligned}
\end{equation*}
where we also employ H\"older's inequality for the second inequality above.

\medskip
\noindent
\textbf{Step 6.}
Reduction of the estimate $\eqref{pri:2.1-w}$. Let $1<r'<r<\infty$,
$q,q'\in(1,\infty)$, and $1<p_0<p_*<p<\infty$.
Here, we adapt the strategy
analogous to $\eqref{pri:2.1}$, and  therefore,
for any $\omega\in A_p$, it suffices to establish
the following estimate:
\begin{equation}\label{f:2.2}
\begin{aligned}
\|\omega_{\frac{T}{4}}\big(\frac{u}{\sqrt{T}},\nabla u\big);(\omega)_{B_1}^{\frac{1}{p}}\|_{p,r',q}
\lesssim T^{\frac{3d}{4}}
\|\omega_{4T}\big(\frac{g}{\sqrt{T}},f\big);
\omega^{\frac{1}{p}}\|_{p,r',q}.
\end{aligned}
\end{equation}
Since the above estimate is shift-invariant and,
similar to $\eqref{f:2.1*}$, one can rewrite $\eqref{f:2.2}$ as
\begin{equation*}
\begin{aligned}
\|\omega_{\frac{T}{4}}(\cdot-z)\big(\frac{u}{\sqrt{T}},\nabla u\big);(\omega)_{B_1}^{\frac{1}{p}}\|_{p,r',q}
\lesssim T^{\frac{3d}{4}}
\|\omega_{4T}(\cdot-z)\big(\frac{g}{\sqrt{T}},f\big);
\omega^{\frac{1}{p}}\|_{p,r',q},
\end{aligned}
\end{equation*}
then taking $L^p$-norm on the both sides above with respect to
$z\in\mathbb{R}^d$, there holds
\begin{equation*}
\|\big(\frac{u}{\sqrt{T}},\nabla u\big);(\omega)_{B_1}^{\frac{1}{p}}\|_{p,r',q}
\lesssim T^{\frac{3d}{4}}
\|\big(\frac{g}{\sqrt{T}},f\big);
\omega^{\frac{1}{p}}\|_{p,r',q},
\end{equation*}
which implies the stated estimate $\eqref{pri:2.1-w}$.

\medskip
\noindent
\textbf{Step 7.}  Arguments for the estimate $\eqref{f:2.2}$.
Let $p_*\in(1,p)$ satisfy $0<p_*-1\ll 1$, and $\theta\in(0,1)$ satisfies
$0<\theta\ll 1$ (to be fixed later on), and
it suffices to establish
\begin{equation}\label{f:2.18}
\|\omega_{T}\big(\frac{u}{\sqrt{T}},\nabla u\big)\|_{\infty,r',q}
\lesssim
\sqrt{T}^{d(1-\theta)}
\|\omega_{T}\big(\frac{g}{\sqrt{T}},f\big)\|_{p_*,r,q}.
\lesssim
\sqrt{T}^{d}
\|\omega_{T}\big(\frac{g}{\sqrt{T}},f\big)\|_{p_*,r,q}.
\end{equation}
By H\"older's inequality, we have
\begin{equation}\label{f:2.3}
\begin{aligned}
&\|\omega_{\frac{T}{4}}\big(\frac{u}{\sqrt{T}},\nabla u\big);(\omega)_{B_1}^{\frac{1}{p}}\|_{p,r',q}
\leq \Big(\int_{\mathbb{R}^d}\omega_{T}^p(x) (\omega)_{B_1(x)}dx\Big)^{\frac{1}{p}}
\|\omega_{T}\big(\frac{u}{\sqrt{T}},\nabla u\big)\|_{\infty,r',q}\\
&\lesssim^{\eqref{f:2.18}} \sqrt{T}^{d}
 \Big(\int_{\mathbb{R}^d}\omega_{T}^p(x) (\omega)_{B_1(x)}dx\Big)^{\frac{1}{p}}
\|\omega_{T}\big(\frac{g}{\sqrt{T}},f\big)\|_{p_*,r,q}\\
&\lesssim \sqrt{T}^{d}
 \Big(\int_{\mathbb{R}^d}\omega_{T}^p(x) (\omega)_{B_1(x)}dx\Big)^{\frac{1}{p}}
\Big(\int_{\mathbb{R}^d}\omega_{4T}^{\frac{pp_*}{p-p_*}} \omega^{-\frac{p_*}{p-p_*}}dx\Big)^{\frac{p-p_*}{pp_*}}
\|\omega_{4T}\big(\frac{g}{\sqrt{T}},f\big);
\omega^{\frac{1}{p}}\|_{p,r,q},
\end{aligned}
\end{equation}
where we note that $\omega_{\frac{T}{4}}=\omega_{T}^2$ and
$\omega_{T}=\omega_{4T}^2$.
%Also, a routine computation leads to
%\begin{equation*}
%\|\omega_{4T'}\big(\frac{\tilde{g}}{\sqrt{T'}},\tilde{f}\big)
%;\omega^{\frac{1}{p}}\|_{p,r,q}
%\lesssim_d
%\bigg(\int_{\mathbb{R}^d}
%\Big\langle\big(\dashint_{B_2(x)}
%|\omega_{4T'}(\frac{\tilde{g}}{\sqrt{T'}},\tilde{f})|^q\big)^{\frac{r}{q}}
%\Big\rangle^{\frac{p}{r}}\omega\bigg)^{\frac{1}{p}}
%\lesssim_d
%\|\omega_{4T'}\big(\frac{\tilde{g}}{\sqrt{T'}},\tilde{f}\big);
%(\omega)_{B_1}^{\frac{1}{p}}\|_{p,r,q},
%\end{equation*}
%where we also use \cite[Lemma 3.10]{Wang-Xu22} in the deterministic case
%for the second inequality.
%Thus,
%\begin{equation}
%\begin{aligned}
%&\|\omega_{\frac{T'}{4}}\big(\frac{u}{\sqrt{T'}},\nabla u\big);(\omega)_{B_1}^{\frac{1}{p}}\|_{p,r',q}\\
%&\lesssim \sqrt{T}^{\frac{d}{2}}
% \Big(\int_{\mathbb{R}^d}\omega_{T'}^p(x) (\omega)_{B_1(x)}dx\Big)^{\frac{1}{p}}
%\Big(\int_{\mathbb{R}^d}\omega_{4T'}^{\frac{pp_*}{p-p_*}} \omega^{-\frac{p_*}{p-p_*}}dx\Big)^{\frac{p-p_*}{pp_*}}
%\|\omega_{4T'}\big(\frac{\tilde{g}}{\sqrt{T'}},\tilde{f}\big);
%(\omega)_{B_1}^{\frac{1}{p}}\|_{p,r,q}.
%\end{aligned}
%\end{equation}
Then, we claim that
\begin{equation}\label{f:2.4}
\Big(\int_{\mathbb{R}^d}\omega_{T}^p(x) (\omega)_{B_1(x)}dx\Big)^{\frac{1}{p}}
\Big(\int_{\mathbb{R}^d}\omega_{4T}^{\frac{pp_*}{p-p_*}} \omega^{-\frac{p_*}{p-p_*}}dx\Big)^{\frac{p-p_*}{pp_*}}
\lesssim_{[\omega]_{A_p}}\sqrt{T}^{\frac{d}{2}}.
\end{equation}
Plugging the estimate $\eqref{f:2.4}$ back into
$\eqref{f:2.3}$, we have the stated estimate $\eqref{f:2.2}$.

\medskip
\noindent
\textbf{Step 8.} Arguments for the first inequality in the estimate $\eqref{f:2.18}$.
Adapting the same strategy similar to Steps 2 and 3, we first obtain
\begin{equation}\label{f:2.19}
\|\omega_{T}\big(\frac{u}{\sqrt{T}},\nabla u\big)\|_{\infty,r',q}
 \lesssim^{\eqref{pri:3.1-a}} \|\omega_{T}\big(\frac{u}{\sqrt{T}},\nabla u\big)\|_{p_*,r',q}.
\end{equation}
Once we established
\begin{equation}\label{f:2.20}
\|\omega_{T}^{\theta}\big(\frac{u}{\sqrt{T}},\nabla u\big)\|_{p_*,r',q}
 \lesssim \sqrt{T}^{d(1-\theta)}\|\omega_{T}^{\theta}\big(\frac{u}{\sqrt{T}},\nabla u\big)\|_{p_*,r,q},
\end{equation}
modifying the constant $C$ in the definition of $\omega_{T}$, we would derive that
\begin{equation*}
\|\omega_{T}\big(\frac{u}{\sqrt{T}},\nabla u\big)\|_{p_*,r',q}
\lesssim
\sqrt{T}^{d(1-\theta)}
\|\omega_{T}\big(\frac{u}{\sqrt{T}},\nabla u\big)\|_{p_*,r,q},
\end{equation*}
and inserting this back into $\eqref{f:2.19}$ leads to
the stated estimate $\eqref{f:2.18}$. Now, we turn to study the estimate
$\eqref{f:2.20}$, which follows from a complex interpolation argument between
$\eqref{f:3.2}$ and
\begin{equation}\label{f:3.3*}
 \|(\frac{u}{\sqrt{T}},\nabla u)\|_{p_0,s',q'}
 \lesssim \sqrt{T}^{d}\|(\frac{g}{\sqrt{T}},f)\|_{p_0,s,q'},
\end{equation}
where $1<s'<s<\infty$ and $q'\in(1,\infty)$,
with $\theta\in(0,1)$ satisfying
\begin{equation*}
\frac{1}{p_*}=\frac{\theta}{2}+\frac{1-\theta}{p_0};\quad
\frac{1}{r'}=\frac{\theta}{2} + \frac{1-\theta}{s'};\quad
\frac{1}{r}=\frac{\theta}{2} + \frac{1-\theta}{s};\quad
\frac{1}{q}=\frac{\theta}{2} + \frac{1-\theta}{q'}.
\end{equation*}
By zero extension of $u$, $g$, and $f$,
the desired estimate $\eqref{f:3.3*}$ follows from $\eqref{pri:2.1}$.

\medskip
\noindent
\textbf{Step 9.} Arguments for the estimate $\eqref{f:2.4}$.
For any $R>0$ (to be fixed later on), set $B_R^c:=\mathbb{R}^d\setminus B_R$, and we make the annular decomposition as follows:
$B_R^{c}=\cup_{k=1}^{\infty} B_{2^kR}\setminus B_{2^{k-1}R}$.
Let $\beta,\beta'\in(1,\infty)$ satisfy $1/\beta+1/\beta'=1$ and
$0<\beta'-1\ll 1$.
Then one can first obtain
\begin{equation*}
\begin{aligned}
\int_{\mathbb{R}^d}\omega_T^p(x) (\omega)_{B_1(x)}dx
&= \int_{B_R}\omega_T^p(x) (\omega)_{B_1(x)}dx
+ \int_{B_R^{c}}\omega_T^p(x) (\omega)_{B_1(x)}dx\\
&\lesssim_d |B_R|\dashint_{B_R}(\omega)_{B_1(x)}dx
+ \sum_{k=1}^{\infty}(2^kR)^d
\Big(\dashint_{B_{2^kR}\setminus B_{2^{k-1}R}}
\omega_T^{p\beta}\Big)^{\frac{1}{\beta}}
\Big(\dashint_{B_{2^kR}}(\omega)_{B_1}^{\beta'}\Big)^{\frac{1}{\beta'}}\\
&\lesssim |B_R|\dashint_{B_{2R}}\omega
+ \sum_{k=1}^{\infty}(2^kR)^d
\Big(\dashint_{B_{2^kR}\setminus B_{2^{k-1}R}}
\omega_T^{p\beta}\Big)^{\frac{1}{\beta}}
\Big(\dashint_{B_{2^{k+1}R}}\omega^{\beta'}\Big)^{\frac{1}{\beta'}}
=:I,
\end{aligned}
\end{equation*}
where we also use the geometry property of integrals. In view of
Lemma $\ref{weight}$, we derive that
\begin{equation}\label{f:2.5}
\begin{aligned}
I
& \lesssim_{[\omega]_{A_p},d}^{\eqref{pri:R-1}} |B_R|\dashint_{B_{2R}}\omega
+ \sum_{k=1}^{\infty}(2^kR)^d
\Big(\dashint_{B_{2^kR}\setminus B_{2^{k-1}R}}
\omega_T^{p\beta}\Big)^{\frac{1}{\beta}}
\Big(\dashint_{B_{2^{k+1}R}}\omega \Big)\\
& \lesssim_{[\omega]_{A_p},d}^{\eqref{f:18}}
\Big(\dashint_{B_{R}}\omega\Big) \bigg\{R^d
+\sum_{k=1}^{\infty}(2^kR)^{d(1-\frac{1}{\beta})}
\Big(\int_{B_{2^{k-1}R}^c}\omega_T^{p\beta}\Big)^{\frac{1}{\beta}}
2^{kp}
\bigg\}.
\end{aligned}
\end{equation}
On the other hand, a routine computation leads to
\begin{equation}\label{f:2.6}
\Big(\int_{B_{2^{k-1}R}^c}\omega_T^{p\beta}\Big)^{\frac{1}{\beta}}
\lesssim  \sqrt{T}^{\frac{d}{\beta}}e^{-\frac{c_1 2^{k}R}{\sqrt{T}}}
\lesssim  \sqrt{T}^{\frac{d}{\beta}} \big(\frac{2^kR}{\sqrt{T}}\big)^{-\gamma},
\end{equation}
where $\gamma>1$ is arbitrary to be fixed later on.
Plugging the estimate $\eqref{f:2.6}$ back into $\eqref{f:2.5}$, we have
\begin{equation}\label{f:2.9}
\begin{aligned}
\int_{\mathbb{R}^d}\omega_T^p(x) (\omega)_{B_1(x)}dx
&\lesssim_{[\omega]_{A_p},d}
\Big(\dashint_{B_{R}}\omega\Big) \bigg\{R^d
+\sum_{k=1}^{\infty}(2^kR)^{d(1-\frac{1}{\beta})}
2^{kp}\sqrt{T}^{\frac{d}{\beta}}\big(\frac{2^kR}{\sqrt{T}}\big)^{-\gamma}
\bigg\}\\
&\lesssim
\Big(\dashint_{B_{R}}\omega\Big) \bigg\{R^d
+ c_0 R^{d-\frac{d}{\beta}-\gamma}\sqrt{T}^{\frac{d}{\beta}+\gamma}
\bigg\}\\
&\lesssim \Big(\dashint_{B_{R}}\omega\Big) \sqrt{T}^d,
\end{aligned}
\end{equation}
where we can choose $\gamma>d-\frac{d}{\beta}+p$, and $R:=\tilde{c}\sqrt{T}$ with $\tilde{c}:=c_0^{\frac{\beta}{d+\beta\gamma}}$.

By the same token, we have
\begin{equation*}
\begin{aligned}
&\int_{\mathbb{R}^d}\omega_{4T}^{\frac{pp_*}{p-p_*}} \omega^{-\frac{p_*}{p-p_*}}dx
\lesssim \Big(\dashint_{B_R}\omega^{-\frac{p_*}{p-p_*}}\Big)|B_R|
+ \sum_{k=1}^{\infty}(2^kR)^d
\Big(\dashint_{B_{2^kR}\setminus B_{2^{k-1}R}}\omega_{4T}^{\frac{pp_*\beta}{p-p_*}}\Big)^{\frac{1}{\beta}}
\Big(\dashint_{B_{2^kR}}
\omega^{-\frac{p_*\beta'}{p-p_*}}\Big)^{\frac{1}{\beta'}}.
\end{aligned}
\end{equation*}
By Lemma $\ref{weight}$, since $\omega\in A_p$ one can infer
from $\eqref{f:19-1}$ that $\omega\in A_{\frac{p}{p_*}}$.
Let $\tilde{\omega}:=\omega^{-\frac{p_*}{p-p_*}}$. Moreover,  $\tilde{\omega}\in A_{(\frac{p}{p_*})'}$ with $1=1/(\frac{p}{p_*})'
+ 1/(\frac{p}{p_*})$. Thus, we have
\begin{equation}\label{f:2.7}
\begin{aligned}
&\int_{\mathbb{R}^d}\omega_{4T}^{\frac{pp_*}{p-p_*}} \omega^{-\frac{p_*}{p-p_*}}dx\\
&\lesssim^{\eqref{pri:R-1}} \Big(\dashint_{B_R}\omega^{-\frac{p_*}{p-p_*}}\Big)|B_R|
+ \sum_{k=1}^{\infty}(2^kR)^d
\Big(\dashint_{B_{2^kR}\setminus B_{2^{k-1}R}}\omega_{4T}^{\frac{pp_*\beta}{p-p_*}}\Big)^{\frac{1}{\beta}}
\Big(\dashint_{B_{2^kR}}\omega^{-\frac{p_*}{p-p_*}}\Big).
\end{aligned}
\end{equation}
On account of $\eqref{f:18}$, we have
\begin{equation}\label{f:2.8}
 \tilde{\omega}(B_{2^{k}R})
 \lesssim \Big(\frac{|B_R|}{|B_{2^{k}R}|}\Big)^{-\frac{p}{p-p_*}}\tilde{\omega}(B_R).
\end{equation}
Combining the estimates $\eqref{f:2.7}$ and $\eqref{f:2.8}$, we have
\begin{equation}\label{f:2.10}
\begin{aligned}
\int_{\mathbb{R}^d}\omega_{4T}^{\frac{pp_*}{p-p_*}} \omega^{-\frac{p_*}{p-p_*}}dx
&\lesssim \Big(\dashint_{B_R}\omega^{-\frac{p_*}{p-p_*}}\Big)
\bigg\{|B_R|+\sum_{k=1}^{\infty}
(2^kR)^{d(1-\frac{1}{\beta})}\Big(\int_{B_{2^{k-1}R}^c}
\omega_{4T}^{\frac{pp_*\beta}{p-p_*}}\Big)^{\frac{1}{\beta}}
2^{\frac{kp}{p-p_*}}\bigg\}\\
&\lesssim^{\eqref{f:2.6}}
\Big(\dashint_{B_R}\omega^{-\frac{p_*}{p-p_*}}\Big)
\bigg\{R^d+\sum_{k=1}^{\infty}
(2^kR)^{d(1-\frac{1}{\beta})}
2^{\frac{kp}{p-p_*}}
\sqrt{T}^{\frac{d}{\beta}}\big(\frac{2^kR}{\sqrt{T}}\big)^{-\gamma}\bigg\}\\
&\lesssim \Big(\dashint_{B_R}\omega^{-\frac{2}{p-2}}\Big)
\sqrt{T}^d,
\end{aligned}
\end{equation}
where we choose $\gamma>d-\frac{d}{\beta}+\frac{p}{p-p_*}$ and
$R =\tilde{c}\sqrt{T}$. Hence, combining the estimates $\eqref{f:2.9}$ and
$\eqref{f:2.10}$, as well as $\eqref{f:18}$, one can derive that
\begin{equation*}
\begin{aligned}
\Big(\int_{\mathbb{R}^d}\omega_T^p(x) (\omega)_{B_1(x)}dx\Big)^{\frac{1}{p}}
\Big(\int_{\mathbb{R}^d}\omega_{4T}^{\frac{2p}{p-2}} \omega^{-\frac{2}{p-2}}dx\Big)^{\frac{p-2}{2p}}
\lesssim \bigg[
\Big(\dashint_{B_{\sqrt{T}}}\omega\Big)
\Big(\dashint_{B_{\sqrt{T}}}\omega^{-\frac{p_*}{p-p_*}}
\Big)^{\frac{p-p_*}{p_*}}\bigg]^{\frac{1}{p}}\sqrt{T}^{\frac{d}{2}}
\lesssim_{[\omega]_{A_p}}\sqrt{T}^{\frac{d}{2}},
\end{aligned}
\end{equation*}
which implies the stated estimate $\eqref{f:2.4}$. The whole proof
is complete.
\qed

\subsection{Annealed estimates at small scales}

\begin{proposition}\label{P:6}
Let $\Omega$ be a (bounded) $C^1$ domain and $T\geq 1$.
Suppose that the ensemble $\langle\cdot\rangle$
is stationary and satisfies local smoothness condition $\eqref{a:3}$.
%the spectral gap condition $\eqref{a:2}$,
%as well as
Let $\tilde{u}$ be associated with
$\tilde{f}$ and $\tilde{g}$ by the equations $\eqref{pde:3.2x}$.
Then, for any $1<s,p,q<\infty$ with $\bar{p}>p$, there holds
the following annealed estimate
\begin{equation}\label{pri:4.9x}
\begin{aligned}
\int_{\Omega} \big\langle|(\nabla u,\frac{u}{\sqrt{T'}} )|^{p}\big\rangle^{\frac{q}{p}}
&\lesssim_{\lambda,d,\partial\Omega,p,q}
\int_{\Omega}
\big\langle|(f,\frac{g}{\sqrt{T'}})|^{\bar{p}}
\big\rangle^{\frac{q}{\bar{p}}}
+ \int_{\Omega}
\Big\langle\big(\dashint_{U_1(z)}
|(\nabla u,\frac{u}{\sqrt{T'}})|^s
\big)^{\frac{\bar{p}}{s}}\Big\rangle^{\frac{q}{\bar{p}}}dz.
\end{aligned}
\end{equation}
Moreover, for any $\omega\in A_q$, we have
\begin{equation}\label{pri:4.10x}
\begin{aligned}
&\int_{\Omega} \big\langle|(\nabla u,\frac{u}{\sqrt{T'}} )|^{p}\big\rangle^{\frac{q}{p}}\omega\\
&\lesssim_{\lambda,d,M_0,p,q}
\int_{\Omega}
\big\langle|(f,\frac{g}{\sqrt{T'}})|^{\bar{p}}
\big\rangle^{\frac{q}{\bar{p}}}\omega
+ \int_{\Omega}
\Big\langle\Big(\dashint_{U_1(z)}
|(\nabla u,\frac{u}{\sqrt{T'}})|^s\Big)^{\frac{\bar{p}}{s}}
\Big\rangle^{\frac{q}{\bar{p}}}\Big(\dashint_{B_1(z)}\omega\Big)dz.
\end{aligned}
\end{equation}
\end{proposition}

\begin{lemma}\label{lemma:4.1}
Let $\Omega\subset\mathbb{R}^d$ with $d\geq 2$ be a (bounded) Lipschitz domain. Suppose that the ensemble $\langle\cdot\rangle$
is stationary and satisfies local smoothness condition $\eqref{a:3}$.
Then, there exists a stationary random field
$\rho_*$ with $\langle\rho_*^{-\beta}\rangle^{1/\beta}\lesssim_{\lambda_2,\beta} 1$
for any $\beta<\infty$, such that, for any random variable $F$ and
$y\in\Omega$, there holds
\begin{equation}\label{pri:4.11x}
 \langle|F|^{p}\rangle^{\frac{q}{p}}
\lesssim_{d,M_0,\lambda_2,p,p_1} \int_{\Omega} \big\langle
I_{B_{\frac{\rho_*(x)}{2}}(x)}(y)|F|^{p_1}\big\rangle^{\frac{q}{p_1}} dx
\end{equation}
for any $1<p,q<\infty$ with $p_1>p$.
\end{lemma}

\begin{proof}
%\noindent
%\textbf{The arguments for the estimate $\eqref{pri:4.11x}$}.
The idea of the proof is due to Prof. Otto's arguments in \cite[Lemma A.2]{Josien-Otto22}, and we provide a proof for the sake of completeness.
The proof will be simply divided into two steps.
\medskip

\noindent
\textbf{Step 1.} Define random small scales $\rho_*(x)$ as follows:
\begin{equation}\label{def:3*}
\rho_*(x):=\sup\big\{\rho\leq 1: \sup_{y,y'\in B_{\rho}(x)}
|a(y)- a(y')|\leq \delta \big\},
\end{equation}
where $\delta>0$ is arbitrarily small parameter.
By the definition, we have $0<\rho_*(x)\leq 1$ $\langle\cdot\rangle$-a.s.;
Moreover, $\rho_*$ is stationary and, we claim the following moment estimates
\begin{equation}\label{pri:15}
 \langle\rho_*^{-\beta}\rangle^{\frac{1}{\beta}}
 \lesssim 1 \qquad
 \forall \beta\in[1,\infty).
\end{equation}
To see this, we may define
\begin{equation}\label{def:3}
\tilde{\rho}_*:=
\frac{1}{2}\big(\delta/[a]_{C^{0,\sigma}(B_1)}\big)^{\frac{1}{\sigma}}\wedge 1.
\end{equation}
It is not hard to see that $\tilde{\rho}_*\leq \rho_*$ since, for any $y,y'\in B_{\tilde{\rho}_*}(0)$, we have
\begin{equation*}
|a(y)- a(y')|
\leq [a]_{C^{0,\sigma}(B_1)}(2r)^\sigma
\leq \delta.
\end{equation*}
Thus, the moments estimate of $\rho_*$ is reduced to that of $\tilde{\rho}_*$, which would be bounded by $\eqref{def:3}$.

\medskip

\noindent
\textbf{Step 2.} Establish the estimate $\eqref{pri:4.11x}$.
Let $p<p_1<p_2$.
The key ingredient is that for any fixed $y\in\Omega$,
there holds that if $x\in \frac{1}{3}B_{\rho_*(y)}(y)$,
we can infer $B_{\frac{2}{3}\rho_*(y)}(x)
\subset B_{\rho_*(y)}(y)$.
Moreover, there hold $\rho_*(x)\geq \frac{2}{3}\rho_*(y)$ and
$y\in B_{\frac{1}{2}\rho_*(x)}(x)$. This concludes that
\begin{equation}\label{f:4.23x}
 \int_{\Omega} I_{B_{\frac{1}{2}\rho_*(x)}(y)}dx
 \geq \int_{\Omega} I_{B_{\frac{7}{24}\rho_*(y)}(y)} dx
 \gtrsim_{d,M_0}  \rho_*(y)^d,
\end{equation}
where the constant is also dependent of the character of $\Omega$.
By using H\"older's inequality, the fact that $\rho_*\leq 1$
$\langle\cdot\rangle$-a.s., and Minkowski's inequality, we have
\begin{equation*}\label{}
\begin{aligned}
\Big( \int_{\Omega} \big\langle
 I_{B_{\frac{1}{2}\rho_*(x)}(y)}|F|^{p_1}\big\rangle^{\frac{q}{p_1}}dx
\Big)^{\frac{1}{q}}
&\gtrsim_d
\int_{\Omega} \big\langle
 I_{B_{\frac{1}{2}\rho_*(x)}(y)}|F|^{p_1}\big\rangle^{\frac{1}{p_1}}dx\\
&\geq
\Big\langle \Big(\int_{\Omega} I_{B_{\frac{1}{2}\rho_*(x)}(y)}|F|dx
\Big)^{p_1}\Big\rangle^{\frac{1}{p_1}}
=\Big\langle \Big(\int_{\Omega} I_{B_{\frac{1}{2}\rho_*(x)}(y)}dx
\Big)^{p_1}|F|^{p_1}
\Big\rangle^{\frac{1}{p_1}}.
\end{aligned}
\end{equation*}
This together with $\eqref{f:4.23x}$ leads to
\begin{equation}\label{f:4.24x}
\Big( \int_{\Omega} \big\langle
 I_{B_{\frac{1}{2}\rho_*(x)}(y)}|F|^{p_1}\big\rangle^{\frac{q}{p_1}}dx
\Big)^{\frac{1}{q}}
\gtrsim \big\langle \rho_*(y)^{p_1d} |F|^{p_1} \big\rangle^{\frac{1}{p_1}}.
\end{equation}
On the other hand, by setting $1/p=1/p_1+1/p_2$, we have
\begin{equation}\label{f:4.25x}
\big\langle |F|^{p} \big\rangle^{\frac{1}{p}}
\leq  \big\langle \rho_*(y)^{p_1d}  |F|^{p_1} \big\rangle^{\frac{1}{p_1}}
\big\langle \rho_*(y)^{-p_2d}\big\rangle^{\frac{1}{p_2}}
\lesssim^{\eqref{pri:15}}  \big\langle \rho_*(y)^{p_1d}  |F|^{p_1} \big\rangle^{\frac{1}{p_1}}.
\end{equation}

As a result, the stated estimate $\eqref{pri:4.11x}$ follows from
the estimates $\eqref{f:4.24x}$ and $\eqref{f:4.25x}$.
\end{proof}

\medskip

\noindent
\textbf{Proof of Proposition $\ref{P:6}$.}
The main idea of the proof may be found in \cite[Lemma A.2]{Josien-Otto22},
and we provide a proof for the reader's convenience.
The whole proof is divided into four steps. Let $1<p<p_1<p_2\leq \bar{p}<\infty$, and $1<q<\infty$.

\medskip
\noindent
\textbf{Step 1.} Construct an auxiliary equation. Recalling the definition
of $\rho_*$ in $\eqref{def:3*}$, for any $x\in\Omega$ we set $B_*(x):=B_{\rho_*(x)}(x)$, $\tilde{f}_x:=fI_{B_1(x)}$ and $\tilde{g}_x:=gI_{B_1(x)}$. Then, we define a new coefficient as follows:
\begin{equation}\label{eq:4.1}
  \tilde{a}_x(y):=\left\{\begin{aligned}
  &a(y)\quad \text{if}\quad y\in B_*(x);\\
  &a(x)\quad \text{otherwise}.
  \end{aligned}\right.
\end{equation}
Now, we proceed to construct the auxiliary equations:
\begin{equation}\label{pde:4.2}
\left\{\begin{aligned}
\frac{1}{T}u_x- \nabla \cdot \tilde{a}_x\nabla \tilde{u}_x
 &= \nabla \cdot\tilde{f}_x + \frac{1}{T}\tilde{g}_x
 &\quad &\text{in}\quad \Omega;\\
 \tilde{u}_x &= 0
 &\quad &\text{on}~~ \partial\Omega,
\end{aligned}\right.
\end{equation}
and
\begin{equation}\label{pde:4.3}
  \nabla \cdot a \nabla (\tilde{u}_x - u) = \frac{1}{T}(\tilde{u}_x - u)
\quad \text{in}\quad B_*(x).
\end{equation}

Concerned with the auxiliary equation $\eqref{pde:4.2}$,
in view of the definitions of the new coefficient
$\tilde{a}_x$ in $\eqref{eq:4.1}$ and the random small-scale radius
$\rho_*(x)$ in $\eqref{def:3*}$, for any given small parameter $\delta>0$,
it is not hard to see that
\begin{equation}\label{perturbation}
   \|\tilde{a}_x - a(x)\|_{L^\infty(\Omega)}\leq \delta.
\end{equation}
Based upon this perturbation $\eqref{perturbation}$, we claim that
\begin{subequations}
\begin{align}
 &\int_{\Omega} |(\nabla \tilde{u}_x,\frac{\tilde{u}_x}{\sqrt{T}})|^q
 \lesssim \int_{U_1(x)}
 |(f,\frac{g}{\sqrt{T}})|^q;\label{pri:4.13x-1}\\
 & \int_{\Omega}
 \big\langle|(\nabla \tilde{u}_x,\frac{\tilde{u}_x}{\sqrt{T}})|^p\big\rangle^{\frac{q}{p}}
 \lesssim \int_{U_1(x)}
 \big\langle|(f,\frac{g}{\sqrt{T}})|^p\big\rangle^{\frac{q}{p}}. \label{pri:4.13x-2}\\
 & \int_{\Omega}
\big\langle|(\nabla \tilde{u}_x,\frac{\tilde{u}_x}{\sqrt{T}})|^p\big\rangle^{\frac{q}{p}}\omega
\lesssim
\int_{U_1(x)}
 \big\langle|(f,\frac{g}{\sqrt{T}})|^p\big\rangle^{\frac{q}{p}}
 \omega. \label{pri:4.13x-3}
\end{align}
\end{subequations}

In terms of the equation $\eqref{pde:4.3}$, we claim that
\begin{subequations}
\begin{align}
&\Big\langle\Big(\dashint_{\frac{1}{2}B_{*}(x)\cap\Omega}
\big|\big(\nabla u-\nabla\tilde{u}_x,\frac{u-\tilde{u}_x}{\sqrt{T}}\big)
\big|^{q}\Big)^{\frac{p}{q}}\Big\rangle^{\frac{1}{p}}
\lesssim
\Big\langle\Big(\dashint_{U_1(x)}
\big|\big(\nabla u,\frac{u}{\sqrt{T}},f,\frac{g}{\sqrt{T}}\big)
\big|^s\Big)^{\frac{p_1}{s}}
\Big\rangle^{\frac{1}{p_1}} \label{pri:4.14x}\\
& \Big\langle\Big(\dashint_{\frac{1}{2}B_{*}(x)\cap\Omega}
\big|\big(\nabla u-\nabla\tilde{u}_x,\frac{u-\tilde{u}_x}{\sqrt{T}}\big)
\big|^{q}\omega\Big)^{\frac{p}{q}}\Big\rangle^{\frac{1}{p}}
\nonumber\\
& \qquad\qquad\qquad\qquad\qquad+
\Big\langle\Big(\dashint_{U_1(x)}
\big|\big(\nabla u,\frac{u}{\sqrt{T}},f,\frac{g}{\sqrt{T}}\big)
\big|^s\Big)^{\frac{p_1}{s}}
\Big\rangle^{\frac{1}{p_1}}\Big(\dashint_{B_1(x)}\omega\Big)^{\frac{1}{q}}
\label{pri:4.14x-1}
\end{align}
\end{subequations}

\medskip

\noindent
\textbf{Step 2.}
Show the estimates $\eqref{pri:4.13x-1}$,
$\eqref{pri:4.13x-2}$, and $\eqref{pri:4.13x-3}$. All of the estimates follow from the same  perturbation argument due to the new defined coefficient $\eqref{eq:4.1}$.
To see this, we rewrite the first line of the equations $\eqref{pde:4.2}$ as follows:
\begin{equation*}
\frac{1}{T}\tilde{u}_x- \nabla\cdot a(x)\nabla \tilde{u}_x
 = \nabla \cdot \big[(\tilde{a}_x-a(x))\nabla\tilde{u}_x+ \tilde{f}_x\big] + \frac{1}{T}\tilde{g}_x
 \quad \text{in}\quad \Omega.
\end{equation*}
Appealing to Lemma $\ref{lemma:ap-2}$ and the perturbation condition $\eqref{perturbation}$, we obtain
\begin{equation*}
\begin{aligned}
&\int_{\Omega} |(\nabla \tilde{u}_x,\frac{\tilde{u}_x}{\sqrt{T}})|^q
 \lesssim_{\lambda,d,M_0,q}^{\eqref{pri:A-1}} \delta^{q}
\int_{\Omega} |\nabla \tilde{u}_x|^q
+ \int_{\Omega}
 |(\tilde{f}_x,\frac{\tilde{g}_x}{\sqrt{T}})|^q;\\
& \int_{\Omega}
 \big\langle|(\nabla \tilde{u}_x,\frac{\tilde{u}_x}{\sqrt{T}})|^p\big\rangle^{\frac{q}{p}}
 \lesssim_{\lambda,d,M_0,p,q}^{\eqref{pri:A-2}}
 \delta^q
 \int_{\Omega}\langle|\nabla\tilde{u}_x|^{p}\rangle^{\frac{q}{p}}
 +\int_{\Omega}
 \big\langle\big|(\tilde{f}_x,\frac{\tilde{g}_x}{\sqrt{T}})\big|^p
 \big\rangle^{\frac{q}{p}};\\
& \int_{\Omega}
 \big\langle|(\nabla \tilde{u}_x,\frac{\tilde{u}_x}{\sqrt{T}})|^p\big\rangle^{\frac{q}{p}}
 \omega
 \lesssim_{\lambda,d,M_0,p,q}^{\eqref{pri:A-3}}
 \delta^q
 \int_{\Omega}\langle|\nabla\tilde{u}_x|^{p}\rangle^{\frac{q}{p}}\omega
 +\int_{\Omega}
 \big\langle\big|(\tilde{f}_x,\frac{\tilde{g}_x}{\sqrt{T}})\big|^p
 \big\rangle^{\frac{q}{p}}\omega.
\end{aligned}
\end{equation*}
Then, by preferring $0<\delta\ll 1$, there holds
\begin{equation*}
\begin{aligned}
&\int_{\Omega} |(\nabla \tilde{u}_x,\frac{\tilde{u}_x}{\sqrt{T}})|^q
 \lesssim_{\lambda,d,M_0,q}
 \int_{\Omega}
 |(\tilde{f}_x,\frac{\tilde{g}_x}{\sqrt{T}})|^q
 \lesssim \int_{U_1(x)}
 |(f,\frac{g}{\sqrt{T}})|^q;\\
& \int_{\Omega}
 \big\langle|(\nabla \tilde{u}_x,\frac{\tilde{u}_x}{\sqrt{T}})|^p\big\rangle^{\frac{q}{p}}
 \lesssim_{\lambda,d,M_0,p,q}
 \int_{\Omega}
 \big\langle\big|(\tilde{f}_x,\frac{\tilde{g}_x}{\sqrt{T}})\big|^p
 \big\rangle^{\frac{q}{p}}
 \lesssim \int_{U_1(x)}
 \big\langle\big|(f,\frac{g}{\sqrt{T}})\big|^p
 \big\rangle^{\frac{q}{p}};\\
&\int_{\Omega}
 \big\langle|(\nabla \tilde{u}_x,\frac{\tilde{u}_x}{\sqrt{T}})|^p\big\rangle^{\frac{q}{p}}\omega
 \lesssim_{\lambda,d,M_0,p,q}
 \int_{\Omega}
 \big\langle\big|(\tilde{f}_x,\frac{\tilde{g}_x}{\sqrt{T}})\big|^p
 \big\rangle^{\frac{q}{p}}\omega
 \lesssim \int_{U_1(x)}
 \big\langle\big|(f,\frac{g}{\sqrt{T}})\big|^p
 \big\rangle^{\frac{q}{p}}\omega,
\end{aligned}
\end{equation*}
which give the stated estimates $\eqref{pri:4.13x-1}$,
$\eqref{pri:4.13x-2}$, and $\eqref{pri:4.13x-3}$.

\medskip
\noindent
\textbf{Step 3.}
Arguments for $\eqref{pri:4.14x}$ and $\eqref{pri:4.14x-1}$. Let $1<s\leq \min\{p,q\}$.
In view of the equation $\eqref{pde:4.3}$, it follows from the reverse H\"older's inequality that
\begin{equation*}
\begin{aligned}
\Big(\dashint_{\frac{1}{2}B_*(x)\cap\Omega}
\big|\big(\nabla u
-\nabla\tilde{u}_x,\frac{u-\tilde{u}_x}{\sqrt{T}}\big)\big|^{q}\Big)^{\frac{1}{q}}
&\lesssim [a]_{C^{0,\eta}(B_{1}(x))}^{\frac{d}{\eta}(\frac{1}{s}-\frac{1}{q})}
\Big(\dashint_{B_*(x)\cap\Omega}
\big|\big(\nabla u
-\nabla\tilde{u}_x,\frac{u-\tilde{u}_x}{\sqrt{T}}\big)\big|^s\Big)^{\frac{1}{s}}\\
&\lesssim
[a]_{C^{0,\eta}(B_{1}(x))}^{\frac{d}{\eta}(\frac{1}{s}-\frac{1}{q})}
\rho_*(x)^{-\frac{d}{s}}
\Big(\dashint_{B_1(x)\cap\Omega}|(\nabla u,\frac{u}{\sqrt{T}},f,\frac{g}{\sqrt{T}})|^s
\Big)^{\frac{1}{s}},
\end{aligned}
\end{equation*}
where we also employ the estimate $\eqref{pri:4.13x-1}$ in the second inequality.
Taking $\langle|\cdot|^{p}\rangle^{\frac{1}{p}}$-norm on
the both sides above, and then together with
$\eqref{a:3}$ and $\eqref{pri:15}$
and using H\"older's inequality, we can
derive the stated estimate
$\eqref{pri:4.14x}$. We now proceed to show $\eqref{pri:4.14x-1}$.
To see so, we introduce the multiplicative random variable $C_*(x)$ which
depends on $[a]_{C^{0,\eta}(B_{1}(x))}$ and $\rho_*^{-1}(x)$ polynomially at most, and there holds
\begin{equation*}
\begin{aligned}
&\Big(\dashint_{\frac{1}{2}B_*(x)\cap\Omega}
\big|\big(\nabla u
-\nabla\tilde{u}_x,\frac{u-\tilde{u}_x}{\sqrt{T}}
\big)\big|^{q}\omega\Big)^{\frac{1}{q}}
\lesssim C_*(x)
\Big(\dashint_{B_*(x)\cap\Omega}
\big|\big(\nabla u
-\nabla\tilde{u}_x,\frac{u-\tilde{u}_x}{\sqrt{T}}\big)\big|^s
\Big)^{\frac{1}{s}}\Big(\dashint_{B_1(x)}\omega\Big)^{\frac{1}{q}}\\
&\lesssim^{\eqref{pri:4.13x-1}} C_*(x)
\Big(\dashint_{U_1(x)}
\big|\big(\nabla u,\frac{u}{\sqrt{T}},f,\frac{g}{\sqrt{T}}\big)\big|^s
\Big)^{\frac{1}{s}}\Big(\dashint_{B_1(x)}\omega\Big)^{\frac{1}{q}}.
\end{aligned}
\end{equation*}
Then, taking $\langle|\cdot|^{p}\rangle^{\frac{1}{p}}$-norm on
the both sides above, we have the stated estimate $\eqref{pri:4.14x-1}$.

\medskip
\noindent
\textbf{Step 4.} Show the estimate $\eqref{pri:4.9x}$.
Appealing to Lemma $\ref{lemma:4.1}$ and taking
$F:=(\nabla u,\frac{u}{\sqrt{T}})(y)$ for any fixed $y\in\Omega$
therein, we have
\begin{equation*}
\begin{aligned}
& \big\langle|(\nabla u,\frac{u}{\sqrt{T}})(y)|^{p}\big\rangle^{\frac{q}{p}} \\
& \lesssim^{\eqref{pri:4.11x}}
 \int_{\Omega} \big\langle
I_{\frac{1}{2}B_{*}(x)}(y)
\big|(\nabla\tilde{u}_x,\frac{\tilde{u}_x}{\sqrt{T}})(y)\big|^{p_1}
\big\rangle^{\frac{q}{p_1}} dx
+
\int_{\Omega} \big\langle
I_{\frac{1}{2}B_*(x)}(y)
\big|(\nabla u-\nabla\tilde{u}_x,\frac{u-\tilde{u}_x}{\sqrt{T}})(y)
\big|^{p_1}
\big\rangle^{\frac{q}{p_1}} dx.
\end{aligned}
\end{equation*}
Then, integrating $y\in\Omega$ on the both sides above,
one can obtain
\begin{equation}\label{f:4.26x}
\begin{aligned}
\int_{\Omega}
\big\langle|(\nabla u,\frac{u}{\sqrt{T}})(y)|^{p}\big\rangle^{\frac{q}{p}} dy
&\lesssim
\int_{\Omega}\int_{\Omega} \big\langle
I_{\frac{1}{2}B_{*}(x)}(y)
\big|(\nabla\tilde{u}_x,\frac{\tilde{u}_x}{\sqrt{T}})(y)\big|^{p_1}
\big\rangle^{\frac{q}{p_1}} dxdy\\
&+\int_{\Omega}\int_{\Omega} \big\langle
I_{\frac{1}{2}B_*(x)}(y)
\big|(\nabla u-\nabla\tilde{u}_x,\frac{u-\tilde{u}_x}{\sqrt{T}})(y)\big|^{p_1}
\big\rangle^{\frac{q}{p_1}} dxdy:=I_1+I_2.
\end{aligned}
\end{equation}
Moreover, by Fubini's theorem, there holds
\begin{equation}\label{f:4.27x}
\begin{aligned}
I_1
&\leq \int_{\Omega}\int_{\Omega} \big\langle
\big|(\nabla\tilde{u}_x,\frac{\tilde{u}_x}{\sqrt{T}})(y)\big|^{p_1}
\big\rangle^{\frac{q}{p_1}} dydx
\lesssim^{\eqref{pri:4.13x-2}}
\int_{\Omega}\int_{\Omega}
\big\langle
|(\tilde{f}_x,\frac{\tilde{g}_x}{\sqrt{T}})(y)|^{p_1}
\big\rangle^{\frac{q}{p_1}} dy dx\\
&\lesssim^{\eqref{g-1*}} \int_{\Omega}\dashint_{U_1(x)}
\big\langle
|(f,\frac{g}{\sqrt{T}})|^{p_1}
\big\rangle^{\frac{q}{p_1}}
\lesssim \int_{\Omega}
\big\langle
|(f,\frac{g}{\sqrt{T}})|^{p_1}
\big\rangle^{\frac{q}{p_1}},
\end{aligned}
\end{equation}
and it follows from Minkowski's inequality that
\begin{equation}\label{f:4.28x}
\begin{aligned}
I_2
&\leq \int_{\Omega}\Big\langle
\Big(\int_{\Omega}
I_{\frac{1}{2}B_*(x)}(y)
\big|(\nabla u-\nabla\tilde{u}_x,\frac{u-\tilde{u}_x}{\sqrt{T}})(y)\big|^{q}
dy \Big)^{\frac{p_1}{q}}\Big\rangle^{\frac{q}{p_1}} dx\\
&\lesssim_d
\int_{\Omega}
\Big\langle\Big(\dashint_{\frac{1}{2}B_*(x)\cap\Omega}
|(\nabla u-\nabla\tilde{u}_x,\frac{u-\tilde{u}_x}{\sqrt{T}})
|^q\Big)^{\frac{p_1}{q}}\Big\rangle^{\frac{q}{p_1}}dx\\
&\lesssim^{\eqref{pri:4.14x}}
\int_{\Omega}
\Big\langle\Big(\dashint_{U_1(x)}
\big|\big(\nabla u,\frac{u}{\sqrt{T}},f,\frac{g}{\sqrt{T}}\big)
\big|^s\Big)^{\frac{p_2}{s}}
\Big\rangle^{\frac{q}{p_2}}dx.
\end{aligned}
\end{equation}
Let $s\leq \min\{p_2,q\}$, and similar to
the argument used for $\eqref{pri:3.1-c}$, it follows from H\"older's inequality,
Minkowski's inequality, and the geometrical property of integrals that
\begin{equation}\label{f:4.29x}
\begin{aligned}
\int_{\Omega}
\Big\langle\big(\dashint_{U_1(x)}
|(f,\frac{g}{\sqrt{T}})|^s\big)^{\frac{p_2}{s}}
\Big\rangle^{\frac{q}{p_2}}
&\leq \int_{\Omega}
\Big( \dashint_{U_1(x)}
\big\langle|(f,\frac{g}{\sqrt{T}})|^{p_2}
\big\rangle^{\frac{s}{p_2}}
\Big)^{\frac{q}{s}}\\
&\leq
\int_{\Omega}
\dashint_{U_1(x)}\big\langle|(f,\frac{g}{\sqrt{T}})
|^{p_2}\big\rangle^{\frac{q}{p_2}}
\lesssim^{\eqref{g-1*}}
\int_{\Omega}
\big\langle|(f,\frac{g}{\sqrt{T}})
|^{p_2}\big\rangle^{\frac{q}{p_2}}.
\end{aligned}
\end{equation}
Plugging the estimates $\eqref{f:4.27x}$ and $\eqref{f:4.28x}$
back into $\eqref{f:4.26x}$, we obtain
\begin{equation*}\label{}
\begin{aligned}
\int_{\Omega}
\big\langle|(\nabla u,\frac{u}{\sqrt{T}})|^{p}\big\rangle^{\frac{q}{p}}
&\lesssim
\int_{\Omega}
\big\langle
|(f,\frac{g}{\sqrt{T}})|^{p_1}
\big\rangle^{\frac{q}{p_1}}
+ \int_{\Omega}
\Big\langle\Big(\dashint_{U_1(x)}
\big|\big(\nabla u,\frac{u}{\sqrt{T}},f,\frac{g}{\sqrt{T}}\big)
\big|^s\Big)^{\frac{p_2}{s}}
\Big\rangle^{\frac{q}{p_2}}\\
&\lesssim^{\eqref{f:4.29x}}
\int_{\Omega}
\big\langle
|(f,\frac{g}{\sqrt{T}})|^{p_2}
\big\rangle^{\frac{q}{p_2}}
+ \int_{\Omega}
\Big\langle\Big(\dashint_{U_1(x)}
\big|\big(\nabla u,\frac{u}{\sqrt{T}}\big)
\big|^s\Big)^{\frac{p_2}{s}}
\Big\rangle^{\frac{q}{p_2}},
\end{aligned}
\end{equation*}
which gives the desired estimate $\eqref{pri:4.9x}$.

\medskip
\noindent
\textbf{Step 5.} Show the estimate $\eqref{pri:4.10x}$.
It suffices to modify the estimates $\eqref{f:4.27x}$
and $\eqref{f:4.28x}$, correspondingly, since we can similarly
obtain by revising $\eqref{f:4.26x}$ that
\begin{equation*}
\begin{aligned}
\int_{\Omega}
\big\langle|(\nabla u,\frac{u}{\sqrt{T}})(y)|^{p}\big\rangle^{\frac{q}{p}} \omega(y)dy
&\lesssim
\int_{\Omega}\int_{\Omega} \big\langle
I_{\frac{1}{2}B_{*}(x)}(y)
\big|(\nabla\tilde{u}_x,\frac{\tilde{u}_x}{\sqrt{T}})(y)\big|^{p_1}
\big\rangle^{\frac{q}{p_1}} dx \omega(y)dy\\
&+\int_{\Omega}\int_{\Omega} \big\langle
I_{\frac{1}{2}B_*(x)}(y)
\big|(\nabla u-\nabla\tilde{u}_x,\frac{u-\tilde{u}_x}{\sqrt{T}})(y)\big|^{p_1}
\big\rangle^{\frac{q}{p_1}} dx\omega(y)dy:=I_3+I_4.
\end{aligned}
\end{equation*}
Therefore, we start from
the following estimate
\begin{equation*}
\begin{aligned}
I_3
&\leq \int_{\Omega}\int_{\Omega} \big\langle
\big|(\nabla\tilde{u}_x,\frac{\tilde{u}_x}{\sqrt{T}})(y)\big|^{p_1}
\big\rangle^{\frac{q}{p_1}}\omega(y) dydx
\lesssim^{\eqref{pri:4.13x-3}}
\int_{\Omega}\int_{\Omega}
\big\langle
|(\tilde{f}_x,\frac{\tilde{g}_x}{\sqrt{T}})(y)|^{p_1}
\big\rangle^{\frac{q}{p_1}}\omega(y) dy dx\\
&\lesssim \int_{\Omega}\dashint_{U_1(x)}
\big\langle
|(f,\frac{g}{\sqrt{T}})(y)|^{p_1}
\big\rangle^{\frac{q}{p_1}} \omega(y) dy dx
\lesssim^{\eqref{g-1*}} \int_{\Omega}
\big\langle
|(f,\frac{g}{\sqrt{T}})|^{p_1}
\big\rangle^{\frac{q}{p_1}}\omega,
\end{aligned}
\end{equation*}
and then we have
\begin{equation*}
\begin{aligned}
I_4
&\leq \int_{\Omega}\Big\langle
\Big(\int_{\Omega}
I_{\frac{1}{2}B_*(x)}(y)
\big|(\nabla u-\nabla\tilde{u}_x,\frac{u-\tilde{u}_x}{\sqrt{T}})(y)\big|^{q}
\omega(y)dy \Big)^{\frac{p_1}{q}}\Big\rangle^{\frac{q}{p_1}} dx\\
&\lesssim_d
\int_{\Omega}
\Big\langle\Big(\dashint_{\frac{1}{2}B_*(x)\cap\Omega}
|(\nabla u-\nabla\tilde{u}_x,\frac{u-\tilde{u}_x}{\sqrt{T}})
|^q \omega(y)dy\Big)^{\frac{p_1}{q}}\Big\rangle^{\frac{q}{p_1}}dx\\
&\lesssim^{\eqref{pri:4.14x-1}}
\int_{\Omega}
\Big\langle\Big(\dashint_{U_1(x)}
\big|\big(\nabla u,\frac{u}{\sqrt{T}},f,\frac{g}{\sqrt{T}}\big)
\big|^s\Big)^{\frac{p_2}{s}}
\Big\rangle^{\frac{q}{p_2}}\Big(\dashint_{B_1(x)}\omega\Big)dx.
\end{aligned}
\end{equation*}
This implies the stated estimate $\eqref{pri:4.10x}$, where
we employ H\"older's inequality and the definition of $A_q$-weight,
and we are done.
\qed

\medskip

\noindent
\textbf{Proof of Theorem $\ref{thm:2}$.}
Let $1<p<p_1<\bar{p}<\infty$ and $s,q\in(1,\infty)$.
The stated estimates $\eqref{pri:5.0}$ and $\eqref{pri:5.0-w}$ mainly follows from Propositions $\ref{P:3}$ and $\ref{P:6}$. We start from
\begin{equation*}
\begin{aligned}
\int_{\Omega} \big\langle|(\nabla u,\frac{u}{\sqrt{T}} )|^{p}\big\rangle^{\frac{q}{p}}
&\lesssim_{\lambda,d,M_0,p,q}^{\eqref{pri:4.9x}}
\int_{\Omega}
\big\langle|(f,\frac{g}{\sqrt{T}})|^{p_1}
\big\rangle^{\frac{q}{p_1}}
+ \int_{\Omega}
\Big\langle\big(\dashint_{U_1(x)}
|(\nabla u,\frac{u}{\sqrt{T}})|^s\big)^{\frac{p_1}{s}}
\Big\rangle^{\frac{q}{p_1}}\\
&\lesssim^{\eqref{pri:2.1}}
\int_{\Omega}
\big\langle|(f,\frac{g}{\sqrt{T}})|^{p_1}
\big\rangle^{\frac{q}{p_1}}
+ \sqrt{T}^{dq}
\bigg(\int_{\Omega}
\Big\langle\Big(\dashint_{U_1(x)}
\big|(\frac{g}{\sqrt{T}},f)\big|^s\Big)^{\frac{\bar{p}}{s}}
\Big\rangle^{\frac{q}{\bar{p}}}dx\bigg)^{\frac{1}{q}}\\
&\lesssim^{\eqref{f:4.29x}}
\sqrt{T}^{dq}\int_{\Omega}
\big\langle|(f,\frac{g}{\sqrt{T}})|^{p_1}
\big\rangle^{\frac{q}{p_1}},
\end{aligned}
\end{equation*}
which gives the stated estimates $\eqref{pri:5.0}$
with $\eqref{rv-bound-1}$, where we can take $s\leq\min\{q,\bar{p}\}$.

Now, we proceed to show $\eqref{pri:5.0-w}$. Assume
$s$ additionally satisfies $0<s-1\ll 1$ such that
$\omega\in A_q$ implies $\omega\in A_{q/s}$. Then,
it follows that
\begin{equation}\label{f:2.21}
\begin{aligned}
&\int_{\Omega} \big\langle|(\nabla u,\frac{u}{\sqrt{T}} )|^{p}\big\rangle^{\frac{q}{p}}\omega\\
&\lesssim_{\lambda,d,\partial\Omega,p,q}^{\eqref{pri:4.10x}}
\int_{\Omega}
\big\langle|(f,\frac{g}{\sqrt{T}})|^{p_1}
\big\rangle^{\frac{q}{p_1}}\omega
+ \int_{\Omega}
\Big\langle\Big(\dashint_{U_1(z)}
|(\nabla u,\frac{u}{\sqrt{T}})|^s\Big)^{\frac{p_1}{s}}
\Big\rangle^{\frac{q}{p_1}}\Big(\dashint_{B_1(z)}\omega\Big)dz\\
&\lesssim^{\eqref{pri:2.1-w}}_{\lambda,d,\partial\Omega,p,q,[\omega]_{A_q}}
\int_{\Omega}
\big\langle|(f,\frac{g}{\sqrt{T}})|^{p_1}
\big\rangle^{\frac{q}{p_1}}\omega
+ T^{\frac{3dq}{4}}
\int_{\Omega}
\Big\langle\Big(\dashint_{U_1(z)}
|(f,\frac{g}{\sqrt{T}})|^s\Big)^{\frac{\bar{p}}{s}}
\Big\rangle^{\frac{q}{\bar{p}}}\omega(z) dz.
\end{aligned}
\end{equation}
On the other hand, a routine computation leads to
\begin{equation}\label{f:2.22}
\begin{aligned}
\int_{\Omega}
\Big\langle\Big(\dashint_{U_1(z)}
|(f,\frac{g}{\sqrt{T}})|^s\Big)^{\frac{\bar{p}}{s}}
\Big\rangle^{\frac{q}{\bar{p}}}
&\omega dz
\leq
\int_{\Omega}
\Big(\dashint_{U_1(z)}
\big\langle|(f,\frac{g}{\sqrt{T}})|^{\bar{p}}
\big\rangle^{\frac{s}{\bar{p}}}\Big)^{\frac{q}{s}}
\omega dz\\
&\leq \int_{\mathbb{R}^d}
\big|\mathcal{M}\big(
\big\langle|(f,\frac{g}{\sqrt{T}})|^{\bar{p}}
\big\rangle^{\frac{s}{\bar{p}}}\big)\big|^{\frac{q}{s}}
\omega
\lesssim^{\eqref{f:19-2}} \int_{\mathbb{R}^d}
\big\langle|(f,\frac{g}{\sqrt{T}})|^{\bar{p}}
\big\rangle^{\frac{q}{\bar{p}}}\omega.
\end{aligned}
\end{equation}
Combining the estimates $\eqref{f:2.21}$ and  $\eqref{f:2.22}$, we
have the desired estimate $\eqref{pri:5.0-w}$ and $\eqref{rv-bound-1}$.
\qed

\section{A non-perturbative regime}\label{sec:4}

\begin{theorem}\label{P:5}
Let $\Omega$ be a (bounded) $C^1$ domain
and $T\geq 1$. Let $1<q,p<\infty$ with $p<\bar{p}$.
Suppose that $\langle\cdot\rangle$ satisfies spectral gap condition $\eqref{a:2}$ and
the (admissible) coefficient satisfies
the regularity condition $\eqref{a:3}$.
Let $u$, $f$ and $g$ be associated by the following equations $\eqref{pde:3.2x}$.
Then, we have the optimal annealed Calder\'on-Zygmund estimate
\begin{equation}\label{pri:7.1}
\Big(\int_{\Omega} \big\langle|(\nabla u,\frac{u}{\sqrt{T}} )|^{p}\big\rangle^{\frac{q}{p}}\Big)^{\frac{1}{q}}
\lesssim_{\lambda,d,M_0,p,q}
\Big(\int_{\Omega}
\big\langle|(f,\frac{g}{\sqrt{T}})|^{\bar{p}}
\big\rangle^{\frac{q}{\bar{p}}}\Big)^{\frac{1}{q}}.
\end{equation}
Moreover,  for any $\omega\in A_q$, we then have
weighted annealed Calder\'on-Zygmund estimate
\begin{equation}\label{pri:7.1-w}
\Big(\int_{\Omega} \big\langle|(\nabla u,\frac{u}{\sqrt{T}} )|^{p}\big\rangle^{\frac{q}{p}}\omega\Big)^{\frac{1}{q}}
\lesssim_{\lambda,d,M_0,p,q,[\omega]_{A_q}}
\Big(\int_{\Omega}
\big\langle|(f,\frac{g}{\sqrt{T}})|^{\bar{p}}
\big\rangle^{\frac{q}{\bar{p}}}\omega\Big)^{\frac{1}{q}}.
\end{equation}
\end{theorem}

In the case of unbounded domains, the main difficulty is that we need
the stationary correctors to characterize the homogenization error
that we will employ in Subsection $\ref{subsec:4.1}$. To avoid the restriction of $d>2$, we introduce the massive extended corrector which had been used in \cite{Josien-Otto22} for the same consideration. In the case of bounded domains, the massive extended corrector seem to be unnecessary
(since we can use extended corrector directly).
However, we have to additionally assume $\Omega\ni\{0\}$ and $\sqrt{T}\geq R_0$ in Theorem $\ref{P:5}$. Therefore, to unify the proof, we adopt
the massive extended corrector instead of
the extended corrector even in the case of bounded domains.

\begin{lemma}[massive extended correctors \cite{Gloria-Neukamm-Otto20,Josien-Otto22}]\label{lemma:*3*}
Let $d\geq 1$ and $\tau\geq 1$. Assume that $\langle\cdot\rangle$ satisfies stationary and ergodic\footnote{In terms of ``ergodic'', we refer the reader to \cite[pp.105]{Gloria-Neukamm-Otto20} or \cite[pp.222-225]{Jikov-Kozlov-Oleinik94} for the definition.} conditions.
Then, there exist three stationary random  fields $\{\phi_{\tau,i}\}_{i=1,\cdots,d}$,
$\{\sigma_{\tau,ijk}\}_{i,j,k=1,\cdots,d}$,
and $\{\psi_{\tau,i}\}_{i=1,\cdots,d}$
in the sense of
$\phi_{i}(a;x+z)=\phi_i(a(\cdot+z);x)$ for any shift
vector $z\in\mathbb{R}^d$ (and likewise for $\sigma_{\tau,ijk}$
and $\psi_{\tau,i}$), which are called extended correctors, satisfying
\begin{equation}\label{eq:*2.1*}
\begin{aligned}
\frac{1}{\tau}\phi_{\tau,i}-\nabla\cdot a(\nabla\phi_i+e_i) &= 0;\\
(\frac{1}{\tau}-\Delta)\sigma_{\tau,ijk} &=
\partial_j q_{\tau,ik} - \partial_k q_{\tau,ij}\\
(\frac{1}{\tau}-\Delta)\psi_{\tau,i}&=
q_{\tau,i}-\bar{a}_{\tau}e_i-\nabla\phi_{\tau,i},
\end{aligned}
\end{equation}
in the distributional sense on $\mathbb{R}^d$ with
$q_{\tau,ij}:=e_j\cdot a(e_i+\nabla\phi_{\tau,i})$ and
$\bar{a}_{\tau}e_i:=\langle q_{\tau,i}\rangle$, and
$e_i$ is the canonical basis of $\mathbb{R}^d$. Also,
the field $\sigma_{\tau}$ is skew-symmetric in its last indices, that is
$\sigma_{\tau,ijk} = -\sigma_{\tau,ikj}$.
%\begin{equation}\label{}
% \sigma_{ijk} = -\sigma_{ikj}.
%\end{equation}
Furthermore, if $\langle\cdot\rangle$ additionally satisfies
the spectral gap condition $\eqref{a:2}$, as well as $\eqref{a:3}$.
For any $p\in[1,\infty)$, there holds
\begin{equation}\label{pri:*2*}
\big\langle |(\nabla\phi_{\tau},\nabla\sigma_{\tau},\frac{\nabla\psi_{\tau}}{\sqrt{\tau}})|^p\big\rangle
\lesssim_{\lambda,\lambda_1,\lambda_2,d,p} 1,
\end{equation}
and
\begin{equation}\label{pri:*3*}
 \Big\langle\big|(\phi_{\tau},\sigma_{\tau},\frac{\psi_{\tau}}{\sqrt{\tau}})\big|^p \Big\rangle^{\frac{1}{p}}
 \lesssim_{\lambda,\lambda_1,\lambda_2,d,p} \mu_d(\sqrt{\tau}):=\left\{\begin{aligned}
 &\sqrt{1+\sqrt{\tau}}; &~&d=1;\\
 &\ln^{\frac{1}{2}} (2+\sqrt{\tau}); &~& d=2;\\
 & 1 &~& d>2.
 \end{aligned}\right.
\end{equation}
\end{lemma}

\subsection{High $\&$ low pass}\label{subsec:4.1}
Let $u$ be the solution of $\eqref{pde:3.2x}$ and $u:=u_{>}+u_{<}$, where $u_{>}$ is called as the high pass and $u_{<}$ is the low pass.
Let $1\leq\tau\leq T$.
For the high pass part, $u_>\in H^{1}_0(\Omega)$ satisfies
\begin{equation}\label{pde:7.1}
\left\{\begin{aligned}
\frac{1}{\tau}u_{>}-\nabla\cdot a\nabla u_{>}
&= \nabla\cdot f +\frac{1}{T}g
&\quad&\text{in}~~\Omega;\\
u_{>} &= 0
&\quad&\text{on}~~\partial\Omega.
\end{aligned}\right.
\end{equation}
For the low pass part,  $u_{<}\in H_0^{1}(\Omega)$ satisfies
\begin{equation}\label{pde:7.2}
\left\{\begin{aligned}
\frac{1}{\tau}u_{<}-\nabla\cdot a\nabla u_{<}
&= \big(\frac{1}{\tau}-\frac{1}{T}\big)u
&\quad&\text{in}~~\Omega;\\
u_{<} &= 0
&\quad&\text{on}~~\partial\Omega.
\end{aligned}\right.
\end{equation}

For the equations $\eqref{pde:7.2}$, we now consider the corresponding homogenized equation, i.e.,
\begin{equation}\label{pde:7.2h-u}
\left\{\begin{aligned}
\frac{1}{\tau}\bar{u}-\nabla\cdot \bar{a}_{\tau}\nabla \bar{u}
&= \big(\frac{1}{\tau}-\frac{1}{T}\big)u
&\quad&\text{in}~~\Omega;\\
\bar{u} &= 0
&\quad&\text{on}~~\partial\Omega,
\end{aligned}\right.
\end{equation}
where we recall $\bar{a}_{\tau}e_i:=\langle a(\nabla\phi_{\tau,i}+e_i)\rangle$ in Lemma $\ref{lemma:*3*}$, and
the massive corrector $\phi_{\tau,i}$ satisfies $\eqref{eq:*2.1*}$.

We decompose the effective solution of
$\eqref{pde:7.2h-u}$ into the related high and low pass, respectively. Let $\bar{u} = \bar{u}_{>}+\bar{u}_{<}$. For the high pass part,
$\bar{u}_{>}$ satisfies
\begin{equation}\label{pde:7.2h-1-u}
\left\{\begin{aligned}
\frac{1}{T}\bar{u}_{>}-\nabla\cdot \bar{a}_{\tau}\nabla \bar{u}_{>}
&= \big(\frac{1}{\tau}-\frac{1}{T}\big)u_{>}
&\quad&\text{in}~~\Omega;\\
\bar{u}_{>} &= 0
&\quad&\text{on}~~\partial\Omega,
\end{aligned}\right.
\end{equation}
while the low pass $\bar{u}_{<}$ satisfies
\begin{equation}\label{pde:7.2h-2-u}
\left\{\begin{aligned}
\frac{1}{T}\bar{u}_{<}-\nabla\cdot \bar{a}_{\tau}\nabla \bar{u}_{<}
&= \big(\frac{1}{\tau}-\frac{1}{T}\big)(u_{<}-\bar{u})
&\quad&\text{in}~~\Omega;\\
\bar{u}_{<} &= 0
&\quad&\text{on}~~\partial\Omega,
\end{aligned}\right.
\end{equation}

Let $\varphi_i\in H_0^1(\Omega)$ be given later on. Appealing to massive correctors, we can define the error of the first-order expansion as follows:
\begin{equation}\label{app-corrector-u}
 w = u_{<} - \bar{u} - \phi_{\tau,i}\varphi_i,
\end{equation}
and there holds $w=0$ on $\partial\Omega$, satisfying
\begin{equation}\label{pde:7.3-u}
\frac{1}{\tau}w-\nabla\cdot a\nabla w
=\nabla\cdot\Big[\big(a\phi_{\tau,i}-\sigma_{\tau,i}\big)\nabla\varphi_i
+(a-\bar{a}_\tau)(1-\eta)\nabla\bar{u}+\frac{1}{\tau}\psi_{\tau,i}\varphi_i\Big]
-\frac{1}{\tau}\phi_{\tau,i}\varphi_i
~\text{in}~ \Omega,
\end{equation}
where the skew symmetric tensor field $\sigma_{\tau,i}$ with the auxiliary field $\psi_{\tau,i}$ is defined in $\eqref{eq:*2.1*}$.
Now, let $\eta\in C^1(\Omega)$ be a cut-off function, satisfying
\begin{equation}\label{cut-off}
\eta=1 ~\text{in}~\Omega\setminus \tilde{O}_{1};
\quad
\eta=0 ~\text{in}~\tilde{O}_{1/2};
\quad |\nabla\eta|\lesssim 1,
\end{equation}
where $O_r:=\{x\in\Omega:
\text{dist}(x,\partial\Omega)\leq r\}$.
Then, we can set $\varphi_i := \eta\partial_i\bar{u}$,
$\varphi_{i,>} := \eta\partial_i\bar{u}_{>}$ and
$\varphi_{i,<} := \eta\partial_i\bar{u}_{<}$.
With the help of the error $\eqref{app-corrector-u}$,
we proceed to decompose the pair $(u,\nabla u)$ into the high pass part $(v_{>},h_{>})$ and the low pass part $(v_{<},h_{<})$, and they are given  as follows:
\begin{itemize}
  \item High pass part:
  \begin{equation}\label{high-pass*-u}
  \left\{\begin{aligned}
  v_{>} &: = u_{>} + \bar{u}_{>};\\
  h_{>} &:= \nabla u_{>} + (e_i+\nabla\phi_{\tau,i})\varphi_{i,>},
  \end{aligned}\right.
  \end{equation}
  \item Low pass part:
  \begin{equation}\label{low-pass*-u}
  \left\{\begin{aligned}
  v_{<} &: = w + \bar{u}_{<} + \phi_{\tau,i}\varphi_i;\\
  h_{<} &:=
  \nabla w + (e_i+\nabla\phi_{\tau,i})
  \varphi_{i,<} +
  \phi_{\tau,i}\nabla\varphi_i + (1-\eta)\nabla\bar{u}.
  \end{aligned}\right.
  \end{equation}
\end{itemize}

For the ease of the statement, we introduce the annealed norms, defined by
\begin{equation}\label{norm-1}
\|f\|_{q,r,\Omega}:=
\Big(\int_{\Omega}\langle|f|^r\rangle^{\frac{q}{r}}\Big)^{\frac{1}{q}};
\qquad
\|f;\omega^{\frac{1}{q}}\|_{q,r,\Omega}:=
\Big(\int_{\Omega}\langle|f|^r
\rangle^{\frac{q}{r}}\omega\Big)^{\frac{1}{q}}.
\end{equation}

\begin{lemma}\label{lemma:7.0}
Let $\Omega$ be a (bounded) $C^1$ domain
and $T>\tau\geq 1$.
For any given $q,r\in(1,\infty)$, assume that $1<s_4<s_3<s_2<s_1<r$ and $1<r_2<r_1<r$.
Let $(v_{>},h_{>})$ and $(v_{<},h_{<})$ be given by $\eqref{high-pass*-u}$ and
$\eqref{low-pass*-u}$, respectively. Then, there hold the following estimates:
\begin{subequations}
\begin{align}
 & \big\|(\frac{v_{>}}{\sqrt{T}},h_{>})\big
 \|_{q,r_2,\Omega}
 \lesssim
 C_{q,r_1,r}(\tau)
 \big\|(\frac{g}{\sqrt{T}},f)\big\|_{q,r,\Omega};
 \label{pri:7.2a}\\
 &
\big\|(\frac{v_{<}}{\sqrt{T}},h_{<})\big\|_{q,s_4,\Omega}
\lesssim \mu_d(\sqrt{T})
C_{q,s_3,s_2}(\tau)
\tau^{-\frac{1}{2}}
C_{q,s_1,r}^2(T)
\big\|(\frac{g}{\sqrt{T}},f)\big\|_{q,r,\Omega},
\label{pri:7.2b}
\end{align}
\end{subequations}
and for any $\omega\in A_q$ one can obtain weighted estimates:
\begin{subequations}
\begin{align}
 & \big\|(\frac{v_{>}}{\sqrt{T}},h_{>});\omega^{\frac{1}{q}}\big
 \|_{q,r_2,\Omega}
 \lesssim
 \tilde{C}_{q,r_1,r}(\tau)
 \big\|(\frac{g}{\sqrt{T}},f);
 \omega^{\frac{1}{q}}\big\|_{q,r,\Omega};  \label{pri:7.3a}\\
&\big\|(\frac{v_{<}}{\sqrt{T}},h_{<});\omega^{\frac{1}{q}}\big\|_{q,s_4,\Omega}
\lesssim \mu_d^2(\sqrt{T})
\tilde{C}_{q,s_3,s_2}(\tau)
\tau^{-\frac{1}{2}}
\tilde{C}_{q,s_1,r}^2(T)
\big\|(\frac{g}{\sqrt{T}},f);
\omega^{\frac{1}{q}}\big\|_{q,r,\Omega},
\label{pri:7.3b}
\end{align}
\end{subequations}
where the multiplicative constants depends on $\lambda,d,M_0$.
\end{lemma}

\begin{lemma}\label{lemma:7.1}
Let $\Omega$ be a (bounded) $C^1$ domain and $T\geq 1$. Suppose that
$\bar{u}$ is the weak solution to $\eqref{pde:7.2h-u}$, and
$\eta$ is the cut-off function satisfying $\eqref{cut-off}$.  Then,
for any $1<p,q<\infty$,
we have
\begin{equation}\label{pri:7.4}
\big\|
(\frac{\eta\nabla\bar{u}}{\sqrt{\tau}},
\eta\nabla^2\bar{u},
(1-\eta)\nabla\bar{u}
)
\big\|_{q,p,\Omega}
\lesssim \sqrt{T/\tau}
\big\|(\frac{u}{\sqrt{T}},\nabla u)\big\|_{q,p,\Omega}.
\end{equation}
Moreover, for any $\omega\in A_q$, there holds
\begin{equation}\label{pri:7.4-w}
\big\|
(\frac{\eta\nabla\bar{u}}{\sqrt{\tau}},
\eta\nabla^2\bar{u},
(1-\eta)\nabla\bar{u}
);\omega^{\frac{1}{q}}
\big\|_{q,p,\Omega}
\lesssim \sqrt{T/\tau}
\big\|\big(\frac{u}{\sqrt{T}},\nabla u\big);\omega^{\frac{1}{q}}\big\|_{q,p,\Omega}.
\end{equation}
\end{lemma}

\begin{proof}
The proof is divided into three steps.

\medskip
\noindent
\textbf{Step 1.} As a preparation, decompose the equations $\eqref{pde:7.2h-u}$. Due to the linearity of $\eqref{pde:7.2h-u}$,
one can set $\bar{u}=\bar{u}^{(1)}+ \bar{u}^{(2)}$, and $\bar{u}^{(1)}$,
$\bar{u}^{(2)}$ satisfy
\begin{equation}\label{pde:7.4}
\frac{1}{\tau}\bar{u}^{(1)} -\nabla\cdot\bar{a}_{\tau}\nabla\bar{u}^{(1)}
= \big(\frac{1}{\tau}-\frac{1}{T}\big)\tilde{u}^{0} \quad \text{in}\quad \mathbb{R}^d,
\end{equation}
and
\begin{equation}\label{pde:7.5}
\left\{\begin{aligned}
\frac{1}{\tau}\bar{u}^{(2)}-\nabla\cdot\bar{a}_{\tau}\nabla\bar{u}^{(2)} &= 0
\quad&\text{in}&\quad \Omega;\\
\bar{u}^{(2)} &= -\bar{u}^{(1)}
\quad&\text{on}&\quad \partial\Omega,
\end{aligned}\right.
\end{equation}
respectively, where $\tilde{u}^0$ is a suitable 0-extension such that
\begin{equation}\label{f:7.4}
\begin{aligned}
&\|\nabla\tilde{u}^0\|_{q,p}\lesssim \|\nabla u\|_{q,p,\Omega};\\
&\|\nabla\tilde{u}^0;\omega^{\frac{1}{q}}\|_{q,p}\lesssim \|\nabla u;\omega^{\frac{1}{q}}\|_{q,p,\Omega}.
\end{aligned}
\end{equation}
%(see e.g. \cite[Theorem 2.1.13]{Turesson00}).
%Let $\|\cdot\|_p:=\langle(\cdot)^p\rangle^{1/p}$.
%For any $\omega\in A_{s}$ with
%$s=\frac{\underline{q}(d-1)}{d-\underline{q}}\in(\underline{q},q)$ and $1/\underline{q}=:1/q+1/d$, we have
%\begin{equation}\label{f:7.7}
%\begin{aligned}
% \Big(\int_{\mathbb{R}^d}\|\tilde{u}^0\|_{p}^q\omega\Big)^{\frac{1}{q}}
%& \lesssim_{d,q,[\omega]_{A_s}}
% \Big(\int_{\mathbb{R}^d}\|\nabla \tilde{u}^0\|_{p}^{\underline{q}}
% \omega^{1-\frac{\underline{q}}{d}}\Big)^{\frac{1}{\underline{q}}}\\
%&\lesssim^{\eqref{f:7.4}}
%\Big(\int_{\Omega}\|\nabla u\|_{p}^{\underline{q}}
% \omega^{1-\frac{\underline{q}}{d}}\Big)^{\frac{1}{\underline{q}}}
% \lesssim
%\Big(\int_{\Omega}\|\nabla u\|_{p}^{q}
% \omega\Big)^{\frac{1}{q}}|\Omega|^{\frac{1}{d}}
% \lesssim \sqrt{T}
% \Big(\int_{\Omega}\|\nabla u\|_{p}^{q}
% \omega\Big)^{\frac{1}{q}}.
%\end{aligned}
%\end{equation}

In terms of the equation $\eqref{pde:7.4}$, it follows from Mikhlin theorem in UMD spaces that
\begin{subequations}
\begin{align}
&\|(\frac{\bar{u}^{(1)}}{\sqrt{\tau}},
 \nabla \bar{u}^{(1)})\|_{q,p}
 \lesssim_{p,q}
 \|\frac{\tilde{u}^{0}}{\sqrt{\tau}}\|_{q,p}
 \lesssim
 \|\frac{u}{\sqrt{\tau}}\|_{q,p,\Omega};
 \label{f:7.6}\\
&\|(\frac{\bar{u}^{(1)}}{\sqrt{\tau}},
 \nabla \bar{u}^{(1)});\omega^{\frac{1}{q}}\|_{q,p}
 \lesssim_{p,q}
 \|\frac{\tilde{u}^{0}}{\sqrt{\tau}};\omega^{\frac{1}{q}}\|_{q,p}
 \lesssim
 \|\frac{u}{\sqrt{\tau}};\omega^{\frac{1}{q}}\|_{q,p,\Omega};
 \label{f:7.6-w}
\end{align}
\end{subequations}
and
\begin{subequations}
\begin{align}
&\|(\frac{\nabla\bar{u}^{(1)}}{\sqrt{\tau}},
 \nabla^2 \bar{u}^{(1)})\|_{q,p}
 \lesssim_{p,q}
 \|\frac{\nabla \tilde{u}^{0}}{\sqrt{\tau}}\|_{q,p}
 \lesssim^{\eqref{f:7.4}}
 \|\frac{\nabla u}{\sqrt{\tau}}\|_{q,p,\Omega};
 \label{f:7.5}\\
&\|(\frac{\nabla\bar{u}^{(1)}}{\sqrt{\tau}},
 \nabla^2 \bar{u}^{(1)});\omega^{\frac{1}{q}}\|_{q,p}
 \lesssim_{p,q}
 \|\frac{\nabla \tilde{u}^{0}}{\sqrt{\tau}};\omega^{\frac{1}{q}}\|_{q,p}
 \lesssim^{\eqref{f:7.4}}
 \|\frac{\nabla u}{\sqrt{\tau}};\omega^{\frac{1}{q}}\|_{q,p,\Omega};
 \label{f:7.5-w}
\end{align}
\end{subequations}

In terms of the equation $\eqref{pde:7.5}$. For any $x\in\Omega\setminus \tilde{O}_{1/2}$,
we start from
\begin{equation*}
  \big|\big(\frac{\nabla\bar{u}^{(2)}}{\sqrt{\tau}},\nabla^2\bar{u}^{(2)}\big)(x)\big|
  \lesssim_{\lambda,d} \frac{1}{\delta(x)}
  \dashint_{B_{\delta(x)}(x)}|\nabla\bar{u}^{(2)}|,
\end{equation*}
and this implies
\begin{equation}\label{f:7.8}
  \big\|\big(\frac{\nabla\bar{u}^{(2)}}{\sqrt{\tau}},\nabla^2\bar{u}^{(2)}\big)(x)
  \big\|_{p}
  \lesssim_{\lambda,d}
  \dashint_{B_{1/2}(x)}\|\nabla\bar{u}^{(2)}\|_{p},
\end{equation}
where we notice that $\delta(x)\geq 1/2$. On the one hand, we have
\begin{equation}\label{f:7.9}
\begin{aligned}
\int_{\Omega}\big\|\big(\frac{\nabla\bar{u}^{(2)}}{\sqrt{\tau}},
\nabla^2\bar{u}^{(2)}\big)
  \big\|_{p}^{q}
&\lesssim^{\eqref{f:7.8}}
\int_{\Omega}
\dashint_{B_{1/2}(x)}\|\nabla\bar{u}^{(2)}\|_{p}^q dx
\lesssim^{\eqref{g-1*}} \int_{\Omega}\|\nabla\bar{u}^{(2)}\|_{p}^q \\
&\lesssim^{\eqref{pri:A-2}}
\int_{\Omega}\big\|\big(\frac{\bar{u}^{(1)}}{\sqrt{\tau}},
\nabla\bar{u}^{(1)}\big)
  \big\|_{p}^{q}
\lesssim^{\eqref{f:7.6}} \|\frac{u}{\sqrt{\tau}}\|_{q,p,\Omega}^{q}.
\end{aligned}
\end{equation}
On the other hand, for any $\omega\in A_q$,
\begin{equation}\label{f:7.9-w}
\begin{aligned}
\int_{\Omega}\big\|\big(\frac{\nabla\bar{u}^{(2)}}{\sqrt{\tau}},
\nabla^2\bar{u}^{(2)}\big)
  \big\|_{p}^{q}\omega
&\lesssim^{\eqref{f:7.8}}
\int_{\Omega}
\mathcal{M}\big(\|\nabla\bar{u}^{(2)}\|_{p}\big)^q\omega
\lesssim \int_{\Omega}\|\nabla\bar{u}^{(2)}\|_{p}^q\omega \\
&\lesssim^{\eqref{pri:A-3}}
\int_{\Omega}\big\|\big(\frac{\bar{u}^{(1)}}{\sqrt{\tau}},
\nabla\bar{u}^{(1)}\big);\omega^{\frac{1}{q}}
  \big\|_{p}^{q}
\lesssim^{\eqref{f:7.6-w}} \|\frac{u}{\sqrt{\tau}};\omega^{\frac{1}{q}}\|_{q,p,\Omega}^{q}.
\end{aligned}
\end{equation}

\medskip
\noindent
\textbf{Step 2.} Show the estimate $\eqref{pri:7.4}$. We start from
\begin{equation}\label{f:7.10}
\begin{aligned}
\big\|
(\frac{\eta\nabla\bar{u}}{\sqrt{\tau}},
\eta\nabla^2\bar{u})
\big\|_{q,p,\Omega}
&\leq \big\|
(\frac{\nabla\bar{u}^{(1)}}{\sqrt{\tau}},
\nabla^2\bar{u}^{(1)})
\big\|_{q,p} +
\big\|
(\frac{\nabla\bar{u}^{(2)}}{\sqrt{\tau}},
\nabla^2\bar{u}^{(2)})
\big\|_{q,p,\Omega}\\
&\lesssim^{\eqref{f:7.5},\eqref{f:7.9}}
\|\frac{\nabla u}{\sqrt{\tau}}\|_{q,p,\Omega}
+ \|\frac{u}{\sqrt{\tau}}\|_{q,p,\Omega}
\lesssim
\sqrt{\frac{T}{\tau}}\|(\frac{ u}{\sqrt{T}},\nabla u)\|_{q,p,\Omega}.
\end{aligned}
\end{equation}
Then, it follows from Lemma $\ref{lemma:ap-2}$ and Poincar\'e's inequality that
\begin{equation}\label{f:7.11}
\big\|(1-\eta)\nabla\bar{u}
\big\|_{q,p,\Omega}
\leq
\big\|\nabla\bar{u}
\big\|_{q,p,\Omega}
\lesssim^{\eqref{pri:A-2}}
\big\|\frac{u}{\sqrt{\tau}}\big\|_{q,p,\Omega}.
\end{equation}

Consequently, the stated estimate $\eqref{pri:7.4}$
follows from $\eqref{f:7.10}$ and $\eqref{f:7.11}$.

\medskip
\noindent
\textbf{Step 3.} Show the estimate $\eqref{pri:7.4-w}$. By the same token,
we have
\begin{equation*}
\begin{aligned}
\big\|
(\frac{\eta\nabla\bar{u}}{\sqrt{\tau}},
\eta\nabla^2\bar{u});\omega^{\frac{1}{q}}
\big\|_{q,p,\Omega}
&\leq \big\|
(\frac{\nabla\bar{u}^{(1)}}{\sqrt{\tau}},
\nabla^2\bar{u}^{(1)});\omega^{\frac{1}{q}}
\big\|_{q,p} +
\big\|
(\frac{\nabla\bar{u}^{(2)}}{\sqrt{\tau}},
\nabla^2\bar{u}^{(2)});\omega^{\frac{1}{q}}
\big\|_{q,p,\Omega}\\
&\lesssim^{\eqref{f:7.5-w},\eqref{f:7.9-w}}
\|\frac{\nabla u}{\sqrt{\tau}};\omega^{\frac{1}{q}}\|_{q,p,\Omega}
+ \|\frac{u}{\sqrt{\tau}};\omega^{\frac{1}{q}}
\|_{q,p,\Omega}
\end{aligned}
\end{equation*}
This together with
\begin{equation*}
\big\|(1-\eta)\nabla\bar{u}
;\omega^{\frac{1}{q}}\big\|_{q,p,\Omega}
\leq
\big\|\nabla\bar{u}
;\omega^{\frac{1}{q}}\big\|_{q,p,\Omega}
\lesssim^{\eqref{pri:A-3}}
\big\|\frac{u}{\sqrt{\tau}};\omega^{\frac{1}{q}}\big\|_{q,p,\Omega}
\end{equation*}
leads to the desired estimate $\eqref{pri:7.4-w}$.
\end{proof}

\medskip
\noindent
\textbf{Proof of Lemma $\ref{lemma:7.0}$.}
The main idea is similar to that given in \cite[Proposition 7.3]{Josien-Otto22}, while we employ extended correctors $(\phi,\sigma)$ instead of the massive extended correctors $(\phi_{\tau},\sigma_{\tau},\psi_{\tau})$ therein. The whole proof is divided into three parts. The first part is
devoted to studying $\eqref{pri:7.2a}$, and the second one is
to show  $\eqref{pri:7.2b}$. The last part is to establish
$\eqref{pri:7.3b}$, since the estimate $\eqref{pri:7.3a}$ can be
established by the same proof as that given in the first part and
we don't reproduce it here.

\medskip
\noindent
\textbf{Part I. Arguments for the estimate $\eqref{pri:7.2a}$.}
Recall that $1<r_2<r_1<r$.
In view of the equation $\eqref{pde:7.1}$, it follows the suboptimal annealed Calder\'on-Zygmund estimate that
\begin{equation*}
\big\|(\frac{u_{>}}{\sqrt{\tau}},\nabla u_{>})\big\|_{q,r_1,\Omega}
\leq^{\eqref{pri:5.0}}
C_{q,r_1,r}(\tau)
\big\|(\frac{\sqrt{\tau}g}{T},f)
\big\|_{q,r,\Omega}
\leq
C_{q,r_1,r}(\tau)
\big\|(\frac{g}{\sqrt{T}},f)
\big\|_{q,r,\Omega},
\end{equation*}
where we use the fact that $T\geq \tau$ in the last inequality. Also, we further have
\begin{equation}\label{f:7.1}
\big\|(\frac{u_{>}}{\sqrt{T}},\nabla u_{>})\big\|_{q,r_1,\Omega}
\leq
C_{q,r_1,r}(\tau)
\big\|(\frac{g}{\sqrt{T}},f)
\big\|_{q,r,\Omega}.
\end{equation}
In view of the equations $\eqref{pde:7.1}$ and $\eqref{pde:7.2h-1-u}$, there holds that
\begin{equation}\label{pde:7.6}
\frac{1}{T}\bar{u}_{>}-\nabla\cdot \bar{a}_{\tau}\nabla \bar{u}_{>}
=\big(1-\frac{\tau}{T}\big)\Big[\nabla\cdot\big(a\nabla u_{>}+f\big)+\frac{g}{T}\Big].
\end{equation}

Then, it follows from Lemma $\ref{lemma:ap-2}$ that
\begin{equation}\label{f:7.2}
\big\|(\frac{\bar{u}_{>}}{\sqrt{T}},\nabla \bar{u}_{>})\big\|_{q,r_{1},\Omega}
\lesssim_{\lambda,d,M_0,q,r_{1}}^{\eqref{pri:A-2}}
\big\|(\frac{g}{\sqrt{T}},f,\nabla u_{>})\big\|_{q,r_{1},\Omega}
 \lesssim^{\eqref{f:7.1}} C_{q,r_1,r}(\tau)
\big\|(\frac{g}{\sqrt{T}},f)\big\|_{q,r,\Omega}.
\end{equation}
Moreover, we have
\begin{equation}\label{f:7.3}
\big\|(\frac{\bar{u}_{>}}{\sqrt{T}},(e_i+\nabla\phi_{\tau,i})
\varphi_{i,>})
\big\|_{q,r_2,\Omega}
\lesssim^{\eqref{pri:*2*}}
\big\|(\frac{\bar{u}_{>}}{\sqrt{T}},\eta\nabla \bar{u}_{>})\big\|_{q,r_1,\Omega}
\lesssim^{\eqref{f:7.2}} C_{q,r_1,r}(\tau)
\big\|(\frac{g}{\sqrt{T}},f)\big\|_{q,r,\Omega}.
\end{equation}

Combining the estimates $\eqref{f:7.1}$ and $\eqref{f:7.3}$ leads to
\begin{equation*}
\big\|(\frac{v_{>}}{\sqrt{T}},h_{>})
\big\|_{q,r_2,\Omega}
\lesssim_{\lambda,\lambda_1,\lambda_2,d,M_0}
 C_{q,r_1,r}(\tau)
\big\|(\frac{g}{\sqrt{T}},f)\big\|_{q,r,\Omega}.
\end{equation*}
which provides us the stated estimate $\eqref{pri:7.2a}$.

\medskip
\noindent
\textbf{Part II. Arguments for the estimate $\eqref{pri:7.2b}$.}
Recall $1<s_4<s_3<s_2<s_1<r$.
In this subsection, we will separately show the estimates for the low pass part $(v_{<},h_{<})$, i.e.,
\begin{equation*}
\big\|(\frac{w}{\sqrt{T}},\nabla w)
\big\|_{q,s_3,\Omega},
~
\big\|(\frac{\bar{u}_{<}}{\sqrt{T}},
(e_i+\nabla\phi_{\tau,i})\varphi_{i,<}
\big\|_{q,s_4,\Omega},
~
\big\|(\frac{\phi_{\tau,i}\varphi_i}{\sqrt{T}},
\phi_{\tau,i}\eta\nabla\partial_i\bar{u})
\big\|_{q,s_2,\Omega},
~\big\|\phi_{\tau,i}(1-\eta)\partial_i\bar{u}
\big\|_{q,s_2,\Omega},
\end{equation*}
and we will talk about them in each of steps.

\noindent
\textbf{Step 1.} Arguments for $\|(\frac{w}{\sqrt{T}},\nabla w)
\|_{q,s_3,\Omega}$. In view of Theorem $\ref{thm:2}$,
Lemmas $\ref{lemma:*3*}$ and $\ref{lemma:7.1}$, we obtain
\begin{equation}\label{f:7.12}
\begin{aligned}
\|(\frac{w}{\sqrt{T}},\nabla w)
\|_{q,s_3,\Omega}
&\lesssim^{\eqref{pri:5.0}} C_{q,s_3,s_2}(\tau)
\Big\{\|(\phi_{\tau,i},\sigma_{\tau,i},\frac{\psi_{\tau,i}}{\sqrt{\tau}})(\frac{\varphi_i}{\sqrt{\tau}},\nabla\varphi_i)
\|_{q,s_2,\Omega}
+ \|(1-\eta)\nabla\bar{u}
\|_{q,s_2,\Omega}\Big\}\\
&\lesssim^{\eqref{pri:*3*}} \mu_d(\sqrt{\tau})C_{q,s_3,s_2}(\tau)
\big\|
(\frac{\eta\nabla\bar{u}}{\sqrt{\tau}},
\eta\nabla^2\bar{u},
(1-\eta)\nabla\bar{u}
)
\big\|_{q,s_1,\Omega}\\
&\lesssim^{\eqref{pri:7.4}}
\mu_d(\sqrt{T})C_{q,s_3,s_2}(\tau)\sqrt{T/\tau}\big\|(\frac{u}{\sqrt{T}},\nabla u)\big\|_{q,s_1,\Omega}\\
&\lesssim^{\eqref{pri:5.0}}
\mu_d(\sqrt{T})C_{q,s_3,s_2}(\tau)\tau^{-\frac{1}{2}}
C_{q,s_1,r}^2(T)
\|(f,\frac{g}{\sqrt{T}})\|_{q,r,\Omega},
\end{aligned}
\end{equation}
where we also use the condition $T\geq\tau$ in the third inequality.

\noindent
\textbf{Step 2.}
Noting the right-hand side of $\eqref{pde:7.2h-2-u}$,
and with the help of the equality $\eqref{app-corrector-u}$,
we have the representation
\begin{equation*}
 u_{<}-\bar{u} = w+\phi_{\tau,i}\varphi_i.
\end{equation*}
Plugging this back into $\eqref{pde:7.2h-2-u}$, we derive that
\begin{equation}\label{pde:7.7}
\begin{aligned}
\frac{1}{T}\bar{u}_{<}-\nabla\cdot \bar{a}\nabla \bar{u}_{<}
&= \big(1-\frac{\tau}{T}\big)
\nabla\cdot\Big[a\nabla w +(a\phi_{\tau,i}-\sigma_{\tau,i})
\nabla\varphi_i+(a-\bar{a})(1-\eta)\nabla\bar{u}+\frac{1}{\tau}\psi_{\tau,i}\varphi_i \Big].
\end{aligned}
\end{equation}
Then, we obtain that
\begin{equation}\label{f:7.13}
\begin{aligned}
\|(\frac{\bar{u}_{<}}{\sqrt{T}},
(e_i&+\nabla\phi_{\tau,i})
\varphi_{i,<})
\|_{q,s_4,\Omega}
\lesssim^{\eqref{pri:*2*}}
\|(\frac{\bar{u}_{<}}{\sqrt{T}},
\nabla\bar{u}_{<})
\|_{q,s_3,\Omega} \\
&\lesssim^{\eqref{pri:A-2}}
\|\nabla w\|_{q,s_3,\Omega} + \|(\phi_{\tau,i},\sigma_{\tau,i},\frac{\psi_{\tau,i}}{\sqrt{\tau}})
\big(\frac{\eta\partial_i\bar{u}}{\sqrt{\tau}},
\eta\nabla\partial_i\bar{u}\big)\|_{q,s_3,\Omega}
 + \|(1-\eta)\nabla\bar{u}\|_{q,s_3,\Omega}\\
&\lesssim^{\eqref{pri:*3*}}
\|\nabla w\|_{q,s_3,\Omega} + \mu_d(\sqrt{T})
\|(\frac{\eta\nabla\bar{u}}{\sqrt{\tau}},
\eta\nabla^2\bar{u},
(1-\eta)\nabla\bar{u}
)\|_{q,s_1,\Omega}\\
&\lesssim^{\eqref{f:7.12},\eqref{pri:7.4},\eqref{pri:5.0}}
\mu_d(\sqrt{T})C_{q,s_3,s_2}(\tau)\tau^{-\frac{1}{2}}
C_{q,s_1,r}^2(T)
\|(f,\frac{g}{\sqrt{T}})\|_{q,r,\Omega}.
\end{aligned}
\end{equation}

\noindent
\textbf{Step 3.} In view of Lemmas $\ref{lemma:*3*}$ and $\ref{lemma:7.0}$,  we obtain that
\begin{equation}\label{f:7.14}
\begin{aligned}
&\big\|\phi_{\tau,i}(\frac{\eta\partial_i\bar{u}}{\sqrt{T}},
\eta\nabla\partial_i\bar{u})
\big\|_{q,s_2,\Omega} +
\big\|\phi_{\tau,i}(1-\eta)\partial_i\bar{u}
\big\|_{q,s_2,\Omega}\\
&\lesssim^{\eqref{pri:*3*}}
\mu_d(\sqrt{T})
\|(\frac{\eta\nabla\bar{u}}{\sqrt{\tau}},
\eta\nabla^2\bar{u},
\nabla\bar{u}
)\|_{q,s_1,\Omega}
\lesssim^{\eqref{pri:7.4},\eqref{pri:5.0}}
\mu_d(\sqrt{T})\tau^{-\frac{1}{2}}
C_{q,s_1,r}^2(T)
\|(f,\frac{g}{\sqrt{T}})\|_{q,r,\Omega}.
\end{aligned}
\end{equation}

Consequently, the estimates $\eqref{f:7.12}$, $\eqref{f:7.13}$, and
$\eqref{f:7.14}$ leads to the desired estimate $\eqref{pri:7.2b}$.

\medskip
\noindent
\textbf{Part III. Arguments for the estimate $\eqref{pri:7.3a}$.}
Recall that $1<r_2<r_1<r$. For any $\omega\in A_q$, it suffices to show the estimates on
\begin{equation*}
\big\|(\frac{u_{>}}{\sqrt{T}},\nabla u_{>});\omega^{\frac{1}{q}}\big\|_{q,r_1,\Omega},
\qquad
\big\|(\frac{\bar{u}_{>}}{\sqrt{T}},(e_i+\nabla\phi_{\tau,i})
\varphi_{i,>});\omega^{\frac{1}{q}}
\big\|_{q,r_2,\Omega},
\end{equation*}
respectively. The arguments is in parallel with those given for $\eqref{pri:7.2a}$, and we start from
\begin{equation}\label{f:7.15}
\big\|(\frac{u_{>}}{\sqrt{T}},\nabla u_{>});\omega^{\frac{1}{q}}\big\|_{q,r_1,\Omega}
\leq \big\|(\frac{u_{>}}{\sqrt{\tau}},\nabla u_{>});\omega^{\frac{1}{q}}\big\|_{q,r_1,\Omega}
\leq^{\eqref{pri:5.0-w}}
\tilde{C}_{q,r_1,r}(\tau)
\big\|(\frac{g}{\sqrt{T}},f);\omega^{\frac{1}{q}}
\big\|_{q,r,\Omega}.
\end{equation}

Then, in view of $\eqref{pde:7.6}$, we obtain that
\begin{equation}\label{f:7.16}
\begin{aligned}
\big\|(\frac{\bar{u}_{>}}{\sqrt{T}},(e_i
&+\nabla\phi_{\tau,i})
\varphi_{i,>});\omega^{\frac{1}{q}}
\big\|_{q,r_2,\Omega}
\lesssim^{\eqref{pri:*2*}}
\big\|(\frac{\bar{u}_{>}}{\sqrt{T}},\eta\nabla \bar{u}_{>});\omega^{\frac{1}{q}}\big\|_{q,r_1,\Omega}\\
&\lesssim_{\lambda,d,\partial\Omega,q,r_{1},[\omega]_{A_q}}^{\eqref{pri:A-3}}
\big\|(\frac{g}{\sqrt{T}},f,\nabla u_{>});\omega^{\frac{1}{q}}\big\|_{q,r_{1},\Omega}
 \lesssim^{\eqref{f:7.15}} \tilde{C}_{q,r_1,r}(\tau)
\big\|(\frac{g}{\sqrt{T}},f);\omega^{\frac{1}{q}}\big\|_{q,r,\Omega}.
\end{aligned}
\end{equation}

Hence, the desired estimate $\eqref{pri:7.3a}$ follows from
$\eqref{f:7.15}$ and $\eqref{f:7.16}$.

\medskip
\noindent
\textbf{Part IV. Arguments for the estimate $\eqref{pri:7.3b}$.}
Similar to the computations given in Part II, we
recall $1<s_4<s_3<s_2<s_1<r$. For any $\omega\in A_q$, it suffices to show the estimates on
\begin{equation*}
\big\|(\frac{w}{\sqrt{T}},\nabla w);\omega^{\frac{1}{q}}
\big\|_{q,s_3,\Omega},
~
\big\|(\frac{\bar{u}_{<}}{\sqrt{T}},
(e_i+\nabla\phi_{\tau,i})\varphi_{i,<});\omega^{\frac{1}{q}}
\big\|_{q,s_4,\Omega},
~
\big\|\phi_{\tau,i}(\frac{\varphi_i}{\sqrt{T}},
\nabla\varphi_{i});\omega^{\frac{1}{q}}
\big\|_{q,s_2,\Omega},
\end{equation*}
respectively. By the same token (i.e., the corresponding weighted estimate replacing the original one), we start from
\begin{equation}\label{f:7.17}
\begin{aligned}
&\|(\frac{w}{\sqrt{T}},\nabla w);\omega^{\frac{1}{q}}
\|_{q,s_3,\Omega}\\
&\lesssim^{\eqref{pri:5.0-w}} \tilde{C}_{q,s_3,s_2}(\tau)
\Big\{\|(\phi_{\tau,i},\sigma_{\tau,i},\frac{\psi_{\tau,i}}{\sqrt{\tau}})(\frac{\varphi_i}{\sqrt{\tau}},\nabla\varphi_i)
;\omega^{\frac{1}{q}}\|_{q,s_2,\Omega}
+ \|(1-\eta)\nabla\bar{u}
;\omega^{\frac{1}{q}}\|_{q,s_2,\Omega}\Big\}\\
&\lesssim^{\eqref{pri:*3*}} \mu_d(\sqrt{T})C_{q,s_3,s_2}(\tau)
\big\|
(\frac{\eta\nabla\bar{u}}{\sqrt{\tau}},
\eta\nabla^2\bar{u},
(1-\eta)\nabla\bar{u}
);\omega^{\frac{1}{q}}
\big\|_{q,s_1,\Omega}\\
%&\lesssim^{}
%\mu_d(\sqrt{T})\tilde{C}_{q,s_3,s_2}(\tau)\sqrt{T}\|\frac{\nabla u}{\sqrt{\tau}};\omega^{\frac{1}{q}}\|_{q,s_1,\Omega}\\
&\lesssim^{\eqref{pri:7.4-w},\eqref{pri:5.0-w}}
\mu_d(\sqrt{T})\tilde{C}_{q,s_3,s_2}(\tau)\tau^{-\frac{1}{2}}
\tilde{C}_{q,s_1,r}^2(T)
\|(f,\frac{g}{\sqrt{T}});\omega^{\frac{1}{q}}\|_{q,r,\Omega}.
\end{aligned}
\end{equation}

Then, in view of $\eqref{pde:7.7}$, we further obtain that
\begin{equation}\label{f:7.18}
\begin{aligned}
&\|(\frac{\bar{u}_{<}}{\sqrt{T}},
(e_i
+\nabla\phi_{i})
\varphi_{i,<});\omega^{\frac{1}{q}}
\|_{q,s_4,\Omega}
\lesssim^{\eqref{pri:*2*}}
\|(\frac{\bar{u}_{<}}{\sqrt{T}},
\nabla\bar{u}_{<});\omega^{\frac{1}{q}}
\|_{q,s_3,\Omega} \\
&\lesssim^{\eqref{pri:A-3}}
\|\nabla w;\omega^{\frac{1}{q}}\|_{q,s_3,\Omega} + \|(\phi_{\tau,i},\sigma_{\tau,i},\frac{\psi_{\tau,i}}{\sqrt{\tau}})\big(\frac{\eta\partial_i\bar{u}}{\sqrt{\tau}},
\eta\nabla\partial_i\bar{u}\big);\omega^{\frac{1}{q}}\|_{q,s_3,\Omega}
 + \|(1-\eta)\nabla\bar{u};\omega^{\frac{1}{q}}\|_{q,s_3,\Omega}\\
&\lesssim^{\eqref{pri:*3*}}
\|\nabla w;\omega^{\frac{1}{q}}\|_{q,s_3,\Omega} + \mu_d(\sqrt{T})
\|(\frac{\eta\nabla\bar{u}}{\sqrt{\tau}},
\eta\nabla^2\bar{u},
(1-\eta)\nabla\bar{u}
);\omega^{\frac{1}{q}}\|_{q,s_1,\Omega}\\
&\lesssim^{\eqref{f:7.17},\eqref{pri:7.4-w},\eqref{pri:5.0-w}}
\mu_d(\sqrt{T})\tilde{C}_{q,s_3,s_2}(\tau)\tau^{-\frac{1}{2}}
\tilde{C}_{q,s_1,r}^2(T)
\|(f,\frac{g}{\sqrt{T}})
;\omega^{\frac{1}{q}}\|_{q,r,\Omega}.
\end{aligned}
\end{equation}

Finally, we can derive that
\begin{equation}\label{f:7.19}
\begin{aligned}
\big\|\phi_{\tau,i}(\frac{\varphi_i}{\sqrt{T}},
\nabla\varphi_i);\omega^{\frac{1}{q}}
\big\|_{q,s_2,\Omega}
&\lesssim^{\eqref{pri:*3*}}
\mu_d(\sqrt{T})
\|(\frac{\eta\nabla\bar{u}}{\sqrt{\tau}},
\eta\nabla^2\bar{u},
(1-\eta)\nabla\bar{u}
);\omega^{\frac{1}{q}}\|_{q,s_1,\Omega}\\
&\lesssim^{\eqref{pri:7.4-w},\eqref{pri:5.0-w}}
\mu_d(\sqrt{T})\tau^{-\frac{1}{2}}
C_{q,s_1,r}^2(T)
\|(f,\frac{g}{\sqrt{T}});\omega^{\frac{1}{q}}
\|_{q,r,\Omega}.
\end{aligned}
\end{equation}

Hence, the desired estimate $\eqref{pri:7.3b}$ follows from
the estimates $\eqref{f:7.17}$, $\eqref{f:7.18}$, and $\eqref{f:7.19}$.
The proof is complete.
\qed

\subsection{Real $\&$ complex interpolations}\label{subsec:4.2}

\begin{lemma}\label{lemma:7.2}
Let $1<s_4<s_3<s_2<s_1<r$, $1<r_2<r_1<r$, and $\theta\in(0,1/3)$.
Assume that there exists a constant $\Lambda$ such that
\begin{equation}\label{pre-1}
C_{q,r_1,r}(\tau)
+ C_{q,s_3,s_2}(\tau)
\leq \Lambda \tau^{\frac{1}{2}\theta}
\qquad \forall\tau\geq 1.
\end{equation}
Then, we derive that
\begin{equation}\label{pri:7.5}
C_{q,s_1,r}(T)\lesssim_{\lambda,d,M_0} 1
\qquad\forall T\geq 1.
\end{equation}
Similarly, if we replace the assumption $\eqref{pre-1}$ with
\begin{equation}\label{pre-1-w}
\tilde{C}_{q,r_1,r}(\tau)
+ \tilde{C}_{q,s_3,s_2}(\tau)
\leq \Lambda \tau^{\frac{1}{2}\theta}
\qquad \forall\tau\geq 1,
\end{equation}
one can also infer that
\begin{equation}\label{pri:7.5-w}
\tilde{C}_{q,s_1,r}(T)\lesssim_{\lambda,d,M_0} 1
\qquad\forall T\geq 1.
\end{equation}
\end{lemma}

\begin{proof}
The proof relies on a real interpolation argument known as K-method
\cite[pp.209]{Adams-Fournier03}, and we will employ an important fact that
$\|\cdot\|_{\theta,\infty,K}=\|\cdot\|_{L^{s,\infty}}$ with
$\frac{1}{s}=\frac{1-\theta}{r_2}+\frac{\theta}{s_4}$
and $\theta\in(0,1)$ (see \cite[Corollary 7.27]{Adams-Fournier03}),
as well as, the relationship $\|\cdot\|_{L^{s_1}}
\lesssim \|\cdot\|_{L^{s,\infty}}\lesssim \|\cdot\|_{L^{s}}$
(see \cite[pp.13]{Grafakos08}) with respect to a finite measure.
(See the relationship between the exponents in Figure $\ref{fig}$.)
Once we have established the proof of $\eqref{pri:7.5}$, the proof of $\eqref{pri:7.5-w}$ is obvious, and we don't reproduce it here.
The following is devoted to showing the estimate $\eqref{pri:7.5}$,
and its proof is divided into two steps.

\begin{figure}[htbp] %H为当前位置，!htb为忽略美学标准，htbp为浮动图形
\centering %图片居中
\includegraphics[width=0.8\textwidth]{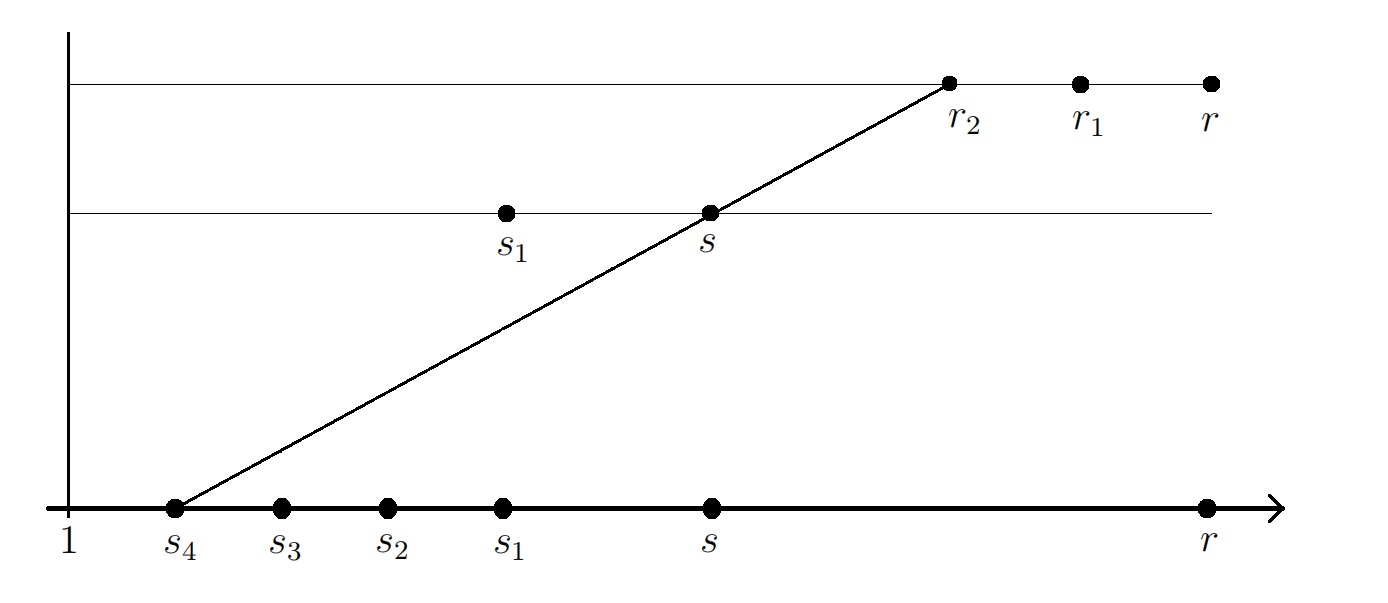} %插入图片，[]中设置图片大小，{}中是图片文件名
\caption{Exponents in the proof} %最终文档中希望显示的图片标题
\label{fig} %用于文内引用的标签
\end{figure}

\noindent
\textbf{Step 1}. Show the estimate $\eqref{pri:7.5}$.
By K-method, for any $\mu>0$ and some $\theta\in(0,1)$,
we can obtain
\begin{equation}\label{f:7.21}
\big\|(\frac{u}{\sqrt{T}},\nabla u)\big\|_{q,s_1,\Omega}
\lesssim \sup_{\mu>0}\inf_{u=v_{>}+v_{<}\atop
\nabla u = h_{>}+h_{<}}\Big\{
\mu^{-\theta}\big\|(\frac{v_{>}}{\sqrt{T}},h_{>})
\big\|_{q,r_2,\Omega}
+\mu^{1-\theta}
\big\|(\frac{v_<}{\sqrt{T}},h_{<})\big\|_{q,s_4,\Omega}\Big\},
\end{equation}
by a splitting $(u,\nabla u)=(v_{>},h_{>})+(v_{<},h_{<})$ such that
\begin{equation}\label{f:7.20}
\big\|(\frac{u}{\sqrt{T}},\nabla u)\big\|_{q,s_1,\Omega}
\leq  \Lambda C_{q,s_1,r}^{3\theta}(T)
\big\|(\frac{g}{\sqrt{T}},f)
\big\|_{q,r,\Omega}.
\end{equation}
Admitting the above estimate holds for a while, this further implies that
\begin{equation*}
C_{q,s_1,r}(T)
\leq \Lambda C_{q,s_1,r}^{3\theta}(T),
\end{equation*}
which finally leads to the stated estimate $\eqref{pri:7.5}$,
provided $\theta\in(0,1/3)$. Thus,
it is reduced to establish the estimate $\eqref{f:7.20}$.

\medskip
\noindent
\textbf{Step 2}. Show the estimate $\eqref{f:7.20}$. Define
\begin{equation*}
  \mu:=\frac{\sqrt{\tau}}{C_{q,s_1,r}^{2}(T)}.
\end{equation*}
Then the right-hand side of $\eqref{f:7.21}$ turns into
\begin{equation}\label{f:7.22}
 C_{q,s_1,r}^{2\theta}(T)
 \sup_{\tau>0}\Big\{\underbrace{\tau^{-\frac{\theta}{2}}
 \big\|(\frac{v_{>}}{\sqrt{T}},h_{>})\big\|_{q,r_2,\Omega}
 +\frac{\tau^{\frac{1}{2}(1-\theta)}}{C_{q,s_1,r}^2(T)}
 \big\|(\frac{v_{<}}{\sqrt{T}}, h_{<})\big\|_{q,s_4,\Omega}}_{R_{\tau}}
 \Big\},
\end{equation}
and quantifying the term $R_{\tau}$ is divided into three cases.

(1). For the case $0<\tau<1$. We take $(v_{>},h_{>})=(0,0)$ and
$(v_{<},h_{<}) = (u,\nabla u)$. Then we have
\begin{equation*}
 \|(\frac{v_{<}}{\sqrt{T}},h_{<})\|_{q,s_4,\Omega}
 \leq
 \|(\frac{u}{\sqrt{T}},\nabla u)\|_{q,s_1,\Omega}
 \leq^{\eqref{pri:5.0}} C_{q,s_1,r}(T)
 \|(\frac{g}{\sqrt{T}},f)\|_{q,r,\Omega}.
\end{equation*}
Plugging this into back into the right-hand side of $\eqref{f:7.22}$, we have
\begin{equation}\label{f:7.23}
R_{\tau}\leq \|(\frac{g}{\sqrt{T}},f)\|_{q,r,\Omega},
\end{equation}
where we employ the fact that $C_{q,s_1,r}(T)\geq 1$.

(2). For the case $\tau\geq T$. We take $(v_{>},h_{>})=(u,\nabla u)$ and
$(v_{<},h_{<})=(0,0)$, and
\begin{equation*}
\begin{aligned}
 \|(\frac{v_{>}}{\sqrt{T}},h_{>})\|_{q,r_2,\Omega}
& \leq
 \|(\frac{u}{\sqrt{T}},\nabla u)\|_{q,s_1,\Omega} \\
&  \leq^{\eqref{pri:5.0}} C_{q,s_1,r}(T)
 \|(\frac{g}{\sqrt{T}},f)\|_{q,r,\Omega}
\leq^{\eqref{pre-1}}\Lambda\tau^{\frac{1}{2}\theta}
\|(\frac{g}{\sqrt{T}},f)
\|_{q,r,\Omega},
\end{aligned}
\end{equation*}
where we employ the fact that $C_{q,s_1,r}(T)\leq  C_{q,s_1,r}(\tau)$
in the case of $\tau\geq T$.
Then, plugging this into back into the right-hand side of $\eqref{f:7.22}$, we have
\begin{equation}\label{f:7.24}
  R_{\tau}\leq  \Lambda
  \|(\frac{g}{\sqrt{T}},f)\|_{q,r,\Omega}.
\end{equation}

(3). For the case $1\leq \tau< T$. On account of Lemma $\ref{lemma:7.0}$, we have
\begin{equation*}
\begin{aligned}
\|(\frac{v_{>}}{\sqrt{T}},h_{>})\|_{q,r_2,\Omega}
&\lesssim^{\eqref{pri:7.2a}}
 C_{q,r_1,r}(\tau)
\big\|(\frac{g}{\sqrt{T}},f)\big\|_{q,r,\Omega}
\lesssim^{\eqref{pre-1}}
\Lambda\tau^{\frac{1}{2}\theta}
\big\|(\frac{g}{\sqrt{T}},f)\big\|_{q,r,\Omega},
\end{aligned}
\end{equation*}
and
\begin{equation*}
\begin{aligned}
\big\|(\frac{v_{<}}{\sqrt{T}},h_{<})\big\|_{q,s_4,\Omega}
& \lesssim^{\eqref{pri:7.2b}}
\mu_d(\sqrt{T})
C_{q,s_3,s_2}(\tau)
\tau^{-\frac{1}{2}}
C_{q,s_1,r}^2(T)
\big\|(\frac{g}{\sqrt{T}},f)\big\|_{q,r,\Omega}\\
& \lesssim^{\eqref{pre-1}}
\mu_d(\sqrt{T})
\Lambda\tau^{-\frac{1}{2}(1-\theta)}
C_{q,s_1,r}^2(T)
 \big\|(f,\frac{g}{\sqrt{T}})
 \big\|_{q,r,\Omega}.
\end{aligned}
\end{equation*}
Plugging the above two estimates
back into the right-hand side of $\eqref{f:7.22}$, we have
\begin{equation}\label{f:7.25}
  R_{\tau}\lesssim_{\lambda,d,M_0}
\mu_d(\sqrt{T})
\Lambda
 \big\|(f,\frac{g}{\sqrt{T}})
 \big\|_{q,r,\Omega}.
\end{equation}

Combining the three cases $\eqref{f:7.23}$, $\eqref{f:7.24}$ and $\eqref{f:7.25}$, we consequently derive that
\begin{equation*}
\big\|(\frac{u}{\sqrt{T}},\nabla u)\big\|_{L^{s,\infty}(\Omega)}
\lesssim
\mu_d(\sqrt{T})\Lambda
C_{q,s_1,r}^{2\theta}(T)
\big\|(\frac{g}{\sqrt{T}},f)
\big\|_{q,r,\Omega}
\lesssim \Lambda
C_{q,s_1,r}^{3\theta}(T)
\big\|(\frac{g}{\sqrt{T}},f)
\big\|_{q,r,\Omega}.
\end{equation*}
which gives the stated estimate $\eqref{f:7.20}$.
This completes the whole proof.
\end{proof}

\begin{lemma}\label{lemma:5.1}
Let $\tau\geq 1$.
Assume $(s_1,s),(r_1,r)\in[2,\infty)^2$, satisfying
\begin{equation}\label{precondition-2}
(1).~\left\{\begin{aligned}
&s_1\leq r_1;\\
&s\leq r,
\end{aligned}\right.
\quad\text{and}\quad (2).~
\frac{1}{r_1}-\frac{1}{r}>\frac{1}{s_1}-\frac{1}{s}>0.
\end{equation}
Then, for any $\vartheta\in (0,1)$ satisfying
\begin{equation}\label{precondition-3}
\vartheta>\max\big\{1-\frac{s}{r},1-\frac{s_1}{r_1}\big\},
\end{equation}
%and for any $\tilde{r}\in (r,\infty)$,
there exists a constant
$\Lambda$
%depending on \textcolor[rgb]{1.00,0.00,0.00}{$\lambda,d,M_0,\tilde{r},\vartheta$ and
%$q$},
such that
\begin{equation}\label{pri:5.7}
C_{q,r_1,r}(\tau)\leq
\Lambda C_{q,s_1,s}^{1-\vartheta}(\tau)
\tau^{\frac{\vartheta d}{2}}.
\end{equation}
Also, there analogously exists a constant
$\tilde{\Lambda}$
%depending on \textcolor[rgb]{1.00,0.00,0.00}{$\lambda,d,M_0,\tilde{r},\vartheta,q$ and
%$[\omega]_{A_q}$},
such that
\begin{equation}\label{pri:5.7-w}
\tilde{C}_{q,r_1,r}(\tau)\leq
\tilde{\Lambda} \tilde{C}_{q,s_1,s}^{1-\vartheta}(\tau)
\tau^{\frac{\vartheta d}{2}}.
\end{equation}
\end{lemma}

\begin{proof}
In view of $\eqref{precondition-2}$, it concludes that
$1-\frac{s}{r},1-\frac{s_1}{r_1}\in(0,1)$. Therefore, the set of $\vartheta$ satisfying $\eqref{precondition-3}$ is not empty.
Then, one can find a pair $(\tilde{r}_1,\tilde{r})\in[2,\infty)^2$ such that
\begin{equation*}
 \frac{1}{r_1} = \frac{1-\vartheta}{s_1} + \frac{\vartheta}{\tilde{r}_1};
 \qquad
 \frac{1}{r} = \frac{1-\vartheta}{s} + \frac{\vartheta}{\tilde{r}},
\end{equation*}
and we can get that $\frac{1}{\tilde{r}}=\frac{1}{\vartheta}(
\frac{1}{r}-\frac{1}{s})+\frac{1}{s}$ and
$\frac{1}{\tilde{r}_1}=\frac{1}{\vartheta}(
\frac{1}{r_1}-\frac{1}{s_1})+\frac{1}{s_1}$.
Moreover, it is not hard to see that
\begin{equation}\label{f:7.26}
\left\{\begin{aligned}
&r_1<\tilde{r}_1;\\
&r<\tilde{r},
\end{aligned}\right.
\quad\text{and}\quad\tilde{r}_1<\tilde{r}.
\end{equation}
On account of Theorem $\ref{thm:2}$,
the later one in $\eqref{f:7.26}$ further implies that
\begin{equation*}
 C_{q,\tilde{r}_1,\tilde{r}}(\tau)\leq^{\eqref{rv-bound-1}}
 C_{\lambda,d,M_0,q,\tilde{r}}
\tau^{\frac{d}{4}}
\leq
 C_{\lambda,d,M_0,q,\tilde{r}}
\tau^{\frac{d}{2}};
\qquad
\tilde{C}_{q,\tilde{r}_1,\tilde{r}}(\tau)
\leq^{\eqref{rv-bound-1}}
\tilde{C}_{\lambda,d,M_0,p,q,[\omega]_{A_q}}
\tau^{\frac{d}{2}}.
\end{equation*}

Thus, by using the complex interpolation, we consequently have
\begin{equation*}
C_{q,r_1,r}(\tau)\leq
C_{q,s_1,s}^{1-\vartheta}(\tau)
C_{q,\tilde{r}_1,\tilde{r}}^{\vartheta}(\tau)
\leq C_{\lambda,d,M_0,q,\tilde{r}}^{\vartheta}
C_{q,s_1,s}^{1-\vartheta}(\tau)
\tau^{\frac{\vartheta d}{2}},
\end{equation*}
which gives the desired estimate $\eqref{pri:5.7}$ by setting
$\Lambda:=C_{\lambda,d,M_0,q,\tilde{r}}^{\vartheta}$.
By the same token, one can obtain the stated estimate
$\eqref{pri:5.7-w}$.
\end{proof}

%\noindent
%\textcolor[rgb]{1.00,0.00,0.00}{\textbf{Remark:} $\vartheta\to 0$, we have $s\to r$ and $s_1\to r_1$, which means
%the slope of the line is higher and higher. It seems that there is no
%relationship about the wide between $r_1$ and $r$.}

\subsection{Bootstrap arguments}

\begin{figure}[htbp] %H为当前位置，!htb为忽略美学标准，htbp为浮动图形
\centering %图片居中
\includegraphics[width=0.8\textwidth]{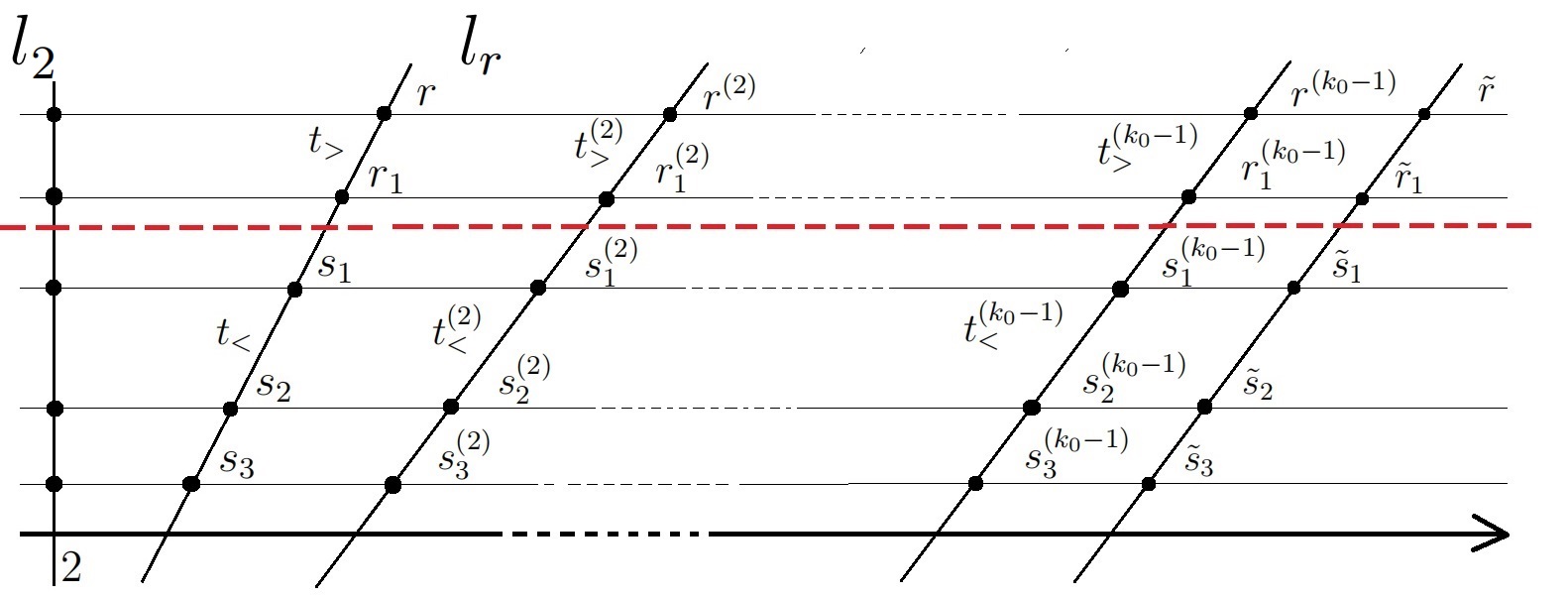} %插入图片，[]中设置图片大小，{}中是图片文件名
\caption{Finite-step iteration} %最终文档中希望显示的图片标题
\label{figure} %用于文内引用的标签
\end{figure}

\noindent
\textbf{Proof of Theorem $\ref{P:5}$.}
Since one can repeat the same arguments for $\eqref{pri:7.1}$ to get $\eqref{pri:7.1-w}$, it suffices to show the demonstration of
$\eqref{pri:7.1}$, and
the proof is divided into two steps by induction. The main idea is the bootstrap
argument (based upon the real and complex interpolations stated in Subsection $\ref{subsec:4.2}$),
which is shown in Figure $\ref{figure}$.

\noindent
\textbf{Step 1.}
For any given $\tilde{r}\in[2,\infty)$\footnote{It is fine
to find the corresponding $\tilde{r}_1$, $\tilde{s}_3$, $\tilde{s}_2$,
and $\tilde{s}_1$ according to Figure $\ref{fig}$.},
and $\theta\in(0,1/3)$ given as in Lemma $\ref{lemma:7.2}$,  let $\vartheta=\theta/d$. For the ease of the statement, we introduce the following notation:
\begin{equation*}
\left\{\begin{aligned}
&  t_{>}:=(r_1,r); \\
& \tilde{t}_{>}:=(\tilde{r}_1,\tilde{r}).
\end{aligned}\right.
\qquad
\left\{\begin{aligned}
&t_{<}:=(s_3,s_2);\\
&\tilde{t}_{<}:=(\tilde{s}_3,\tilde{s}_2),
\end{aligned}\right.
\end{equation*}
and they formally satisfy the following relationship:
\begin{equation}\label{formula-1}
\frac{1}{t_{>}}
=\frac{1-\vartheta}{m} + \frac{\vartheta}{\tilde{t}_{>}};
 \quad\text{and}\quad
\frac{1}{t_{<}}=
\frac{1-\vartheta}{n}+\frac{\vartheta}{\tilde{t}_{<}},
\end{equation}
which means that their components obey the corresponding formula.

Now, we can fix $(t_{>},t_{<})$ by taking
$(m,n)=(2,2)$, i.e.,
\begin{equation*}
\frac{1}{t_{>}}
=\frac{1-\vartheta}{2} + \frac{\vartheta}{\tilde{t}_{>}};
 \quad\text{and}\quad
\frac{1}{t_{<}}=
\frac{1-\vartheta}{2}+\frac{\vartheta}{\tilde{t}_{<}},
\end{equation*}
In view of Lemma $\ref{lemma:5.1}$, we obtain
\begin{equation*}
   C_{q,r_1,r}(\tau)
   +C_{q,s_3,s_2}(\tau)
   \lesssim \tau^{\frac{1}{2}\theta},
\end{equation*}
and therefore it follows from Lemma $\ref{lemma:7.2}$ that
\begin{equation*}
 C_{q,s_1,r}(T')\lesssim 1
\qquad\forall T'\geq 1.
\end{equation*}

%\noindent
%\textbf{Step 1.}
%To do so, for any $\vartheta\in(0,1)$,
%it follows from
%interpolation inequality that
%\begin{equation}\label{interplay}
% C_{q_1}(\tau)\lesssim  C_2^{1-\frac{\vartheta}{2}}C_{p}^{\vartheta}(\tau)
% \qquad \forall\tau\geq 1, \quad 1/q_1=(1-\vartheta)/2+\vartheta/p.
%\end{equation}
%So, we may prefer $\vartheta\in(0,1)$ small enough and, therefore determine the index $q_1$, such that
%\begin{equation*}
%   C_{q_1}(\tau)\lesssim \sqrt{\tau}^{\frac{\theta}{2}}.
%\end{equation*}
%Clearly, $q_1>2$. In view of $\eqref{pri:3.6x}$, we have
%\begin{equation*}
%  \bar{C}_{q_1}(\sqrt{T'}) \lesssim
%C_{q_1}(\sqrt{\tau})
% +C_{q_1}(\sqrt{\tau})\sqrt{\tau}^{-\frac{1}{2}}
% \bar{C}_{q_1}(\sqrt{T'}).
%\end{equation*}
%This gives
%\begin{equation}\label{f:3.8x}
%  \bar{C}_{q_1}(\sqrt{T'}) \lesssim 1.
%\end{equation}

\noindent
\textbf{Step 2.} Let $(s_1,r)$ be fixed as in Step 1.
Let $k\geq 0$ be any integer, and we set
\begin{equation}\label{iteration-1}
  (m_{k},n_{k}) := (s_1^{(k)},r^{(k)});
  \quad (m_{1},n_{1}) := (s_1,r);
  \quad (m_{0},n_{0}) := (2,2).
\end{equation}
By the relationship $\eqref{formula-1}$, we have the following iteration formula:
\begin{equation}\label{iteration-2}
\frac{1}{t_{>}^{(k)}}
=\frac{1-\vartheta}{m_{k-1}} + \frac{\vartheta}{\tilde{t}_{>}};
 \quad\text{and}\quad
\frac{1}{t_{<}^{(k)}}=
\frac{1-\vartheta}{n_{k-1}}+\frac{\vartheta}{\tilde{t}_{<}},
\end{equation}
where $t_{>}^{(k)}:=((r_1^{(k)},r^{(k)}))$ and
$t_{<}^{(k)}:=((s_3^{(k)},s_2^{(k)}))$ with
\begin{equation*}
(r_1^{(1)},r^{(1)}):= (r_1,r);
\quad
(s_3^{(1)},s_2^{(1)})
:= (s_3,s_2).
\end{equation*}

Thus,
if admitting $(m_{k-1},n_{k-1})$
has been derived for any $k\geq 2$,  by the formula $\eqref{iteration-2}$ and
Lemma $\ref{lemma:5.1}$, we have
\begin{equation}\label{formula-k}
   C_{q,r_1^{(k)},r^{(k)}}(\tau)
   +C_{q,s_3^{(k)},s_2^{(k)}}(\tau)
   \lesssim \tau^{\frac{1}{2}\theta}.
\end{equation}
Then, using $\eqref{formula-k}$ as the input in Lemma $\ref{lemma:7.2}$, we derive that
\begin{equation}\label{pri:3.15x}
 C_{q,s_1^{(k)},r^{(k)}}(T')
 \lesssim 1
\qquad\forall T'\geq 1.
\end{equation}
As a result, plugging the above result back into the notation $\eqref{iteration-1}$, we obtain $(m_{k},n_{k})$, which, as a new input in the formula $\eqref{iteration-2}$,
further leads to $(m_{k+1},n_{k+1})$ by repeating the above procedure.

Now, we can claim that there exists $k_0<\infty$ such  that $r^{(k_0)}>\tilde{r}$. By noting the iteration formula $\eqref{iteration-2}$, for any integer $k\geq 1$, we have
\begin{equation*}
 \frac{1}{r^{(k)}}\geq \frac{1-\vartheta}{s_1^{(k-1)}}
 > \frac{1-\vartheta}{r^{(k-1)}}>\frac{1}{2}(1-\vartheta)^{k}.
\end{equation*}
One may find it by
\begin{equation*}
  k_0:=\inf\Big\{k\in\mathbb{N}: k>\frac{\ln(1/\tilde{r})}{\ln((1-\vartheta)/2)}\Big\}.
\end{equation*}
That means for any $(\tilde{s}_1,\tilde{r})\in [2,\infty)^2$, we can get
the estimate
\begin{equation}\label{f:7.27}
 C_{q,\tilde{s}_1,\tilde{r}}(T')
 \lesssim 1
\qquad\forall T'\geq 1,
\end{equation}
merely by finite steps. For the case of $\tilde{r}\in(1,2]$, the corresponding estimate $\eqref{f:7.27}$ follows from a duality argument, and we have completed the whole arguments.
\qed

\begin{remark}\label{remark:4.1}
\emph{Although $k_0$ in the above proof depends on $\tilde{r}$ itself, we should mention that it is independent of $\tilde{r}-\tilde{s}_1$.
Speaking a bit more, based on Meyer-type estimates, the line $l_r$ in Fig.$\ref{figure}$ can be obtained by translating $l_2$ whenever $0<r-2\ll1$, whereas it is essential to discuss $l_r$ with
a slope strictly less than $\pi/2$ when $r$ is far from 2, and
the most important is that the above demonstration does not relies on the degree of the slope.}
\end{remark}

\bigskip
\noindent
\textbf{Proof of Theorem $\ref{thm:1}$ and $\ref{thm:1-w}$.}
Let $f$ be square-integrable.
To handle the equations $\eqref{pde:0}$, for any $T\geq 1$, we consider the massive version of the elliptic operator $\mathcal{L}$
as follows:
\begin{equation}\label{pde:0*}
\left\{\begin{aligned}
\frac{1}{T}u_{T}-\nabla\cdot a \nabla u_{T}
&= \nabla\cdot f
&\quad&\text{in}~~\Omega;\\
u &= 0
&\quad&\text{on}~~\partial\Omega.
\end{aligned}\right.
\end{equation}
Then, in view of Theorem $\ref{P:5}$, we obtain
\begin{subequations}
\begin{align}
& \Big(\int_{\tilde{\Omega}} \big\langle|(\nabla u_{T},\frac{u_{T}}{\sqrt{T}} )|^{p}\big\rangle^{\frac{q}{p}}\Big)^{\frac{1}{q}}
\lesssim_{\lambda,d,M_0,p,\bar{p},q}^{\eqref{pri:7.1}}
\Big(\int_{\tilde{\Omega}}
\big\langle|f|^{\bar{p}}
\big\rangle^{\frac{q}{\bar{p}}}\Big)^{\frac{1}{q}};
\label{f:6.1}\\
&\Big(\int_{\tilde{\Omega}} \big\langle|(\nabla u_{T},\frac{u_{T}}{\sqrt{T}} )|^{p}\big\rangle^{\frac{q}{p}}\omega\Big)^{\frac{1}{q}}
\lesssim_{\lambda,d,M_0,p,\bar{p},q,[\omega]_{A_q}}^{\eqref{pri:7.1-w}}
\Big(\int_{\tilde{\Omega}}
\big\langle|f|^{\bar{p}}
\big\rangle^{\frac{q}{\bar{p}}}\omega\Big)^{\frac{1}{q}}.
\label{f:6.1-w}
\end{align}
\end{subequations}
Since the multiplicative constant is independent of $T$, letting $T\to\infty$ (up to a subsequence),
there holds $\nabla u_{T} \rightharpoonup \nabla u$ weakly in $L^q(\Omega;L^p_{\langle\cdot\rangle})$, and $\nabla u$ satisfies
\begin{equation*}
 -\nabla\cdot a\nabla u = \nabla\cdot f \quad \text{in} \quad\Omega,
\end{equation*}
Furthermore, by weakly lower semi-continuity of the norm, there holds
\begin{equation*}
\begin{aligned}
&\Big(\int_{\tilde{\Omega}} \big\langle|\nabla u|^{p}\big\rangle^{\frac{q}{p}}\Big)^{\frac{1}{q}}
\lesssim_{\lambda,d,M_0,p,q}^{\eqref{f:6.1}}
\Big(\int_{\tilde{\Omega}}
\big\langle|f|^{\bar{p}}
\big\rangle^{\frac{q}{\bar{p}}}\Big)^{\frac{1}{q}};
\label{}\\
&\Big(\int_{\tilde{\Omega}} \big\langle|\nabla u|^{p}\big\rangle^{\frac{q}{p}}\omega\Big)^{\frac{1}{q}}
\lesssim_{\lambda,d,M_0,p,q}^{\eqref{f:6.1-w}}
\Big(\int_{\tilde{\Omega}}
\big\langle|f|^{\bar{p}}
\big\rangle^{\frac{q}{\bar{p}}}\omega\Big)^{\frac{1}{q}};
\end{aligned}
\end{equation*}
Since $u_{T}\in H_0^1(\Omega)$ is uniformly bounded with respect to $T$,
we can infer that $u_{T} \to 0$ strongly in $L^2(\partial\Omega)$
whenever $\Omega$ is bounded domain. By the uniqueness of
the weak limit, it concludes that $u = 0$ on $\partial\Omega$ in the sense of the trace. If $\Omega$ is an unbounded domain with a local compactness of $\partial\Omega$, one may find any ball $B_0$ with $x_{B_0}\in\partial\Omega$. Then, by the same token we can have that $u=0$ on $\partial\Omega\cap B_0$, and we can infer that
$u=0$ on $\partial\Omega$ due to the local compactness and
the diagonal process.
\qed

\section{Resolvent estimates}\label{sec:6}

\begin{theorem}[resolvent estimates in $L^2$]\label{thm:5.0}
Let $\Omega\subset\mathbb{R}^d$ with $d\geq 2$
be a Lipschitz domain. Let $\mu\in\Sigma_{\theta}$.
Suppose that the coefficient $a$ satisfies the uniform ellipticity condition $\eqref{a:1}$ and $\eqref{a:4}$.
Then, for any  $(\mathbf{f},\mathbf{g})\in L^2(\Omega;\mathbb{C}^d)\times L^2(\Omega;\mathbb{C})$, there exists the unique solution
$\mathbf{v}\in H_0^1(\Omega;\mathbb{C})$ satisfying the equation
$\eqref{pde:0-r}$. Moreover, if $\Omega\ni\{0\}$ is bounded, then we have
\begin{equation}\label{pri:5.4}
\begin{aligned}
(|\mu|+R_0^{-2})\|\mathbf{v}\|_{L^2(\Omega)}
+ (|\mu|+R_0^{-2})^{1/2}\|\nabla\mathbf{v}\|_{L^2(\Omega)}
\lesssim_{\lambda,\theta,d} \|\mathbf{g}\|_{L^2(\Omega)}
+(|\mu|+R_0^{-2})^{1/2}
\|\mathbf{f}\|_{L^2(\Omega)},
\end{aligned}
\end{equation}
and by setting $\omega_\mu:=\exp(-C_{d,\lambda,\theta}\sqrt{|\mu|}|x|)$
with sufficiently large constant $C_{d,\lambda,\theta}$, we have
\begin{equation}\label{pri:5.4-w}
|\mu|\|\mathbf{v}\omega_{\mu}\|_{L^2(\Omega)}
+ |\mu|^{1/2}\|\nabla\mathbf{v}\omega_{\mu}\|_{L^2(\Omega)}
\lesssim_{\lambda,\theta,d} \|\mathbf{g}\omega_{\mu}\|_{L^2(\Omega)}
+|\mu|^{1/2}
\|\mathbf{f}\omega_{\mu}\|_{L^2(\Omega)}.
\end{equation}
If $\Omega$ is unbounded, there still holds $\eqref{pri:5.4-w}$, as well as,
\begin{equation}\label{pri:5.5}
\begin{aligned}
&|\mu|\|\mathbf{v}\|_{L^2(\Omega)}
+ |\mu|^{1/2}\|\nabla\mathbf{v}\|_{L^2(\Omega)}
\lesssim_{\lambda,\theta,d} \|\mathbf{g}\|_{L^2(\Omega)}
+|\mu|^{1/2}
\|\mathbf{f}\|_{L^2(\Omega)}.
%&|\mu|\|\mathbf{v}\omega_{\mu}\|_{L^2(\Omega)}
%+ |\mu|^{1/2}\|\nabla\mathbf{v}\omega_{\mu}\|_{L^2(\Omega)}
%\lesssim_{\lambda,\theta,d} \|\mathbf{g}\omega_{\mu}\|_{L^2(\Omega)}
%+|\mu|^{1/2}
%\|\mathbf{f}\omega_{\mu}\|_{L^2(\Omega)}.
\end{aligned}
\end{equation}
In particular, if $\mathbf{g}$ and $\mathbf{f}$ are vanished, for any
ball with the properties: either
$B\subset\mathbb{R}^d$ with $x_B\in\partial\Omega$ or $3B\subset\Omega$,  the local solution $\mathbf{v}$ satisfies the Caccioppoli's inequality:
\begin{equation}\label{pri:5.6}
\left\{\begin{aligned}
&\|\mathbf{v}\|_{L^2(B\cap\Omega)}
\lesssim_{\lambda,\theta,d,k}(\sqrt{|\mu|}r_B+1)^{-k}
\|\mathbf{v}\|_{L^2(2B\cap\Omega)};\\
&\|\nabla\mathbf{v}\|_{L^2(B\cap\Omega)}
\lesssim_{\lambda,\theta,d}r_B^{-1}
\|\mathbf{v}\|_{L^2(2B\cap\Omega)}.
\end{aligned}\right.
\end{equation}
\end{theorem}

\begin{proof}
The idea of the proof is standard, and all the above estimates start from
the following energy identity $\eqref{energy}$. Except for Poincar\'e's inequality,
the case of the unbounded domain indeed follows from the same idea as
that given for the case of bounded domains, so we omit the details. For any $\mathbf{z}\in H_0^1(\Omega;\mathbb{C})$ being arbitrary test function, it follows that
\begin{equation}\label{energy}
\int_{\Omega}\mu\mathbf{v}\overline{\mathbf{z}}
+\int_{\Omega}a\nabla\mathbf{v}\cdot\overline{\nabla\mathbf{z}}
=\int_{\Omega}
\big(\mathbf{g}\overline{\mathbf{z}}
-\mathbf{f}\cdot\overline{\nabla\mathbf{z}}\big).
\end{equation}
Taking $\mathbf{z}=\mathbf{v}$ leads to the stated estimate $\eqref{pri:5.4}$ (and \eqref{pri:5.5}) by using the additional conditions $\eqref{a:4}$ and
$\eqref{a:5}$ (see e.g. \cite{Wei-Zhang14}), while taking $\mathbf{z}=\mathbf{v}\omega_\mu^2$
gives the desired estimate $\eqref{pri:5.4-w}$, where we take
$\omega_\mu:=\exp(-C_{d,\lambda,\theta}\sqrt{|\mu|}|x|)$ with
$C_{d,\lambda,\theta}$ being sufficiently large. Let $\eta\in C_0^1(2B)$
be a cut-off function with $\eta=1$ in $B$.
Then, the estimates $\eqref{pri:5.6}$ follow from
the standard arguments for Caccioppoli's inequality by additionally considering the conditions
$\eqref{a:4}$ and $\eqref{a:5}$, and we refer the reader to \cite[Lemma 2.1]{Shen95}
for some details. This completes the proof.
\end{proof}

\begin{remark}
\emph{To the authors' best knowledge, it's unclear whether
$\eqref{pri:5.4}$ can be extended to the weighted estimate
involving $\omega_\mu=\exp(-C_{d,\lambda,\theta}\sqrt{|\mu|}|x|)$
in the above proof for $\Omega\in\{0\}$, since
we don't have the desired form of Poincar\'e inequality with $\omega_\mu$ as the weight.}
\end{remark}

\begin{theorem}[quenched-type Calder\'on-Zygmund estimates]\label{thm:5.1}
Let $\hat{\Omega}\subset\mathbb{R}^d$ with $d\geq 2$ and $\hat{\Omega}\ni\{0\}$
be a bounded $C^1$ domain, and $\varepsilon\in(0,1]$.
Suppose that $\langle\cdot\rangle$ is stationary and satisfies spectral gap condition $\eqref{a:2}$,
and the (admissible) coefficient additionally satisfies the local regularity condition $\eqref{a:3}$.
Let $1<q<\bar{q}<\infty$ and $1/\bar{p}=1/\bar{q}+1/d$ if $\bar{q}>\frac{d}{d-1}$; $\bar{p}>1$ if $\bar{q}=\frac{d}{d-1}$;
$\bar{p}=1$ if $1<\bar{q}<\frac{d}{d-1}$.
Then, there exists a random variable $\mathcal{C}_{*,\hat{R}_0/\varepsilon}(0)$
(which is scaling-invariant) such that
\begin{equation}\label{f:5.5}
\big\langle\mathcal{C}_{*,\hat{R}_0/\varepsilon}(0)^{\beta}\big\rangle^{1/\beta}
\lesssim_{\lambda,\lambda_1,\lambda_2,d,M_0,\beta} 1
\qquad \forall\beta\in[1,\infty),
\end{equation}
and there holds the following estimate:
\begin{equation}\label{pri:5.1}
\Big(\dashint_{\hat{\Omega}}|\nabla u_\varepsilon|^q\Big)^{\frac{1}{q}}
\leq \mathcal{C}_{*,\hat{R}_0/\varepsilon}(0)\bigg\{
\Big(\dashint_{\hat{\Omega}}|f|^{\bar{q}}\Big)^{\frac{1}{\bar{q}}}
+\hat{R}_0\Big(\dashint_{\hat{\Omega}}|g|^{\bar{p}}\Big)^{\frac{1}{\bar{p}}}\bigg\},
\end{equation}
provided that $u_\varepsilon$, $f$ and $g$ are associated with
the equations
\begin{equation}\label{pde:5.1}
\left\{\begin{aligned}
-\nabla\cdot a(\cdot/\varepsilon)\nabla u_\varepsilon
&=\nabla\cdot f+g &\quad&\text{in}\quad\hat{\Omega};\\
u&=0&\quad&\text{on}~~\partial\hat{\Omega}.
\end{aligned}\right.
\end{equation}
\end{theorem}

\begin{proof}
To see the estimate $\eqref{pri:5.1}$, it is fine to establish the
corresponding estimate for $g=0$ at first (since the case of $g\not=0$
follows from the Sobolev embedding theorem).
Different from the quenched estimates as those established in the form of  $L^2$-norm, where the multiplicative constant is deterministic and
the diameter of the integral region involves randomness and therefore known as the \emph{minimum radius} in \cite{Gloria-Neukamm-Otto20} which indeed corresponds to the large-scale estimates,
the multiplicative constant in $\eqref{pri:5.1}$ is quite involved.

Roughly speaking, we can not obtain the multiplicative
constant $\mathcal{C}_{*,\hat{R}_0/\varepsilon}$ of $\eqref{pri:5.1}$
by one step, while we have to adopt the non-perturbation argument in the quenched sense to achieve it, analogical to that presented in Section $\ref{sec:4}$. Therefore, not surprising, we still appeal to the massive equation
$\eqref{pde:2}$ with $T>0$ as an auxiliary equation:
\begin{equation*}
\left\{\begin{aligned}
\frac{1}{T}u_{\varepsilon,T}-\nabla\cdot a(\cdot/\varepsilon)\nabla u_{\varepsilon,T}
&=\nabla\cdot f+g &\quad&\text{in}\quad\hat{\Omega};\\
u&=0&\quad&\text{on}~~\partial\hat{\Omega}.
\end{aligned}\right.
\end{equation*}
We outline the proof by two key steps, whose details will be given in a separated work.

\noindent
\textbf{Step 1.} Establish the suboptimal estimate as follows:
\begin{equation}\label{f:5.9}
\Big(\dashint_{\hat{\Omega}}\big|(\frac{u_{\varepsilon,T}}{\sqrt{T}},\nabla u_{\varepsilon,T})\big|^q\Big)^{\frac{1}{q}}
\lesssim \mathbb{X}_{*,\hat{R}_0/\varepsilon}(0)\big(\frac{\hat{R}_0}{\varepsilon}
\big)^{\kappa_d}
\Big(\dashint_{\hat{\Omega}}\big|(f,\frac{g}{\sqrt{T}})\big|^{q}\Big)^{\frac{1}{q}},
\end{equation}
which follows from small and large scale estimates, as well as,
the Lipschitz property of the minimal radius, where $\kappa_d$ is a constant depending only on $d$, and $\mathbb{X}_{*,\hat{R}_0/\varepsilon}$ relies on the coefficient and minimal radius polynomially at most.

\noindent
\textbf{Step 2.} From $\eqref{f:5.9}$, one can establish the optimal estimate
as follows:
\begin{equation}\label{f:5.10}
\Big(\dashint_{\hat{\Omega}}\big|(\frac{u_{\varepsilon,T}}{\sqrt{T}},\nabla u_{\varepsilon,T})\big|^q\Big)^{\frac{1}{q}}
\lesssim \mathcal{C}_{*,\hat{R}_0/\varepsilon}(0)
\Big(\dashint_{\hat{\Omega}}|(f,\frac{g}{\sqrt{T}})|^{\bar{q}}
\Big)^{\frac{1}{\bar{q}}}.
\end{equation}
By the non-perturbation argument,
there exists a stationary
random field
$\mathcal{C}_{*,\hat{R}_0/\varepsilon}$ such that
the desired estimate $\eqref{pri:5.1}$ finally holds,
%\begin{equation}
%\Big(\dashint_{\hat{\Omega}}|(\frac{u_\varepsilon}{\sqrt{T}},\nabla u_\varepsilon)|^q\Big)^{\frac{1}{q}}
%\lesssim \mathcal{C}_{*,\hat{R}_0/\varepsilon}(0)
%\Big(\dashint_{\hat{\Omega}}|(f,\frac{g}{\sqrt{T}})
%|^{\bar{q}}\Big)^{\frac{1}{\bar{q}}},
%\end{equation}
where $\bar{q}>q$ is arbitrary. As we mentioned in Remark $\ref{remark:4.1}$, to carry out the non-perturbation argument,
the slope (determined by $q$ and $\bar{q}$)
strictly less than $\pi/2$ (when $q$ is far from 2) is
very important, which is indeed guaranteed by the condition $\bar{q}>q$.
On the other hand, the multiplicative constant never depends on the quantity $\bar{q}-q$. Therefore, we indeed start from $\eqref{f:5.9}$ by modifying its right-hand side into $L^{\bar{q}}$-integrability with $\bar{q}>q$ arbitrary.
Consequently, one can derive the estimate $\eqref{pri:5.1}$
by letting $T\to\infty$ as we did in Theorem $\ref{thm:1}$. This ends the
whole arguments.
\end{proof}

\begin{remark}\label{remark:5.1}
\emph{Let $\varepsilon\in(0,1]$.
Let $\rho\in(0,1/2)$ and $q>2$, and we set $\bar{q}:=q/(1-q\rho)$
and $1/\bar{p}=1/\bar{q}+1/d$. It follows from the estimate $\eqref{pri:5.1}$ that
\begin{equation*}
 \|\nabla u_\varepsilon\|_{L^q(\hat{\Omega})}
\leq
\mathcal{C}_{*,\hat{R}_0/\varepsilon}(0)
\big(\frac{\hat{R}_0}{\varepsilon}
\big)^{d\rho}
\Big\{\|f\|_{L^{\bar{q}}(\hat{\Omega})}
+\|g\|_{L^{\bar{p}}(\hat{\Omega})}\Big\}.
\end{equation*}
Moreover, for $q\gg 2$, one can prefer $\rho = 1/q^2
\wedge\log_{(\hat{R}_0/\varepsilon)^d}2$ such that
\begin{equation}\label{pri:5.1*}
\|\nabla u_\varepsilon\|_{L^q(\hat{\Omega})}
\leq 2\mathcal{C}_{*,\hat{R}_0/\varepsilon}(0)
\Big\{\|f\|_{L^{\bar{q}}(\hat{\Omega})}
+\|g\|_{L^{\bar{p}}(\hat{\Omega})}\Big\},
\end{equation}
which is a good replacement of $\eqref{pri:5.1}$ to be applied later on.
In such the case, we can infer that $0<\bar{q}-q\ll 1$, which is indeed convenient
to product the same bootstrap process as we did in Lemma $\ref{lemma:ap-2}$.}
\end{remark}

\begin{corollary}[local estimates]
Let $q\gg 2$ and $\varepsilon\in(0,1]$,
and $\rho$ with $\bar{q}$ is given as in Remark $\ref{remark:5.1}$ and $1/\bar{p}=1/\bar{q}+1/d$.
Assume the same conditions as in Theorem $\ref{thm:5.1}$.
Let $u_\varepsilon$ and $f$ with $g$ be
associated by $\eqref{pde:5.1}$.
Then, for any ball $B$ with the property: $3B\subset\Omega$ or
$x_B\in\partial\Omega$, there exists a constant $\beta_d<\infty$, depending only on $d$, such that
we have the following local estimates:
\begin{equation}\label{pri:5.3}
 \Big(\dashint_{B\cap\Omega}|\nabla u_\varepsilon|^q\Big)^{\frac{1}{q}}
 \lesssim \mathcal{C}_{*,\hat{R}_0/\varepsilon}^{\beta_d}(0)\bigg\{
 \Big(\dashint_{2B\cap\Omega}|\nabla u_\varepsilon|^2\Big)^{\frac{1}{2}}
 + \Big(\dashint_{2B\cap\Omega}|f|^{\bar{q}}
 \Big)^{\frac{1}{\bar{q}}}
 + r_B\Big(\dashint_{2B\cap\Omega}|g|^{\bar{p}}
 \Big)^{\frac{1}{\bar{q}}}\bigg\}.
\end{equation}
In particular, assume that the coefficient additionally satisfies
$\eqref{a:4}$, and if $\mu\in\Sigma_{\theta}$ with
$\theta\in(0,\pi/2)$, $p>0$, and $\alpha>1$. Then,  there hold
\begin{equation}\label{pri:5.2}
\begin{aligned}
& \Big(\dashint_{B\cap\Omega}|\textbf{u}_\varepsilon|^{q}\Big)^{1/q}
 \lesssim_{\alpha,p,\theta} \mathcal{C}_{*,\hat{R}_0/\varepsilon}^{2\beta_d}(0)
 \Big(\dashint_{\alpha B\cap\Omega}|\textbf{u}_\varepsilon|^{p}\Big)^{1/p};\\
& \Big(\dashint_{B\cap\Omega}|\nabla\textbf{u}_\varepsilon|^{q}\Big)^{1/q}
 \lesssim_{\alpha,p,\theta} \mathcal{C}_{*,\hat{R}_0/\varepsilon}^{2\beta_d}(0)
 \Big(\dashint_{\alpha B\cap\Omega}|\nabla\textbf{u}_\varepsilon|^{p}\Big)^{1/p},
\end{aligned}
\end{equation}
provided that $\textbf{u}_\varepsilon$ is complex valued and
satisfies the equations
$\nabla\cdot a(\cdot/\varepsilon)\nabla \textbf{u}_\varepsilon=
\mu\textbf{u}_\varepsilon$ in $2B\cap\Omega$ with $\textbf{u}_\varepsilon
=0$ on $2B\cap\partial\Omega$.
Note that we can prefer $\beta_d:=[\frac{d}{2}]+1$ in $\eqref{pri:5.3}$, which indeed
represents the number of steps in bootstrap iterations.
\end{corollary}

\begin{proof}
Since the arguments have been shown elsewhere, we merely outline the proof.
The moment estimate $\eqref{f:5.5}$ has guaranteed that $\mathcal{C}_{*,\hat{R}_0/\varepsilon}(0)$ is finite almost surely, which made the succeeding computations pathwise. Therefore, the desired estimate
$\eqref{pri:5.3}$ follows from $\eqref{pri:5.1*}$ by a standard localization argument (see e.g. \cite[Lemma 2.19]{Xu16}), which involves the bootstrap process as shown in Lemma $\ref{lemma:ap-2}$. After established the stated
estimate $\eqref{pri:5.3}$, one can apply it to complex valued equations
with $f=0$ and $g=\mu\mathbf{u}$. Repeating the same bootstrap iteration
as in Step 2 of the proof in Lemma $\ref{lemma:ap-2}$, we merely substitute
Caccioppoli's inequality with its complex-valued counterpart (see $\eqref{pri:5.6}$)  therein, and the multiplicative constant additionally involves the parameter $\theta$ given in $\eqref{a:5}$. Finally, we mention that it is fine to use the rescaling argument to start from the case of $r_B=1$ by noting that  $\mathcal{C}_{*,\hat{R}_0/\varepsilon}$ is the  scaling invariant quantity.
\end{proof}

\begin{lemma}\label{shen's lemma}
Let $T$ be a sublinear operator in $\Omega\subset\mathbb{R}^d$. Let $0<q_0<q<q_1<\infty$ and $\omega\in A_{q/q_0}$ weight.
Assume that $N>1$ and $\alpha_2>\alpha_1>1$.
Suppose that
\begin{enumerate}
  \item $T$ is bounded on $L^{q_0}(\Omega;\mathbb{C}^m)$ with $\|T\|_{L^{q_0}\to L^{q_0}}\leq C_0$;
  \item For any ball $B\subset\mathbb{R}^d$
  with the property: either  $\alpha_{2}B\subset\Omega$ or $x_B\in\Omega$,
  and for any $\mathbf{g}\in L_c^\infty(\Omega;\mathbb{C}^m)$ with
  $\text{supp}(\mathbf{g})\subset\Omega\setminus \alpha_2B$,
  one has
  \begin{equation}\label{shen-4}
  \Big(\dashint_{B\cap\Omega}|T(\mathbf{g})|^{q_1}
  \omega\Big)^{1/q_1}
  \leq N\bigg\{\Big(\dashint_{\alpha_1B\cap\Omega}
  |T(\mathbf{g})|^{q_0}\Big)^{1/q_0}
  +\sup_{B'\supseteq B}\Big(
  \dashint_{B'}|\mathbf{g}|^{q_0}\Big)^{1/q_0}\bigg\}
  \Big(\dashint_{B}\omega\Big)^{1/q_1}.
  \end{equation}
\end{enumerate}
Then, $T$ is bounded in weighted-$L^{q}(\Omega;\mathbb{C}^m)$ for any $q_0<q<q_1$. Moreover, there exists $\kappa$, depending on $q_1$ and $q$, such that
\begin{equation}\label{pri:shen-1}
  \Big(\int_{\Omega}|T(\mathbf{g})|^q\omega\Big)^{\frac{1}{q}}
  \leq CN^{\kappa}
  \Big(\int_{\Omega}|\mathbf{g}|^q\omega\Big)^{\frac{1}{q}},
\end{equation}
where $C$ depends only on $d,m,K,q_0,q_1,q,C_0$, and $[\omega]_{q/q_0}$.
\end{lemma}

\begin{proof}
The proof can be found in \cite[Theorem 3.3]{Shen05} and \cite[Theorem 3.1]{Shen23}, and the estimate $\eqref{pri:shen-1}$
\end{proof}

\bigskip
\noindent
\textbf{Proof of Theorem $\ref{thm:resolvent}$.}
Before giving the concrete proof, we point out that the estimates $\eqref{C-1}$ and $\eqref{C-2}$ are different. Let us start from
the estimate $\eqref{C-2}$. In fact, we can derive it by merely
modifying Step 4 in the proof of Proposition $\ref{P:3}$
(where the estimate $\eqref{f:3.0}$ will be substituted for
$\eqref{pri:5.4-w}$, correspondingly) and adjusting
the corresponding statement of Theorems $\ref{thm:2}$ and $\ref{P:5}$
(since the approach to go from the suboptimal estimate to optimal one is only related to the size of $|\mu|$.).
Although we can derive the same estimate as $\eqref{C-2}$ for the case of bounded domains,
the estimate $\eqref{C-1}$ can not be simply obtained from it by using Poincar\'e inequality.
That's the main reason why we additionally introduce Lemma $\ref{shen's lemma}$ to be the basic ingredient for the proof of $\eqref{C-1}$.

Therefore, the following proof merely focus on showing the proof
of $\eqref{C-1}$. Roughly speaking, we repeatedly utilize Lemma $\ref{shen's lemma}$ to reach our purpose, whose idea is analogy to that used in Lemma $\ref{lemma:ap-2}$. But, the stochastic integrability had a loss in the annealed Caldero\'n-Zygmund estimates (see e.g. Theorems $\ref{thm:1}$ and $\ref{thm:1-w}$), and we introduce the ``total probability rule'' to deal with this difficulty.

It is fine to assume $\mathbf{f}=0$ to establish
the estimate $\eqref{C-1}$ since one may simply utilize the result established for $\mathbf{g}$ by setting
$\mathbf{g}=(|\mu|+R_0^{-2})^{1/2}\mathbf{f}$ to get the desired estimate.
Due to Theorem $\ref{thm:5.0}$, we can define the linear operators as follows:
\begin{equation*}
T_{\mu}(\mathbf{g}):=(|\mu|+R_0^{-2})\mathbf{u};
\qquad
Q_{\mu}(\mathbf{g}):=(|\mu|+R_0^{-2})^{1/2}\nabla\mathbf{u}.
\end{equation*}
By the above definition, one can infer that
\begin{equation}\label{f:5.1}
Q_{\mu}(\mathbf{g})
=(|\mu|+R_0^{-2})^{-1/2}\nabla T_{\mu}(\mathbf{g}),\quad
T_{\mu}^{*}=T_{\bar{\mu}},\quad
\text{and}\quad
Q_{\mu}^{*}=Q_{\bar{\mu}}.
\end{equation}
Let $\Psi$ be the probability space governed by $\langle\cdot\rangle$, and
we have a decomposition by means of the stationary random field $\mathcal{C}_{*,R_0}$ as follows: for any integer $k\geq 0$, there are
\begin{equation}\label{decomposition}
\Xi^{0}:=\{w\in\Psi:\mathcal{C}_{*,R_0}^{w}(0)\leq 1\};
\quad
 \Xi^{k}:=\{w\in\Psi:2^{k-1}<\mathcal{C}_{*,R_0}^{w}(0)\leq
 2^{k}\}, \quad\text{and}\quad
 \Psi=\cup_{k=0}^{\infty}\Xi^{k},
\end{equation}
where we merely take $R_0=\hat{R}_0/\varepsilon$
(since one can formally obtain the elliptic operators $\mathcal{L}$ by
rescaling the multi-scale one $-\nabla\cdot a(\cdot/\varepsilon)\nabla$
with the price of the dilation of the region). Then, the conditional
version of $T_{\mu}$ and $Q_{\mu}$ can be defined as
\begin{equation*}
T_{\mu}^{k}(\mathbf{g}):=T_{\mu}(\mathbf{g})\mathds{1}_{\Xi^k};
\qquad
Q_{\mu}^{k}(\mathbf{g}):=Q_{\mu}(\mathbf{g})\mathds{1}_{\Xi^k},
\end{equation*}
and $T_{\mu}^{k}$ with $Q_{\mu}^{k}$
satisfies the same property as in $\eqref{f:5.1}$ since we can
multiply $\mathds{1}_{\Xi^k}$ to both sides of the equation $\eqref{pde:0-r}$ directly. Therefore, one can hand over
the corresponding $L^p$-boundedness estimates to their ``sublinear counterparts'' which
are defined as follows:
\begin{equation*}
\begin{aligned}
&T_{\mu,*}^{k}(\mathbf{g}):=|T_{\mu}^{k}(\mathbf{g})|;
&\qquad&
T_{\mu,p}^{k}(\mathbf{g}):=\langle |T_{\mu}^{k}(\mathbf{g})|^{p}\rangle^{1/p},\\
&Q_{\mu,*}^{k}(\mathbf{g}):=|Q_{\mu}^{k}(\mathbf{g})|;
&\qquad&
Q_{\mu,p}^{k}(\mathbf{g}):=\langle |Q_{\mu}^{k}(\mathbf{g})|^{p}\rangle^{1/p}.
\end{aligned}
\end{equation*}

Then, based upon Lemma $\ref{shen's lemma}$, the proof of $\eqref{C-1}$ will  be divided into three steps.

\noindent
\textbf{Step 1.}
Establish $L^p$-boundedness of $(T_{\mu}^{k},
Q_{\mu}^{k})$ for each integer $k\geq 0$.
By multiplying $\mathds{1}_{\Xi^k}$ to both sides of the equation $\eqref{pde:0-r}$, and using Theorem $\ref{thm:5.0}$, we can infer that
$T_{\mu,*}^{k}(\mathbf{g})$ and $Q_{\mu,*}^{k}(\mathbf{g})$ are $L^2$-bounded uniformly with respective to $k$. Thus, in view of Lemma $\ref{shen's lemma}$ (taking $\omega=1$ and $\alpha_1=3$ therein), it suffices to verify that for any $2<p<p_1$, there holds
\begin{equation*}
\begin{aligned}
\Big(\dashint_{B\cap\Omega}\big|\big(T_{\mu,*}^{k}(\mathbf{g}),
Q_{\mu,*}^{k}(\mathbf{g})
&\big)\big|^{p_1}\Big)^{1/p_1}
  \lesssim 2^{2\beta_dk}\\
&\bigg\{\Big(\dashint_{3B\cap\Omega}
  \big|\big(T_{\mu,*}^{k}(\mathbf{g}),
Q_{\mu,*}^{k}(\mathbf{g})\big)\big|^{2}\Big)^{1/2}
  +\sup_{B'\supseteq B}\Big(
  \dashint_{B'}|\mathbf{g}\mathds{1}_{\Xi^k}|^{2}\Big)^{1/2}\bigg\},
\end{aligned}
\end{equation*}
which indeed given by the estimates $\eqref{pri:5.2}$.
As a result, we obtain that
\begin{equation*}
\Big(\int_{\Omega}
\big|\big(T_{\mu}^{k}(\mathbf{g}),
Q_{\mu}^{k}(\mathbf{g})
\big)\big|^{p}\Big)^{\frac{1}{p}}
\leq
\Big(\int_{\Omega}
\big|\big(T_{\mu,*}^{k}(\mathbf{g}),
Q_{\mu,*}^{k}(\mathbf{g})
\big)\big|^{p}\Big)^{\frac{1}{p}}
\lesssim^{\eqref{pri:shen-1}} 2^{2\beta_dk\kappa_1}
\Big(\int_{\Omega}
\big|\mathbf{g}\mathds{1}_{\Xi^k}\big|^{p}\Big)^{\frac{1}{p}},
\end{equation*}
and this implies
\begin{equation}\label{f:5.2}
\Big(\int_{\Omega}
\big|\big(T_{\bar{\mu}}^{k}(\mathbf{g}),
Q_{\bar{\mu}}^{k}(\mathbf{g})
\big)\big|^{p}\Big)^{\frac{1}{p}}
\leq^{\eqref{f:5.1}} \Big(\int_{\Omega}
\big|\big(T_{\mu}^{k}(\mathbf{g}),
Q_{\mu}^{k}(\mathbf{g})
\big)\big|^{p}\Big)^{\frac{1}{p}}
\lesssim 2^{2\beta_dk\kappa_1}
\Big(\int_{\Omega}
\big|\mathbf{g}\mathds{1}_{\Xi^k}\big|^{p}\Big)^{\frac{1}{p}}.
\end{equation}

Finally, we can apply the duality arguments to get
$\eqref{f:5.2}$ for the case of $1<p<2$ and this
established $L^p$-boundedness of $(T_{\mu}^{k},
Q_{\mu}^{k})$ with $1<p<\infty$.

\medskip
\noindent
\textbf{Step 2.} For any $1<p,q<\infty$,
we can establish $L^q(\Omega;L^{p}_{\langle\cdot\rangle})$-boundedness of $(T_{\mu}^{k},
Q_{\mu}^{k})$. Firstly, taking $p$-th power
on the both sides of $\eqref{f:5.2}$ and then taking
the expectation $\langle\cdot\rangle$, using the Fubini theorem, we can derive that
\begin{equation*}
\int_{\Omega}
\Big\langle\big|\big(T_{\mu}^{k}(\mathbf{g}),
Q_{\mu}^{k}(\mathbf{g})
\big)\big|^{p}\Big\rangle
\lesssim 2^{2\beta_dk\kappa_1 p}
\int_{\Omega}
\big\langle\big|\mathbf{g}\mathds{1}_{\Xi^k}\big|^{p}\big\rangle.
\end{equation*}
By taking $p$-th root on the both sides above, we indeed have
$L^p(\Omega;L^{p}_{\langle\cdot\rangle})$-boundedness for
$(T_{\mu}^{k},
Q_{\mu}^{k})$ and,
in view of Lemma $\ref{shen's lemma}$, this is our start point in this step. Now, for any $p<q<q_1$, it suffices to establish
\begin{equation*}
\begin{aligned}
\Big(\dashint_{B\cap\Omega}\big|\big(T_{\mu,p}^{k}(\mathbf{g}),
Q_{\mu,p}^{k}(\mathbf{g})
&\big)\big|^{q_1}\Big)^{1/q_1}
  \leq 2^{2\beta_dk}\\
&\bigg\{\Big(\dashint_{3B\cap\Omega}
  \big|\big(T_{\mu,p}^{k}(\mathbf{g}),
Q_{\mu,p}^{k}(\mathbf{g})\big)\big|^{p}\Big)^{1/p}
  +\sup_{B'\supseteq B}\Big(
  \dashint_{B'}\langle|\textbf{g}\mathds{1}_{\Xi^k}|^{p}
  \rangle\Big)^{1/p}\bigg\},
\end{aligned}
\end{equation*}
and it will be yielded by the following computations:
\begin{equation}\label{f:5.3}
\begin{aligned}
\Big(\dashint_{B\cap\Omega}\Big\langle\big|\big(\mathbf{u},
\nabla\mathbf{u}\big)\mathds{1}_{\Xi^k}\big|^{p}\Big\rangle^{\frac{q_1}{p}}
\Big)^{\frac{p}{q_1}}
&\leq \Big\langle
\Big(\dashint_{B\cap\Omega}
\big|(\mathbf{u},
\nabla\mathbf{u})\mathds{1}_{\Xi^k}\big|^{q_1}\Big)^{\frac{p}{q_1}}\Big\rangle \\
&\lesssim^{\eqref{pri:5.2},\eqref{decomposition}} 2^{2\beta_dkp}
\dashint_{3B\cap\Omega}\Big\langle\big|\big(\mathbf{u},
\nabla\mathbf{u}\big)\mathds{1}_{\Xi^k}\big|^{p}\Big\rangle.
\end{aligned}
\end{equation}
By using Lemma $\ref{shen's lemma}$ again, we derive that
\begin{equation*}
\Big(\int_{\Omega}
\big\langle|(T_{\mu}^{k}(\mathbf{g}),
Q_{\mu}^{k}(\mathbf{g})
)|^{p}\big\rangle^{\frac{q}{p}}\Big)^{\frac{1}{q}}
\leq
\Big(\int_{\Omega}
\big|\big(T_{\mu,p}^{k}(\mathbf{g}),
Q_{\mu,p}^{k}(\mathbf{g})
\big)\big|^{q}\Big)^{\frac{1}{q}}
\lesssim^{\eqref{pri:shen-1}} 2^{2\beta_dk\kappa_2}
\Big(\int_{\Omega}\big\langle
\big|\mathbf{g}\mathds{1}_{\Xi^k}\big|^{p}\big\rangle^{\frac{q}{p}}\Big)^{\frac{1}{q}}.
\end{equation*}

Repeating the argument given for $\eqref{f:5.2}$ and using duality arguments, we can complete the proof of $L^q(\Omega;L^{p}_{\langle\cdot\rangle})$-boundedness of $(T_{\mu}^{k},
Q_{\mu}^{k})$ for any $1<p,q<\infty$.

\medskip
\noindent
\textbf{Step 3.}
Establish weighted $L^q(\Omega;L^{p}_{\langle\cdot\rangle})$-boundedness of $(T_{\mu}^{k}, Q_{\mu}^{k})$. Let $0<1-q_0\ll 1$ and
$q_1\gg 1$ be arbitrarily fixed, and then we set $q_0<p,q<q_1$ and $\gamma,\gamma'\in(1,\infty)$ with $1/\gamma+1/\gamma'=1$. On account of Lemma $\ref{shen's lemma}$ and the result
established in Step 2, to get the target,  it is reduced to show that
for any $\omega\in A_q$ (which implies that $\omega\in A_{q/q_0}$ by $\eqref{f:19-1}$), there holds the estimate $\eqref{shen-4}$
for $(T_{\mu}^{k}, Q_{\mu}^{k})$, up to a factor $\alpha^{2\beta_d(k+\rho)\kappa_2}$. Then, it can be given by modifying
the computations given for $\eqref{f:5.3}$, and there are
\begin{equation*}
\begin{aligned}
\Big(\dashint_{B\cap\Omega}\Big\langle\big|\big(\mathbf{u},
\nabla\mathbf{u}\big)\mathds{1}_{\Xi^k}\big|^{p}\Big\rangle^{\frac{q_1}{p}}
&\omega\Big)^{\frac{p}{q_1}}
\lesssim_{d,M_0} \Big(\dashint_{B\cap\Omega}\Big\langle\big|\big(\mathbf{u},
\nabla\mathbf{u}\big)\mathds{1}_{\Xi^k}\big|^{p}\Big\rangle^{\frac{q_1\gamma}{p}}
\Big)^{\frac{p}{q_1\gamma}}\Big(\dashint_{B}\omega^{\gamma'}
\Big)^{\frac{p}{q_1\gamma'}} \\
&\lesssim^{\eqref{pri:R-1}} \Big\langle
\Big(\dashint_{B\cap\Omega}
\big|(\mathbf{u},
\nabla\mathbf{u})\mathds{1}_{\Xi^k}\big|^{q_1\gamma}\Big)^{\frac{p}{q_1\gamma}}
\Big\rangle
\Big(\dashint_{B}\omega
\Big)^{\frac{p}{q_1}}\\
&\lesssim^{\eqref{pri:5.2},\eqref{decomposition}} 2^{2\beta_dkp}
\Big(\dashint_{3B\cap\Omega}\big\langle\big|\big(\mathbf{u},
\nabla\mathbf{u}\big)\mathds{1}_{\Xi^k}\big|^{p}\big\rangle^{\frac{q_0}{p}}
\Big)^{\frac{p}{q_0}}\Big(\dashint_{B}\omega
\Big)^{\frac{p}{q_1}}.
\end{aligned}
\end{equation*}
Hence, one can derive that
\begin{equation}\label{f:5.4}
\Big(\int_{\Omega}
\big\langle|(T_{\mu}^{k}(\mathbf{g}),
Q_{\mu}^{k}(\mathbf{g})
)|^{p}\big\rangle^{\frac{q}{p}}\omega\Big)^{\frac{1}{q}}
\lesssim^{\eqref{pri:shen-1}} 2^{2\beta_dk\kappa_3}
\Big(\int_{\Omega}\big\langle
\big|\mathbf{g}\mathds{1}_{\Xi^k}\big|^{p}
\big\rangle^{\frac{q}{p}}\omega\Big)^{\frac{1}{q}},
\end{equation}
which further implies the conditional version of $\eqref{C-1}$, i.e.,
\begin{equation}\label{f:5.7}
\begin{aligned}
(|\mu|+R_0^{-2})&\Big(\int_{\Omega}
\big\langle|\mathbf{u}\mathds{1}_{\Xi^k}|^p
\big\rangle^{\frac{q}{p}}\omega\Big)^{\frac{1}{q}}
+ (|\mu|+R_0^{-2})^{\frac{1}{2}}
\Big(\int_{\Omega}\big\langle|\nabla\mathbf{u}\mathds{1}_{\Xi^k}|^p
\big\rangle^{\frac{q}{p}}\omega\Big)^{\frac{1}{q}}\\
&\lesssim^{\eqref{f:5.4}}
2^{2\beta_dk\kappa_3}
\bigg\{\Big(\int_{\Omega}\big\langle|\mathbf{g}\mathds{1}_{\Xi^k}|^{p}
\big\rangle^{\frac{q}{p}}\omega\Big)^{\frac{1}{q}}
+ (|\mu|+R_0^{-2})^{1/2}\Big(\int_{\Omega}\big\langle|\mathbf{f}\mathds{1}_{\Xi^k}|^{p}
\big\rangle^{\frac{q}{p}}\omega\Big)^{\frac{1}{q}}
\bigg\},
\end{aligned}
\end{equation}
where we indeed employ the estimate $\eqref{f:5.4}$ twice,
and one instance involves the setting of $\mathbf{g}=(|\mu|+R_0^{-2})^{1/2}\mathbf{f}$.
Also, for any $k\geq 1$, $N_0>1$, and $\bar{p}>p$, we have
\begin{equation}\label{f:5.8}
\begin{aligned}
 \big\langle|(\mathbf{g},\mathbf{f})\mathds{1}_{\Xi^k}
 |^{p}\big\rangle^{\frac{1}{p}}
 &\leq \big\langle|(\mathbf{g},\mathbf{f})|^{\bar{p}}\big\rangle^{\frac{1}{\bar{p}}}
 \big\langle\mathds{1}_{\{\mathcal{C}_{*,R_0}> 2^{k-1}\}}\big\rangle^{\frac{\bar{p}-p}{p\bar{p}}}\\
&\leq
\Big(\frac{\langle|\mathcal{C}_{*,R_0}(0)|^{N_0}\rangle}{2^{(k-1)N_0}}
\Big)^{\frac{\bar{p}-p}{p\bar{p}}}
\lesssim_{\lambda,\lambda_1,\lambda_2,d,M_0,N_0}^{\eqref{f:5.5}} 2^{-(k-1)N_0(\frac{\bar{p}-p}{\bar{p}p})},
\end{aligned}
\end{equation}
where we employ Chebyshev’s inequality in the second step.

Thus, in view of the decomposition $\eqref{decomposition}$,
it follows from the triangle inequality that
\begin{equation}\label{f:5.6}
\begin{aligned}
&(|\mu|+R_0^{-2})\Big(\int_{\Omega}
\big\langle|\mathbf{u}|^p
\big\rangle^{\frac{q}{p}}\omega\Big)^{\frac{1}{q}}
+ (|\mu|+R_0^{-2})^{\frac{1}{2}}
\Big(\int_{\Omega}\big\langle|\nabla\mathbf{u}|^p
\big\rangle^{\frac{q}{p}}\omega\Big)^{\frac{1}{q}}\\
&\leq
(|\mu|+R_0^{-2})\sum_{k=0}^{\infty}\Big(\int_{\Omega}
\big\langle|\mathbf{u}\mathds{1}_{\Xi^k}|^p
\big\rangle^{\frac{q}{p}}\omega\Big)^{\frac{1}{q}}
+ (|\mu|+R_0^{-2})^{\frac{1}{2}}\sum_{k=0}^{\infty}
\Big(\int_{\Omega}\big\langle|\nabla\mathbf{u}\mathds{1}_{\Xi^k}|^p
\big\rangle^{\frac{q}{p}}\omega\Big)^{\frac{1}{q}}\\
&\lesssim^{\eqref{f:5.7},\eqref{f:5.8}}\sum_{k=0}^{\infty}
2^{k[2\beta_d\kappa_3-N_0(\frac{\bar{p}-p}{\bar{p}p})]}
\bigg\{\Big(\int_{\Omega}\big\langle|\mathbf{g}|^{\bar{p}}
\big\rangle^{\frac{q}{\bar{p}}}\omega\Big)^{\frac{1}{q}}
+ (|\mu|+R_0^{-2})^{1/2}\Big(\int_{\Omega}\big\langle|\mathbf{f}|^{\bar{p}}
\big\rangle^{\frac{q}{\bar{p}}}\omega\Big)^{\frac{1}{q}}
\bigg\},
\end{aligned}
\end{equation}
and we can choose $N_0$ to make the power series convergent in
the last line of $\eqref{f:5.6}$. This gives the desired estimate
$\eqref{C-1}$ and completes the whole proof.
\qed

\section{Appendix}\label{sec:5}

\begin{lemma}[\cite{Duoandikoetxea01,Grafakos08}]\label{weight}
Let $\omega\in A_p$, $1\leq p<\infty$. Then, we have the following properties:
\begin{subequations}
\begin{align}
& [\tilde{\omega}]_{A_p} = [\omega]_{A_p},
~\text{where}~\tilde{\omega} = \omega(\varepsilon\cdot);  \label{f:18-0}\\
& \omega(Q)\big(\frac{|S|}{|Q|}\big)^{p}\leq [\omega]_{A_p}\omega(S)
\quad \forall S\subset Q\subset\mathbb{R}^d. \label{f:18}
\end{align}
\end{subequations}
Also, there exists $\epsilon>0$, depending only on $p$ and $[\omega]_{A_p}$, such that
for any $Q$ we have
\begin{subequations}
\begin{align}
&\Big(\dashint_{Q}\omega^{1+\epsilon}\Big)^{\frac{1}{1+\epsilon}}
\lesssim_{d,p,[\omega]_{A_p}}\dashint_{Q}\omega;\label{pri:R-1} \\
& \omega(S)\lesssim_{d,p,[\omega]_{A_p}}  \big(\frac{|S|}{|Q|}\big)^{\delta_0}\omega(Q),
\qquad \forall S\subset Q,
\label{f:19}
\end{align}
\end{subequations}
where $\delta_0 := \epsilon/(1+\epsilon)$. Moreover, given any $1<p<\infty$ and $c_0>1$
there exist positive constants $C=C(d,p,c_0)$ and $\gamma=\gamma(d,p,c_0)$ such that
for all $\omega\in A_p$ we have
\begin{equation}\label{f:19-1}
  [\omega]_{A_p}\leq c_0
\quad \Rightarrow \quad
  [\omega]_{A_{p-\gamma}}\leq C.
\end{equation}
Let $\mathcal{M}$ be the Hardy-Littlewood maximal function. Then,
\begin{equation}\label{f:19-2}
\int_{\mathbb{R}^d}|\mathcal{M}(f)|^p\omega
\lesssim_p \int_{\mathbb{R}^d}|f|^p\omega
\end{equation}
if and only if $\omega\in A_p$.
\end{lemma}

\begin{lemma}[Shen's real argument of weighted version \cite{Shen23}]\label{shen's lemma2}
Let $0<p_0<p<p_1$, $\omega\in A_{\frac{p}{p_0}}$weight
and $\Omega$ be a bounded Lipschitz domain,
$F\in L^{p_0}(\Omega)$ and $g\in L^{p_1}(\Omega)$.
Let $B_0:=B_{r_0}(x_0)$, where $x_0\in\partial\Omega$ and
$0<r_0<c_0\text{diam}(\Omega)$. Suppose that for each ball $B$
with the property that $|B|\leq  c_1|B_0|$
and either $4B\subset 2B_0\cap\Omega$ or $x_B\in\partial\Omega\cap 2B_0$, there exist two measurable functions
$F_B$ and $R_B$ on $\Omega\cap2B$,
such that $|F|\leq |W_B| +
|V_B|$ on $\Omega\cap2B$, and
\begin{subequations}
\begin{align}
\Big(\dashint_{2B\cap\Omega}|V_B|^{p_1}\omega
\Big)^{\frac{1}{p_1}}
&\lesssim_{N_1} \bigg\{
\Big(\dashint_{\alpha B\cap\Omega}|F|^{p_0}\Big)^{\frac{1}{p_0}}
+\sup_{4B_0\supseteq B^\prime\supseteq B}
\Big(\dashint_{B^\prime\cap\Omega}|g|^{p_0}
\Big)^{\frac{1}{p_0}}
\bigg\}\Big(\dashint_{B}\omega\Big)^{1/{p_1}},
\label{shen-1}\\
\Big(\dashint_{2B\cap\Omega}|W_B|^{p_0}
\Big)^{\frac{1}{p_0}}
&\lesssim_{N_2}
\sup_{4B_0\supseteq B^\prime\supseteq B}
\Big(\dashint_{B^\prime\cap\Omega}|g|^{p_0}
\Big)^{\frac{1}{p_0}}
+\eta\Big(\dashint_{\alpha B\cap\Omega}
|F|^{p_0}\Big)^{\frac{1}{p_0}},
\label{shen-2}
\end{align}
\end{subequations}
where $N_1, N_2>0$, $\eta\geq0$
and $0<c_1<1$ with $\alpha>2$. Then there exists $\theta_{0}>0$, depending on
$d,p_0,p_1,p,N_1,[\omega]_{A_{p/p_0}}$ and the
Lipschitz character of $\Omega$, with the property that if $0\leq \eta\leq\eta_0$, then
\begin{equation}\label{shen-3}
\Big(\dashint_{B_0\cap\Omega}|F|^p\omega\Big)^{\frac{1}{p}}
\lesssim \bigg\{\Big(\dashint_{4B_0\cap\Omega}|F|^{p_0}\Big)^{\frac{1}{p_0}}
\Big(\dashint_{B_0}\omega\Big)^{1/{p}}
+\Big(\dashint_{4B_0\cap\Omega}|g|^p\omega\Big)^{\frac{1}{p}}\bigg\},
\end{equation}
where the multiplicative constant relies on
$d,p_0,p_1,p,c_1,\alpha,N_1,N_2,[\omega]_{A_{p/p_0}}$ and the Lipschitz
character of $\Omega$.
\end{lemma}

\begin{proof}
See \cite[Theorem 4.1]{Shen23}.
\end{proof}

\begin{lemma}[primary geometry on integrals]\label{integral-geometry}
Let $f\in L^1_{\emph{loc}}(\mathbb{R}^d)$, and
$\Omega\subset\mathbb{R}^d$ be a (bounded) Lipschitz domain with
$X\in\partial\Omega$ and $R_0\gg 1$. Then
there hold the following inequalities:
\begin{itemize}
  \item For all $B_r(X)
  \subset\mathbb{R}^d$ and $x_0\in B_r(X)\cap\Omega$ with
   $r<(1/4)$, we have
  \begin{equation}\label{g-0}
  \Big(\dashint_{U_{1}(x_0)}|f|\Big)^{\frac{1}{s}}
  \lesssim
  \dashint_{D_{5r}(X)}
  \Big(
  \dashint_{U_{1}(x)}
  |f|\Big)^{\frac{1}{s}} dx,
  \end{equation}
  where $s>0$.
  \item For all balls $B_r(X)
  \subset\mathbb{R}^d$ with $r\geq (1/4)$,
  we have
\begin{subequations}
\begin{align}
  &\dashint_{D_r(X)} |f|
  \lesssim \dashint_{D_{2r}(X)}
  \Big(\dashint_{U_{1}(x)}|f|\Big)dx
  \lesssim \dashint_{D_{7r}(X)} |f|;
  \label{g-1}\\
  &
\int_{\Omega}|f|\sim\int_{\Omega}
  \Big(\dashint_{U_{1}(x)}|f|\Big)dx,
  \label{g-1*}
\end{align}
\end{subequations}
\end{itemize}
where the multiplicative constant depends only on $d$ and $M_0$.
\end{lemma}

\begin{proof}
See \cite[Lemma 6.2]{Wang-Xu-Zhao21} or \cite[Lemma 2.12]{Wang-Xu22}.
\end{proof}

\begin{lemma}\label{weight-1}
Let $R\geq 1$.
Assume that $F\in C^{\infty}_0(\Omega;L^\infty_{\langle\cdot\rangle})$\footnote{
It represents the space of smooth functions with compact support and values taken in $L^\infty_{\langle\cdot\rangle}$,
and the subscript of $L^\infty_{\langle\cdot\rangle}$ stresses that its probability measure is introduced from
the ensemble $\langle\cdot\rangle$.},
$1\leq p,q,s<\infty$.
Then, there holds
\begin{equation}\label{w-1}
\bigg(\int_{\Omega}\Big\langle\big(
\dashint_{\textcolor[rgb]{1.00,0.00,0.00}{U_{R}(x)}}|F|^s\big)^{\frac{p}{s}}
\Big\rangle^{\frac{q}{p}} dx\bigg)^{\frac{1}{q}}
\lesssim_{\textcolor[rgb]{0.00,0.00,1.00}{d,M_0}}
R^{d(1-\frac{1}{s})}
\bigg(\int_{\Omega}\Big\langle\big(
\dashint_{\textcolor[rgb]{1.00,0.00,0.00}{U_1(x)}}|F|^s\big)^{\frac{p}{s}}
\Big\rangle^{\frac{q}{p}} dx\bigg)^{\frac{1}{q}}.
\end{equation}
\end{lemma}

\begin{proof}
See \cite[pp.43-44]{Josien-Otto22} for the details.
\end{proof}

\paragraph{Acknowledgements.}
The authors are grateful to Prof. Felix Otto for sharing his very helpful insight on the part of the suboptimal Calder\'on-Zygmund estimate, which greatly simplifies our previous proof of Proposition $\ref{P:3}$.
The first author also appreciates
Prof. Zhifei Zhang for his instruction and encouragement when she held a post-doctoral position in Peking University.
The first author was supported by China Postdoctoral
Science Foundation (Grant No. 2022M710228).
The second author was supported by the Young Scientists Fund of the National Natural Science Foundation of China (Grant No. 11901262), and by the Fundamental Research Funds for the Central Universities (Grant No.lzujbky-2021-51);

%\paragraph{Data availability statement.}
% Data sharing not applicable to this article as no datasets were generated or analysed during the current study.

\addcontentsline{toc}{section}{References}

%\bibliographystyle{plain}
%\bibliography{mybib}

%\newpage
%
%
%\noindent Li Wang\\
%School of Mathematic Sciences,
%Peking University, Peking 100871, China.\\
%E-mail:wangli@math.pku.edu.cn\\
%
%\noindent Qiang Xu\\
%School of Mathematics and Statistics,
%Lanzhou University, Lanzhou 730000, China.\\
%E-mail:xuq@lzu.edu.cn\\
%
%\noindent Zhifei Zhang\\
%School of Mathematic Sciences,
%Peking University, Peking 100871, China.\\
%E-mail:zfzhang@math.pku.edu.cn\\

\end{document}